\documentclass[a4paper, 11pt, twoside, reqno]{amsart} \usepackage[a4paper,inner=2cm,outer=2cm,top=2.5cm,bottom=2.5cm]{geometry}
\usepackage{tikz}
\usepackage{amsmath,amscd}
\usepackage{amssymb}
 \def\diam{ {{\rm diam}}}

\usepackage{ulem}
\usepackage{amsthm}
\usepackage{comment}
\usepackage{mathrsfs}
\usepackage{graphicx, xcolor}
\usepackage{xcolor}
\usepackage{mathtools}
\usepackage[dvipsnames]{xcolor}
\definecolor{myred}{RGB}{220,20,60}
\definecolor{chatakred}{RGB}{255,0,0}
\usepackage[ocgcolorlinks, linkcolor=blue,citecolor=red,urlcolor=blue]{hyperref}
\usepackage{bm}
\usepackage{bbm}
\usepackage{url}
\usepackage[utf8]{inputenc}
\usepackage{mathtools,amssymb}
    \usepackage{amssymb}
\usepackage{tikz}
\usepackage{dsfont}
\usepackage{relsize}
\usepackage{url}
\usepackage{xcolor}
\usepackage{graphicx}
\usepackage{mathrsfs}
\usepackage[shortlabels]{enumitem}
\usepackage{lineno}
\usepackage{amsmath}
\usepackage{enumitem}
\usepackage{amsthm} 
\usepackage{verbatim}
\usepackage{dsfont}
\usepackage[utf8]{inputenc}
\usepackage{tikz}
\DeclareMathOperator{\supp}{supp}
\newcommand{\R}{{\mathbb R}}
\numberwithin{equation}{section}

\allowdisplaybreaks
 
 \mathtoolsset{showonlyrefs}
\theoremstyle{plain}
\newtheorem{theorem}{Theorem}[section]
\newtheorem{lemma}[theorem]{Lemma}
\newtheorem{corollary}[theorem]{Corollary}

\allowdisplaybreaks[0]
\usepackage{blindtext}
\usepackage{cleveref}
\renewcommand{\O}{\Omega}
\renewcommand{\o}{\omega}

\title[An inverse problem for a semilinear damped wave operator]{A partial data coefficient identification inverse problem for a semilinear damped wave operator}

\author[M. Kumar]{Mandeep Kumar}
\author[P. Kumar]{Parveen Kumar}
\author[M. Vashisth]{Manmohan Vashisth}
\address{{Department of Mathematics, Indian Institute of Technology Ropar, Rupnagar, Punjab-140001, INDIA.}}

\email{mandeep.sansanwal@gmail.com}
\email{parveen.24maz0013@iitrpr.ac.in}
\email{manmohanvashisth@iitrpr.ac.in}

\begin{document}

\begin{abstract}
This manuscript deals with a coefficient identification inverse problem for a semilinear damped wave operator in a bounded domain of $\mathbb{R}^{1+d}\ (d\geq 2)$. We establish the unique recovery of the damping coefficient, zeroth-order linear term, and the coefficient of the power-type nonlinearity from the partial Dirichlet-to-Neumann map. We investigate the corresponding uniqueness problem under the assumption that the coefficients are known in a neighborhood of the boundary, while the Neumann boundary data are prescribed only on an arbitrarily small open subset of the boundary. The analysis is largely based on the unique continuation principle, Fourier Analysis and the higher-order linearization technique. 

\vspace{.2cm}
		
	\end{abstract}
	\maketitle
\section{Introduction}
\subsection{Mathematical formulation and statement of main result}
In this manuscript, we study the inverse problem of identifying the coefficients of a semilinear damped wave operator with a power-type nonlinearity. Let $\Omega \subset \mathbb{R}^d$ $(d \geq 2)$ be an open, bounded, and connected domain with smooth boundary $\partial \Omega$. Furthermore, let $\Box$ depicts the standard wave operator $\partial_{t}^2 - \Delta_x$ in $\mathbb{R}^{1+d}$. For an integer $\ell \geq 2$, we then consider the following initial boundary value problem (IBVP) for a semilinear damped wave operator:
\begin{align}\label{equation; IBVP}
    \begin{cases}
      \Box u(t,x) + a(x)\partial_t u(t,x) + q(x) u(t,x) +r(x) u^{\ell}(t,x) = 0,  & (t,x) \in Q:=(0,T)\times \Omega,\\
      u(t,x)  = f(t,x), & (t,x) \in\Sigma:=(0,T)\times \partial \Omega,\\
      u(0, x )= \partial_t u(0, x ) = 0,  & x \in\Omega,
    \end{cases}
\end{align}
where $T>0$ is fixed, \(\Box\) denotes the d’Alembertian operator on \((0,T)\times\Omega\), \(a(x)\) represents the (spatially dependent) damping coefficient, \(q(x)\) is a potential term, and \(r(x)u^{\ell}(t,x)\) encodes the nonlinear contribution of order \(\ell\).  To state the well-posedness result, we first define some notations and function spaces. Since $\Omega$ has a smooth boundary, following  \cite{evans2022partial}, for any $m \in \mathbb{N}$, we denote the completion of $(C_c^{\infty}(\Omega), \lVert \cdot \rVert_{H^m(\Omega)})$ by $   H^{m}_{0}(\Omega)$ and is given by
\begin{align*}
     H^{m}_{0}(\Omega)
  := \left\{ h\in H^{m}(\Omega):\;
            \partial^{\alpha} h=0 \ \text{on}\ \partial\Omega
            \ \text{ for } |\alpha|\le m-1 \right\}. 
\end{align*}
We now proceed to introduce a class of Sobolev spaces that explicitly depend on time. For integers $s_1, s_2 \geq 0$ and for $\O^\sharp$ equal either to $\O$ or to $\partial\O$, we define the anisotropic Sobolev space
\[
H^{s_1, s_2}((0,T)\times \O^\sharp) := H^{s_1}\bigl(0,T; L^2(\O^\sharp)\bigr)\cap L^2\bigl(0,T; H^{s_2}(\O^\sharp)\bigr),
\]
endowed with the norm
\[
\lVert u \rVert_{H^{s_1, s_2}((0,T)\times \O^\sharp)} 
:= \lVert u \rVert_{H^{s_1}(0,T; L^2(\O^\sharp))} 
 + \lVert u \rVert_{L^2(0,T; H^{s_2}(\O^\sharp))}.
\]
 From (Theorem 2.3 in \cite{lions1972nonhomogeneous1}), one has the continuous embedding
\[
H^{m}((0,T)\times\O)\hookrightarrow H^{m}\!\bigl(0,T;L^{2}(\O)\bigr)\cap L^{2}\!\bigl(0,T;H^{m}(\O)\bigr).
\]
Consequently, the spaces \(H^{m,m}(Q)\) and \(H^{m,m}(\Sigma)\) can be identified with \(H^{m}(Q)\) and \(H^{m}(\Sigma)\), respectively, and the norms induced by these identifications are equivalent. Also, for $p,s\geq 0$, we define
\begin{align*}
     H^{p}_{0}(Q)
  &:= \left\{ h\in H^{p}(Q)\;:\;
            \partial_t^k h(0,\cdot) \in H_{0}^{p-k}(\Omega)
            \ \text{for } k=0,\dots,p-1 \right\},\\
            H^{p}_{0}\left(0,T;H^{s}(\partial\Omega)\right)
  &:= \{ h\in H^{p}(0,T;H^{s}(\partial\Omega)): \partial_t^{k} h(0,\cdot)=0 \text{ in } H^{s}(\partial\Omega)
      \ \text{for } k=0,\dots,p-1\}.
\end{align*}
Next, for integers $p \geq s \geq 0$, we consider the energy space $\mathscr{E}_p$ given by
\begin{align*}
    \mathscr{E}_p 
    := \bigcap_{s=0}^p C^s\bigl([0,T];H^{p-s}(\Omega)\bigr),
\end{align*}
endowed with the norm
\begin{align}\label{norm-Em}
    \|u\|_{\mathscr{E}_p} 
    := \left(\sup_{t\in[0,T]} 
       \sum_{s=0}^{p} 
       \bigl\| \partial_t^{\,s} u(t) \bigr\|_{H^{p-s}(\Omega)}^{2}\right)^{1/2},
\end{align}
where, for each $0 \le s \le p$,
\begin{align*}
     C^{s}\bigl([0,T];H^{p-s}(\Omega)\bigr)
     := \Bigl\{\,u : [0,T] \to H^{p-s}(\Omega)\;:\;
           \partial_t^j u \in C\bigl([0,T];H^{p-s}(\Omega)\bigr)
           \ \text{for all } j = 0,\dots,s \Bigr\}.
\end{align*}
In particular, $\mathscr{E}_p$ consists of all functions whose time derivatives up to order $p$ exist and are continuous with values in the corresponding spatial Sobolev spaces, with the energy norm \eqref{norm-Em} measuring the supremum in time of the sum of the squared Sobolev norms of these derivatives.
For $m>d+1$, we consider the Dirichlet data $f\in \mathcal{F}^\varrho_{m+1}$, where 
\begin{align}\label{E-delta}
\begin{aligned}
    \mathcal{F}^\varrho_{m+1} := \bigg \{ f\in H^{m+1}(\Sigma)\;:\; f\in H_0^{m+1-k}(0,T; H^k(\partial\Omega)) \ \text{ such that } \|f\|_{H^{m+1}(\Sigma)}< \varrho,\\ \mbox{ for } \, k= 0,1,\dots, m  \bigg\}
    \end{aligned}
\end{align}
for $\varrho>0$ sufficient small.
Also, for some $R>0$, we define 
\begin{align*}
    \mathscr{B}_R(0) := \left\{ u \in \mathscr{E}_{m+1} \;:\; \|u\|_{\mathscr{E}_{m+1}} <R\right\}.
\end{align*}
Then, for the forward problem \eqref{equation; IBVP}, we have the following well-posedness result:
\begin{theorem}{\cite[Theorem~2.3]{bhardwaj2026reconstructionpotentialdampingcoefficients}}\label{existence: forward problem}
    Let $m>d+1$, and $T>0$. Assume that $a, q,r\in C_c^{\infty}(\Omega)$  and  $f \in \mathcal{F}^\varrho_{m+1}$, with $\partial_t^k f(0) \in H_{0}^{m-k}(\Omega)$ for $k=0,1,\dots,m-1$, then the IBVP \eqref{equation; IBVP} has a unique solution $u\in \mathscr{B}_R(0) $ satisfying 
\begin{align}
  \|u\|_{\mathscr{E}_{m+1}} + \|\partial_\nu u\|_{H^{m}(\Sigma)}
\le C\,e^{CT} \|f\|_{H^{m+1}(\Sigma)}, 
\end{align}
where $C>0$, depending only on $\Omega, a,q$ and $r$. Here, $\nu$ stands for the unit normal pointing outward to $\partial \Omega$.
\end{theorem} 
Using Theorem \ref{existence: forward problem},   the  Dirichlet-to-Neumann (DtN) map $\mathcal{N}_{a,q,r}:\mathcal{F}^\varrho_{m+1} \rightarrow H^{m}(\Sigma)$, given  by 
\begin{align}\label{eq:DN map}
\begin{aligned}
\mathcal{N}_{a,q,r}(f):=\partial_{\nu} u_f \big|_{\Sigma},
\end{aligned} 
\end{align}
is well defined whenever $u_f$ is solution to \eqref{equation; IBVP} with    Dirichlet data $f\in \mathcal{F}^\varrho_{m+1}$.
The main  goal of this manuscript is to 
uniquely identify the coefficients $a$, $q$, and $r$ appearing in IBVP \eqref{equation; IBVP}
from the boundary measurements of solutions provided on a small subset of the boundary. To do so, we start by defining the partial DtN map.
 Let $\Gamma \subseteq \partial \Omega$ be an arbitrary nonempty open set  and denote $\Gamma_T := (0,T)\times \Gamma$. We denote the partial DtN map associated with \eqref{eq:main} by $\Lambda_{a,q,r}$ and it is defined by
\begin{equation}
\label{eq:revise_def_DN_map}
\Lambda_{a,q,r} (f):= \partial_\nu u_f |_{\Gamma_T} 
\end{equation}
where  $u_f$ is  the unique solution to the IBVP \eqref{equation; IBVP} when the Dirichlet data  $f\in \mathcal{F}^\varrho_{m+1} $. 

Building on previous discussions, we now state the main result of this manuscript  in the following Theorem. 
\begin{theorem}
\label{thm:main_result_partial_data}{(Statement of main result)}
Let $T> \diam(\O)$ and $\Omega \subset \mathbb{R}^d$, $m>d+1$ and $d\ge 2$, be a bounded domain with smooth boundary $\partial \O$.
Let $\mathcal{O}$ be a nonempty  open neighborhood of $\partial\Omega$ in $\overline{\Omega}$ (see Figure~\ref{fig:smooth-domain}).
Suppose that for $i=1,2$, $a_{i},q_{i},r_{i} \in C^{\infty}_c(\O)$  such that $(a_{1},q_{1},r_{1}) = (a_{2},q_{2},r_{2})$ in $ \mathcal{O}$. Then  $\Lambda_{a_{1},q_{1},r_{1}}(f)=\Lambda_{a_{2},q_{2},r_{2}}(f)$ for all $f \in \mathcal{F}^\varrho_{m+1}$ implies that $a_{1}=a_{2}$, $q_{1}= q_{2}$ and  $r_{1}=r_{2}$ in  $\Omega$.
\end{theorem}
\begin{figure}[ht]
\centering
\begin{tikzpicture}[scale=0.8]
\fill[blue!10, even odd rule]
    plot [smooth cycle, tension=0.8]
    coordinates {
        (-3,0)
        (-2.2,1.3)
        (-0.5,1.8)
        (1.2,1.5)
        (2.7,0.7)
        (2.4,-0.8)
        (1.0,-1.5)
        (-0.8,-1.4)
        (-2.3,-1.0)
    }
    --
    plot [smooth cycle, tension=0.8]
    coordinates {
        (-1.8,0)
        (-1.2,0.8)
        (-0.1,1.0)
        (1.0,0.7)
        (1.5,0)
        (1.1,-0.7)
        (0,-0.9)
        (-1.2,-0.7)
    };
\draw[thick]
    plot [smooth cycle, tension=0.8]
    coordinates {
        (-3,0)
        (-2.2,1.3)
        (-0.5,1.8)
        (1.2,1.5)
        (2.7,0.7)
        (2.4,-0.8)
        (1.0,-1.5)
        (-0.8,-1.4)
        (-2.3,-1.0)
    };
\draw[thick,dotted]
    plot [smooth cycle, tension=0.8]
    coordinates {
        (-1.8,0)
        (-1.2,0.8)
        (-0.1,1.0)
        (1.0,0.7)
        (1.5,0)
        (1.1,-0.7)
        (0,-0.9)
        (-1.2,-0.7)
    };
\node at (0,0) {$\Omega_0$};
\node at (3.4,1.5) {$\partial\Omega$};

\draw[-, thick]
    (3.15,1.35) -- (2.45,0.97);
\node at (-0.8,1.3) {$\mathcal{O}=\overline{\O}\setminus\overline\O_0$};
\node at (0.0,-2)
{$\O=(\overline\O_0\cup\mathcal{O})\setminus\partial\O$};
\end{tikzpicture}
\caption{}
\label{fig:smooth-domain}
\end{figure}
\subsection{Physical significance and motivation}
The wave operator arises naturally and fundamentally in the formulation of mathematical models describing wave propagation across a wide spectrum of physical disciplines, including acoustics, cosmology, elasticity theory, seismology, and quantum mechanics. In recent years, nonlinear inverse problems associated with several physically motivated models, such as the Kuznetsov equation,
Westervelt equation, nonlinear Klein–Gordon (KG) operator, nonlinear elastic wave equations, and gravitational wave propagation, have been extensively investigated; see, for example,~\cite{Dorich2024, DEHOOP2019347, Megias2022,
AcostaUhlmannZhai2022,
ChoquetBruhat1952,
FIKO,
liu2025partialdatainverseproblem}.

In this manuscript, we investigate the semilinear damped wave operator with potential associated with the evolution equation
\begin{align}\label{eq:main}
    \Box u(t,x) + a(x)\,\partial_t u(t,x) + q(x)\,u(t,x) + r(x)\,u^{\ell}(t,x) = 0,\quad \text{in } Q,
\end{align}
where \(a(x)\) denotes the damping coefficient, \(q(x)\) is a zeroth-order potential term, and \(r(x)\) is the coefficient of the nonlinear term of order \(\ell\). In particular, in the absence of damping (i.e., when \(a \equiv 0\)) and for the choice \(q \equiv m^{2}\), the operator in \eqref{eq:main} reduces to the nonlinear Klein–Gordon (KG) operator, which plays a central role in cosmology, hadronic physics, and high-energy physics (see \cite{Megias2022} and the references therein). Consequently, the operator in Equation~\eqref{eq:main} may be interpreted as a natural generalization of the KG equation, which motivates the present investigation of this problem.
\subsection{Literature Survey}This manuscript examines the inverse problem of identifying the coefficients in a general semilinear damped wave operator from the DtN map given on a suitable portion of the boundary. Over the past few decades, such inverse problems have attracted considerable attention due to their wide range of applications. The problem investigated in the present article falls within the scope of the Calder\'on problem, named after Calder\'on \cite{Cal80}. A major advance was achieved by Sylvester and Uhlmann \cite{SylvesterUhlmann1987}, who introduced the method of constructing complex geometric optics solutions for the Calder\'on-type problem and proved unique recovery of the conductivity coefficient in the associated conductivity equation.  In the context of inverse scattering problems, an analogous approach has been examined in \cite{Novikov1988}.
Their approach has since provided a foundational framework for establishing uniqueness and stability results related to identification of coefficients from boundary measurements for a broad class of  PDEs through such specialized solutions.

Next, we mention some results related to the linear wave equations. The unique determination of a time-dependent potential in a linear hyperbolic PDE from the boundary measurements, under appropriate convexity assumptions, was established by Bukhgeim and Klibanov in~\cite{Bukhgeim1981}. In \cite{Rakesh01011988}, Rakesh and Symes achieved unique recovery of a time-independent potential in a hyperbolic PDE from the Neumann-to-Dirichlet (NtD) map by constructing so-called \textit{beam solutions}. Subsequently, for $r=0$ in Equation~\eqref{eq:main}, Isakov~\cite{Isakov1991AnIH} refined the beam solution method and simultaneously recovered the coefficients $a$ and $q$, thereby extending the result of Rakesh and Symes. Eskin~\cite{MR2235639, MR2441006} further modified the boundary control method and established the uniqueness of the time independent coefficients for a self-adjoint linear hyperbolic operator from the DtN map in a general geometric setting.
Using the Cauchy data set, Kian and Oksanen~\cite{KianOksanen2019} uniquely recover the time dependent potential term in a Riemannian manifold setting. 
We also refer to additional related works on linear hyperbolic equations concerned with the unique recovery of the coefficients; see, for instance,~\cite{
Krishnan17082020,
Kian2016RecoveryOT,
Hu2017DeterminationOS,
Kian2017, doi:10.1137/23M1588676, LiuSaksalaYan2025, Salazar2013}. For stability results pertaining to the recovery of coefficients in hyperbolic PDEs, we refer the reader to
~\cite{Aicha2015,BellassouedDosSantosFerreira2011, BellassouedStabilityWaveJMAA,Kian2016Stability, Sun1990, CipolattiLopez2005,   Senapati_2021, KumarZimmermann2026} and references therein. Also, we refer to \cite{ER97, AvdoninBelishevIvanov1992, KhanferBukhgeim2019,KumarSarkarVashisth2024,MR4343270}

and the references therein for results concerning coupled hyperbolic systems.

We now turn to some contributions concerning nonlinear hyperbolic PDEs. Isakov~\cite{Isakov1993OnUI} was the first to introduce the method of \textit{higher-order linearization} in the context of inverse problems for nonlinear parabolic PDEs, thereby opening a new approach to treating inverse problems for nonlinear evolution equations.
Nakamura et al.~\cite{NakamuraWatanabeKaltenbacher2009} investigated a one-dimensional quasilinear wave equation in divergence form and identified a time-independent coefficient from the associated input–output map by employing the method of higher-order linearization. In the setting of higher spatial dimensions, this result was further extended in \cite{Nakamura2021}.
For semilinear wave equations (without a damping term), Lin et al.~\cite{lin2024determining} analyzed the simultaneous determination of coefficients and initial data from an input–output map. More recently, Qiu et al.~\cite{Qiu2025UniquenessRF} established a uniqueness result for a semilinear wave operator with source terms, based on the Dirichlet-to-Neumann (DtN) map. In particular, they considered power-type nonlinearities and proved the uniqueness of the corresponding coefficients together with the source term. 
In a $(3+1)$-dimensional Lorentzian manifold, a significant result was obtained by Kurylev et al.~\cite{Kurylev2018InversePF}, where they settled the problem of unique determination of global topology, the differentiable structure, and the conformal class of the metric from the knowledge of the source-to-solution map. This line of research was subsequently advanced by Hintz,  Uhlmann and Zhai~\cite{HintzUhlmannZhai2022b}, who proved that the DtN map uniquely determines both the underlying Lorentzian metric and the nonlinear coefficient. Furthermore, we refer to \cite{FeizmohammadiLassasOksanen2021, ChenLassasOksanenPaternain2022, FeizmohammadiOksanen2022, KaltenbacherRundell2022,
KaltenbacherLorenzi2007,
AcostaUhlmannZhai2022, UhlmannZhai2021} for some further results on single nonlinear hyperbolic equations. For additional studies concerning inverse problems related to nonlinear PDEs which are closely related to the problem studied  in the present manuscript, the reader is referred to \cite{CARSTEA2019121, Kurylev2022, Lassas2021, SaBarretoStefanov2024} and the references cited therein.

 Recently, Bhardwaj et al. \cite{bhardwaj2026reconstructionpotentialdampingcoefficients} studied the well-posedness of the IBVP \eqref{equation; IBVP} and reconstructed the unknown coefficients from measurements of the DtN map on the entire boundary. Furthermore, to reconstruct the nonlinear coefficient, they constructed an appropriate asymptotic solution for the underlying linearized operator (see Section~3 of \cite{bhardwaj2026reconstructionpotentialdampingcoefficients}), which  plays a key role in our analysis and is employed in the present work to establish the unique determination of the coefficient $r$.

Ru-Yu Lai et al. \cite{RuYuLaiEtAlPartialDataStabilityProblem} to settle the partial data inverse problem for power-type nonlinearities in the Schr\"odinger equation using the idea of the unique continuation principle established in \cite{Bellassoued_Fraj}. Also, for the linear wave operator with a zeroth-order perturbation, an analog UCP was derived in \cite{BellassouedStabilityWaveJMAA}. Inspired by these studies, we consider
a partial data inverse problem for a semilinear damped wave operator. 
\subsection{Outline}
The remainder of this manuscript is organized as follows: Section~\ref{Unique continuation principle} is devoted to establishing the unique continuation principle. Next, in 
Section~\ref{Section: proof of main result}, we provide the proof of  Theorem \ref{thm:main_result_partial_data}. To achieve this, first we establish the unique recovery of the damping term and zeroth-order potential using the first-order linearization method, while the unique recovery of the nonlinear coefficient is achieved using the higher-order linearization method.
\section{Unique continuation principle}\label{Unique continuation principle}
This section is devoted to establishing the unique continuation principle (UCP). 
For the wave operator with a zeroth-order perturbation, the UCP was established in \cite{BellassouedStabilityWaveJMAA}. More recently, the UCP for the linear magnetic Schrödinger equation under arbitrary boundary observations was obtained in \cite{Bellassoued_Fraj}. This latter result was then utilized by Ru-Yu Lai et al. \cite{RuYuLaiEtAlPartialDataStabilityProblem} to resolve the partial data inverse problem for power-type nonlinearities in the dynamical Schrödinger equation. The objective of the present work is to extend the result of \cite{BellassouedStabilityWaveJMAA} to the setting in which the wave operator incorporates both a damping term and a zeroth-order perturbation, and subsequently to adapt the approach developed in \cite{RuYuLaiEtAlPartialDataStabilityProblem} in order to derive the principal result of this article.

First, we introduce some notations, that will be used in the forthcoming analysis. Let $\mathcal{O}\subseteq \overline{\Omega}$ be a nonempty  relatively open neighborhood of $\partial\Omega$ in $\overline{\Omega}$. We then define an admissible class of coefficients $a$ and $q$, denoted by $\mathcal{A}$ and  given by

\begin{equation}
\label{admissible class}
\begin{aligned}
\mathcal{A}
:= &
\left\{(a,q)\in C^{\infty}(\overline {\Omega})\times C^{\infty}(\overline {\Omega}):
a=q=0
\text{ in } \mathcal{O}
\right\}.
\end{aligned}
\end{equation}

Next, using the relatively open set $\mathcal{O}$, we introduce a family of relatively compact open subsets of $\mathcal{O}$ in $\overline{\Omega}$. Specifically, we consider  open sets $\mathcal{O}_{j} \subseteq \mathcal{O}$, $j=1,2,3$, such that $\overline{\mathcal{O}}_{j+1} \subset \mathcal{O}_{j}$ for $j=1,2$ and  $\partial\Omega\subset\partial\mathcal{O}_{j}$ for $j=1,2,3$. We refer to Figure \ref{fig:nested-domains} and explanation given below it, for more details about the shape of these sets.  
We also assume that
$\overline{\mathcal O}_1\cap(\partial\mathcal O\setminus\partial\Omega)
=\emptyset$, which is used in the construction of the cutoff functions in the subsequent analysis. 
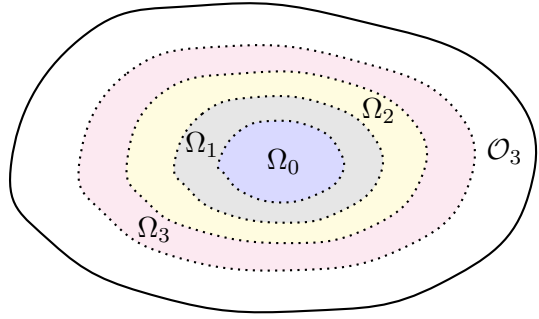
\begin{figure}[ht]
\centering

\begin{tikzpicture}[scale=0.9]
\draw[thick]
    plot [smooth cycle, tension=0.8]
    coordinates {
        (-4.0,0)
        (-3.0,1.8)
        (-1.2,2.3)
        (1.0,2.1)
        (3.2,1.2)
        (3.6,-0.5)
        (2.3,-1.8)
        (0.3,-2.2)
        (-1.8,-2.0)
        (-3.4,-1.2)
    };

\fill[magenta!10]
    plot [smooth cycle, tension=0.8]
    coordinates {
        (-3.0,0)
        (-2.3,1.3)
        (-0.8,1.7)
        (1.0,1.5)
        (2.5,0.8)
        (2.7,-0.4)
        (1.7,-1.3)
        (0.1,-1.6)
        (-1.6,-1.4)
        (-2.6,-0.8)
    };

\draw[thick,dotted]
    plot [smooth cycle, tension=0.8]
    coordinates {
        (-3.0,0)
        (-2.3,1.3)
        (-0.8,1.7)
        (1.0,1.5)
        (2.5,0.8)
        (2.7,-0.4)
        (1.7,-1.3)
        (0.1,-1.6)
        (-1.6,-1.4)
        (-2.6,-0.8)
    };

\fill[yellow!15]
    plot [smooth cycle, tension=0.8]
    coordinates {
        (-2.3,0)
        (-1.8,1.0)
        (-0.6,1.3)
        (0.8,1.2)
        (1.9,0.6)
        (2.0,-0.3)
        (1.2,-1.0)
        (0.0,-1.2)
        (-1.3,-1.0)
        (-2.0,-0.6)
    };

\draw[thick,dotted]
    plot [smooth cycle, tension=0.8]
    coordinates {
        (-2.3,0)
        (-1.8,1.0)
        (-0.6,1.3)
        (0.8,1.2)
        (1.9,0.6)
        (2.0,-0.3)
        (1.2,-1.0)
        (0.0,-1.2)
        (-1.3,-1.0)
        (-2.0,-0.6)
    };

\fill[gray!20]
    plot [smooth cycle, tension=0.8]
    coordinates {
        (-1.6,0)
        (-1.2,0.7)
        (-0.4,0.95)
        (0.6,0.85)
        (1.3,0.4)
        (1.4,-0.25)
        (0.8,-0.75)
        (0.0,-0.9)
        (-0.9,-0.75)
        (-1.4,-0.4)
    };

\draw[thick,dotted]
    plot [smooth cycle, tension=0.8]
    coordinates {
        (-1.6,0)
        (-1.2,0.7)
        (-0.4,0.95)
        (0.6,0.85)
        (1.3,0.4)
        (1.4,-0.25)
        (0.8,-0.75)
        (0.0,-0.9)
        (-0.9,-0.75)
        (-1.4,-0.4)
    };

\fill[blue!15]
    plot [smooth cycle, tension=0.8]
    coordinates {
        (-0.95,0)
        (-0.65,0.45)
        (-0.15,0.60)
        (0.45,0.50)
        (0.80,0.20)
        (0.85,-0.20)
        (0.45,-0.50)
        (-0.10,-0.60)
        (-0.60,-0.45)
        (-0.85,-0.20)
    };

\draw[thick,dotted]
    plot [smooth cycle, tension=0.8]
    coordinates {
        (-0.95,0)
        (-0.65,0.45)
        (-0.15,0.60)
        (0.45,0.50)
        (0.80,0.20)
        (0.85,-0.20)
        (0.45,-0.50)
        (-0.10,-0.60)
        (-0.60,-0.45)
        (-0.85,-0.20)
    };

\node at (3.25,0.15) {$\mathcal{O}_3$};

\node at (-1.9,-1.0) {$\O_3$};

\node at (1.40,0.75) {$\O_2$};

\node at (-1.20,0.25) {$\O_1$};
\node at (0,0) {$\O_0$};
\node[
    align=left,
    anchor=north,
    font=\small
] at (0,-2.3)
{
$\displaystyle
\mathcal{O}:=\overline \Omega\setminus\overline{\O_0},
\quad
\mathcal{O}_1:=\overline \Omega\setminus(\overline{\Omega_0}\cup \O_1),\quad 
\mathcal{O}_2:=\overline\Omega\setminus(\overline{\Omega_2} \cup \O_1\cup\overline{\O_0}),\quad
\mathcal{O}_3:=\overline \O\setminus( \overline{\Omega_3}\cup\Omega_2\cup \overline{\O_1}\cup \O_0).
$
};

\end{tikzpicture}

\caption{Nested domains and the corresponding subregions.}
\label{fig:nested-domains}

\end{figure}

We are now in a position to state the main result of this subsection, which is given by the following Theorem.
\begin{theorem}\label{Carleman type estimate required for ucp}
    For $(a,q)\in \mathcal{A}$, $F \in L^2(Q)$, let $u\in H^2(Q)$ solve the linear IBVP
    \begin{align}\label{eq:wave-UCP}
        \begin{cases}
            \Box u(t,x) + a(x)\,\partial_t u(t,x) + q(x)\,u(t,x) = F(t,x)  & (t,x) \in Q,\\[0.3em]
            u(t,x) = 0 & (t,x) \in \Sigma,\\[0.3em]
            u(0,x) = 0\quad \partial_t u(0,x) = 0 &\quad x \in \Omega.
        \end{cases}
    \end{align}
   Then for any $T_{\flat}\in(0,T)$, there exist positive constants $C = C(Q,\mathcal{O}) > 0$, $\beta> 0$ and $\eta_0 > 1$ such that, for every $\eta \geq \eta_0$, the following estimate  
   
     \begin{align}\label{carleman for ucp}
        \bigl\lVert u\bigr\rVert_{H^1\bigl((0,T_{\flat})\times (\mathcal{O}_2\setminus \mathcal{O}_3)\bigr)}
        \leq \frac{C}{\sqrt{\eta}} \bigl\lVert u\bigr\rVert_{H^2(Q)}
        +C e^{\beta \eta} \left( \bigl\lVert F\bigr\rVert_{L^2\bigl((0,T)\times \mathcal{O}\bigr)}
        + \bigl\lVert \partial_\nu u\bigr\rVert_{L^2(\Gamma_T)} \right)
    \end{align}
    holds.
\end{theorem}
Following Theorem \ref{Carleman type estimate required for ucp}, we obtain the following unique continuation result.
\begin{corollary}\label{UCP} 
In addition to the hypotheses of Theorem \ref{Carleman type estimate required for ucp}, assume that \(\operatorname{supp}(F) \subset (0,T) \times (\Omega\setminus\mathcal{O})\), and that the Neumann boundary condition \(\partial_{\nu}u = 0\) holds on \(\Gamma_T\). Under these assumptions, it follows that  $u = 0$, in $(0,T) \times (\mathcal{O}_2\setminus \mathcal{O}_3)$.
\end{corollary}
\begin{proof}
    From the support condition, we obtain $\bigl\lVert F\bigr\rVert_{L^2\bigl((0,T)\times \mathcal{O}\bigr)}=0$, and the Neumann boundary condition gives $\|\partial_\nu u\|_{L^2(\Gamma_T)}=0$. Hence,  the estimate \eqref{carleman for ucp} gives
    \begin{equation*}
         \bigl\lVert u\bigr\rVert_{H^1\bigl((0,T_{\flat})\times (\mathcal{O}_2\setminus \mathcal{O}_3)\bigr)}
        \leq \frac{C}{\sqrt{\eta}} \bigl\lVert u\bigr\rVert_{H^2(Q)}\qquad \text{ for } \eta\geq\eta_0.
    \end{equation*}
    Letting $\eta\to\infty$ yields $\bigl\lVert u\bigr\rVert_{H^1\bigl((0,T_{\flat})\times (\mathcal{O}_2\setminus \mathcal{O}_3)\bigr)}=0$, since $T_{\flat}<T$ is arbitrary, which proves $u = 0$ in $(0,T) \times (\mathcal{O}_2\setminus \mathcal{O}_3)$. 
\end{proof}
Next, we will prove Theorem \ref{Carleman type estimate required for ucp} and follow the methods developed in \cite{BellassouedStabilityWaveJMAA}. Firstly, we will extend the IBVP \eqref{eq:wave-UCP} in the interval $(-T,T)\times\O$, and to obtain this, we will define $\widetilde{u}$ and $\widetilde{F}$ as follows
\begin{equation*}
\widetilde{u}(t,x)=
    \begin{cases}
u(t,x) \qquad &(t,x)\in(0,T)\times \Omega,\\
0\qquad &(t,x)\in(-T,0]\times \Omega.
    \end{cases} 
    \qquad
\widetilde F(t,x)
=
\begin{cases}
 F(t,x)&\text{ for }(t,x)\in(0,T)\times \Omega,\\
 0&\text{ for }(t,x)\in(-T,0]\times \Omega.
\end{cases}
\end{equation*}
Next in the following Lemma, we show that $\widetilde{u}\in H^{2}\left( (-T,T)\times \Omega\right)$ and  $\widetilde{u}$ solves the following IBVP \begin{align}\label{eq:wave-UCP-extended}
        \begin{cases}
            \Box\widetilde{u}(t,x) + a(x)\,\partial_t\widetilde{u}(t,x) + q(x)\,\widetilde{u}(t,x) = \widetilde{F}(t,x)  &(t,x)\in (-T,T)\times  \Omega,\\[0.3em]
           \widetilde{u}(t,x) = 0 &(t,x)\in (-T,T)\times \partial\Omega,\\[0.3em]
           \widetilde{u}(0,x) = 0\quad \partial_t\widetilde{u}(0,x) = 0 &\quad x \in \Omega.
        \end{cases}
    \end{align}
    \begin{lemma}
        Let  $F\in L^{2}(Q)$, $u\in H^{2}(Q)$ solve \eqref{eq:wave-UCP} for $(a,q)\in\mathcal{A}$. Then
$\widetilde u\in H^2((-T,T)\times\Omega)$, $\widetilde F\in L^2((-T,T)\times\Omega)$, and
$\widetilde u$ solves \eqref{eq:wave-UCP-extended}. 
    \end{lemma}
    \begin{proof}
        For any $\varphi\in C_c^\infty((-T,T)\times\Omega)$ and using integration by parts in time variable, we get
        \begin{align*}
            \int_{-T}^{T}\int_{\Omega}\widetilde u\,\partial_t\varphi\,dx\,dt& =\int_{0}^{T}\int_{\Omega}u\,\partial_t\varphi\,dx\,dt
=-\int_{0}^{T}\int_{\Omega}\partial_tu\,\varphi\,dx\,dt-\int_{\Omega}u(0,x)\varphi(0,x)\,dx\\
& =-\int_{-T}^{T}\int_{\Omega}\widetilde{\partial_tu}\,\varphi\,dx\,dt,
        \end{align*}
        where we have used that $u(0,\cdot)=0,$ in $\O$. Hence $\partial_t\widetilde u=\widetilde{\partial_tu}\in L^2((-T,T)\times\Omega)$, where $\widetilde{\partial_tu}$ denotes the extension of $\partial_{t}u$ by zero in the interval $(-T,0]\times\Omega$.  Repeating the computation above with $\partial_tu$ in place of $u$ gives the boundary term $-\int_\Omega\partial_tu(0,x)\varphi(0,x)\,dx$ which is again zero because of $\partial_{t}u(0,\cdot)=0$ in $\O$, hence $\partial_t^2\widetilde u=\widetilde{\partial_t^2u}\in L^2((-T,T)\times\Omega)$. For a
spatial derivative there is no boundary term, so $\partial^2_{ij}\widetilde
u=\widetilde{\partial^2_{ij}u}$ for $i,j=1,2,\dots,d$. Similarly, we can show for the mixed derivative that $\partial_t\partial_{j}\widetilde u=\widetilde{\partial_t\partial_{j}u}$ for $j=1,2,\dots,d$. Therefore $\widetilde u\in H^2((-T,T)\times\Omega)$.
Since the weak derivative of $\widetilde{u}$ upto order two are zero extension of $u$ and the coefficients $a$ and $q$ are function of only spatial variable proves that $\Box\widetilde{u} + a\,\partial_t\widetilde{u} + q\,\widetilde{u}$ is zero extension of $ \Box u + a\,\partial_t u + q\,u$ which proves \eqref{eq:wave-UCP-extended} and completes the proof of the lemma.
    \end{proof}
    For our further calculations, we will continue to denote the extensions $\widetilde{u}$ and $\widetilde{F}$ by $u$ and $F$ respectively. Now let $0<T_{0}<T$, and choose a time-cutoff function $\psi\in C_{c}^{\infty}(-T,T)$ such that $0\leq\psi\leq1$ and $\psi=1$ on $[-T_{0},T_{0}]$. For $\eta\geq 1$, we define the Fourier-Bros-Iagolnitzer (FBI) transform  of $\psi u$ as follows
\begin{equation}\label{definition of V eta}
V^{\eta}(z,x):= \sqrt{\frac{\eta}{2\pi}} \int_{\mathbb{R}} \mathfrak{K}^{\eta}(z-\gamma) \psi(\gamma)u(\gamma,x)\,d\gamma,
\end{equation}
where $\mathfrak{K}^{\eta}(z)=\exp\left(\frac{-\eta z^2}{2}\right) $ for $z\in\mathbb{C}$. The mapping $ \mathbb{C}\ni z\longmapsto
\mathfrak{K}^{\eta}(z)\in \mathbb{C}$ is entire, and hence for every fixed $\gamma\in\mathbb{R}$, the translated mapping $z \longmapsto \mathfrak{K}^{\eta}(z-\gamma)$ is also entire. Moreover, since $\psi\in C_{c}^{\infty}(-T,T)$, the function $\gamma\longmapsto \psi(\gamma)u(\gamma,x)$ has compact support in $(-T,T)$. Therefore, for every fixed $x\in\Omega$, the function $V^{\eta}(\cdot,x) $ defined in \eqref{definition of V eta} is entire.

Now write $z=t+\mathrm{i} s$ then, by the chain rule, we get
\begin{align}\label{eq:kernel-derivatives}
   \partial_{s}\mathfrak{K}^{\eta}(z-\gamma)&= -\mathrm{i}\eta(z-\gamma)\mathfrak{K}^{\eta}(z-\gamma),\\
   \partial_\gamma\mathfrak K^\eta(z-\gamma)&=\eta(z-\gamma)
\mathfrak K^\eta(z-\gamma).
\end{align}
On the other hand, using the consequence  $\partial_{s}\mathfrak{K}^{\eta}(z-\gamma)= -\mathrm{i}\partial_{\gamma}\mathfrak{K}^{\eta}(z-\gamma)$ of    \eqref{eq:kernel-derivatives} together with integration by parts, we obtain
\begin{align}\label{eq:first-s-derivative-V-eta}
    \partial_{s}V^{\eta}(t+\mathrm{i}s,x)&=-\mathrm{i}\sqrt{\frac{\eta}{2\pi}}\int_{\mathbb{R}}\partial_{\gamma}\mathfrak{K}^{\eta}(z-\gamma) \psi(\gamma)u(\gamma,x)\,d\gamma\\
    &= \mathrm{i}\sqrt{\frac{\eta}{2\pi}}\int_{\mathbb{R}}\mathfrak{K}^{\eta}(z-\gamma) \partial_{\gamma}\left(\psi(\gamma)u(\gamma,x)\right)\,d\gamma\\
     &= \mathrm{i}\sqrt{\frac{\eta}{2\pi}}\int_{\mathbb{R}}\mathfrak{K}^{\eta}(z-\gamma) \partial_{\gamma}\psi(\gamma)u(\gamma,x)\,d\gamma+  \mathrm{i}\sqrt{\frac{\eta}{2\pi}}\int_{\mathbb{R}}\mathfrak{K}^{\eta}(z-\gamma) \psi(\gamma)\partial_{\gamma}u(\gamma,x)\,d\gamma\\
     &= \mathrm{i}\mathcal T^\eta(\psi'u)(t+\mathrm{i}s,x)+\mathrm{i}\mathcal T^\eta
\bigl(\psi\partial_\gamma u\bigr)(t+\mathrm{i}s,x),
\end{align}
where $\mathcal{T}^{\eta}(f)(z,x):= \sqrt{\frac{\eta}{2\pi}}\int_{\mathbb{R}}\mathfrak{K}^{\eta}(z-\gamma) f(\gamma,x)\,d\gamma$. In particular,
\begin{equation}\label{eq:solve-time-derivative}
    \mathcal{T}^{\eta}(\psi\partial_\gamma u)
    =-\mathrm{i}\partial_sV^\eta-\mathcal{T}^{\eta}(\psi'u).
\end{equation}
Applying the same argument again gives
\begin{align}\label{eq:second-derivative}
\partial_s^2V^\eta
&=\mathrm{i}^2\mathcal{T}^{\eta}\bigl(\partial_\gamma^2(\psi u)\bigr)\\
&=-\mathcal{T}^{\eta}(\psi''u)
  -2\mathcal{T}^{\eta}(\psi'\partial_\gamma u)
  -\mathcal{T}^{\eta}(\psi\partial_\gamma^2u).
\end{align}
Combining the above equation together with $ \Delta_xV^\eta=\mathcal{T}^{\eta}(\psi\Delta_xu)$, we deduce 
\begin{align}
\partial_s^2V^\eta+\Delta_xV^\eta
&=-\mathcal{T}^{\eta}(\psi''u)
  -2\mathcal{T}^{\eta}(\psi'\partial_\gamma u)
  -\mathcal{T}^{\eta}\bigl(\psi(\partial_\gamma^2-\Delta_x)u\bigr).
\end{align}
Using equation \eqref{eq:wave-UCP-extended} in the last term of the above expression, we obtain 
\begin{align}
\partial_s^2V^\eta+\Delta_xV^\eta
=&-\mathcal{T}^{\eta}(\psi''u)
-2\mathcal{T}^{\eta}(\psi'\partial_\gamma u)
-\mathcal{T}^{\eta}(\psi F)
+a(x)\mathcal{T}^{\eta}(\psi\partial_\gamma u)
+q(x)V^\eta.
\end{align}
From the equation \eqref{eq:solve-time-derivative}, we obtained the modified FBI transformed equation as follows
\begin{equation}\label{eq:modified-fbi-equation}
\begin{aligned}
&\Bigl(
\partial_s^2+\Delta_x
+\mathrm{i} a(x)\partial_s-q(x)
\Bigr)V^\eta(t+\mathrm{i} s,x)
\\
&\quad
=-\mathcal{T}^{\eta}(\psi F)(t+\mathrm{i} s,x)
  -\mathcal{T}^{\eta}(\psi''u)(t+\mathrm{i} s,x)
\\
&\qquad\qquad
  -2\mathcal{T}^{\eta}(\psi'\partial_\gamma u)(t+\mathrm{i} s,x)
  -a(x)\mathcal{T}^{\eta}(\psi'u)(t+\mathrm{i} s,x).
\end{aligned}
\end{equation}
Equivalently, in integral form,
\begin{align}\label{eq:correct-fbi-integral}
\begin{aligned}
&\Bigl(
\partial_s^2+\Delta_x
+\mathrm{i} a(x)\partial_s-q(x)
\Bigr)V^\eta(t+\mathrm{i} s,x)
\\
&\quad
=-\sqrt{\frac{\eta}{2\pi}}\int_\mathbb{R}\mathfrak{K}^{\eta}(t+\mathrm{i} s-\gamma)
\Bigl[
\psi(\gamma)F(\gamma,x)
+\psi''(\gamma)u(\gamma,x)
\\
&\hspace{7.4cm}
+2\psi'(\gamma)\partial_\gamma u(\gamma,x)
+a(x)\psi'(\gamma)u(\gamma,x)
\Bigr]\,d\gamma.
\end{aligned}
\end{align}

Next, we derive a Carleman estimate associated with the
operator $\mathcal{P}_{a,q}:=\partial_s^2+\Delta_x+\mathrm{i}a(x)\partial_s-q(x).$

 We begin by introducing suitable weight functions. There exists a function $\ell\in C^{2}(\overline{\mathcal{O}})$ that satisfies the following properties
\begin{equation*}
 \ell>0
    \quad \text{in } \mathcal{O},
    \qquad
    |\nabla\ell|>0
    \quad \text{on } \overline{\mathcal{O}},
\end{equation*}
and
\begin{equation*}
   \ell=0,
    \qquad
    \partial_{\nu}\ell\leq 0
    \quad \text{on }
    \partial\mathcal{O}\setminus\Gamma 
\end{equation*}
where $\nu$ denotes the outward unit normal vector to
$\partial\mathcal{O}$. The existence of such a function is discussed
in \cite[Section 4.1]{Bellassoued_Fraj}.

Since $\ell$ is strictly positive in $\mathcal{O}$ and
$\overline{\mathcal{O}_{2}\setminus\mathcal{O}_{3}}$ is a compact subset
of $\mathcal{O}$, there exists a constant $\kappa>0$ such that 
$\ell(x)\geq 3\kappa$ for $ x\in\mathcal{O}_{2}\setminus\mathcal{O}_{3}$. On the other hand, since $\ell|_{\partial\mathcal{O}\setminus\Gamma}=0$, the continuity of $\ell$ implies that there exists a sufficiently small
open neighborhood $\mathcal{O}'$ of
$\partial\mathcal{O}\setminus\Gamma$ such that $\ell(x)\leq \kappa$ for $x\in\mathcal{O}^{\prime}$ and $\mathcal{O}^{\prime}\cap \overline{\mathcal{O}_{1}}=\emptyset$. Finally,  we define the weight function $\Phi(s,x):= \exp\{\varrho(\ell(x)-\alpha s^2+\alpha_{0})\}$ for $(s,x)\in \R\times\mathcal{O}$, $\varrho\geq\varrho_{0}$, $\alpha>0$, $\varrho_{0}\geq 1$ and $\alpha_{0}\geq 0$. Building on these, we  are now ready to state the Carleman estimate associated with the weight function $\Phi$ given above, in the following lemma. 
\begin{lemma}\label{lemma:Carleman_estimate}
Let $T_\star>0$ then for  $(a,q)\in \mathcal{A}$, there exist  positive constants $\varrho_{0}$, $\mu_{0}$ and $C$ such that, for every
$\varrho\geq\varrho_0$, $\mu\geq\mu_0$ and every $v\in H^{2}((-T_{\star},T_{\star})\times \mathcal{O})$ satisfying
   \begin{equation}
\begin{cases}
v(\pm T_{\star},x)=\partial_s v(\pm T_{\star},x)=0,
& x\in\mathcal{O},\\[0.3em]
v(s,x)=0,
& (s,x)\in(-T_{\star},T_{\star})\times\partial\mathcal{O},
\end{cases}
\end{equation}
the estimate
\begin{align}
    &\varrho
\int_{(-T_{\star},T_{\star})\times \mathcal{O}}
\left( |\partial_s v|^2 + |\nabla_xv|^2 + |v|^2
\right)e^{2\mu\Phi}\,dx\,ds\\
&\qquad\qquad\leq C\int_{(-T_{\star},T_{\star})\times \mathcal{O}}
\left| \mathcal P_{a,q}v\right|^2 e^{2\mu\Phi}\,dx\,ds + C\int_{(-T_\star,T_\star)\times\Gamma} \mu\varrho\Phi|\partial_\nu v|^2
e^{2\mu\Phi}\,dS_x\,ds
\end{align}
holds whenever  $\alpha_0>\alpha T_\star^2$ and  $\Phi(s,x):= \exp\{\varrho(\ell(x)-\alpha s^2+\alpha_{0})\}$ for $(s,x)\in \R\times\mathcal{O}$.
\end{lemma}
\begin{proof}
Let $\mathcal{P}_0 := \partial_s^2+\Delta_x$ and $\sigma(s,x):=\mu\varrho\Phi(s,x)$. By the Carleman estimate for the operator $\mathcal{P}_0$ given in \cite[Lemma 4.1]{Bellassoued_Fraj}, there exist positive constants $\varrho_0$, $\mu_0$, and $C$ such that, for every
$\varrho\geq\varrho_0$ and $\mu\geq\mu_0$, the estimate 
\begin{align}\label{eq:strong-principal-Carleman}
&\varrho
\int_{(-T_\star,T_\star)\times\mathcal O}\Bigg(\sigma^{-1}\sum_{|\beta|=2}
|\partial^\beta v|^2+\sigma\left(|\partial_sv|^2+|\nabla_xv|^2\right)
+\sigma^3|v|^2\Bigg)e^{2\mu\Phi}\,dx\,ds\notag\\
&\qquad\leq C \int_{(-T_\star,T_\star)\times\mathcal O} |\mathcal P_0v|^2
e^{2\mu\Phi}\,dx\,ds + C\int_{(-T_\star,T_\star)\times\Gamma}\sigma|\partial_\nu v|^2
e^{2\mu\Phi}\,dS_x\,ds
\end{align}
holds.

Since $\ell>0$ in $\mathcal{O}$, continuity gives $\ell\geq 0$ on $\overline{\mathcal O}$. Hence, for $|s|\leq T_\star$, the assumption
$\alpha_0>\alpha T_\star^2$ yields $\ell(x)-\alpha s^2+\alpha_0 \geq \alpha_0-\alpha T_\star^2>0$. Consequently, $\Phi(s,x) = \exp\left( \varrho \bigl(\ell(x)-\alpha s^2+\alpha_0\bigr)\right)\geq1$. Enlarging $\varrho_0$ and $\mu_0$, we can assume that $\sigma(s,x)=\mu\varrho\Phi(s,x)\geq1$ for $(s,x)\in(-T_\star,T_\star)\times\mathcal O$. It follows that
\[
\sigma \left(|\partial_sv|^2+|\nabla_xv|^2\right)+ \sigma^3|v|^2\geq
|\partial_sv|^2+|\nabla_xv|^2+|v|^2.
\]
Dropping the nonnegative second-order term in
\eqref{eq:strong-principal-Carleman}, we obtain
\begin{align}\label{eq:weak-principal-Carleman}
&\varrho
\int_{(-T_\star,T_\star)\times\mathcal O}\left(|\partial_sv|^2+|\nabla_xv|^2+|v|^2\right)e^{2\mu\Phi}\,dx\,ds\\
&\qquad\leq C \int_{(-T_\star,T_\star)\times\mathcal O} |\mathcal P_0v|^2
e^{2\mu\Phi}\,dx\,ds + C\int_{(-T_\star,T_\star)\times\Gamma}\mu\varrho\Phi|\partial_\nu v|^2e^{2\mu\Phi}\,dS_x\,ds.
\end{align}
Since $a=q=0$ in $\mathcal{O}$, we have $\mathcal P_{a,q}v = \mathcal P_{0}v $ in $(-T_\star,T_\star)\times\mathcal O$. This completes the proof of the Lemma.
\end{proof}
Next, we will define some notations which will be useful in our subsequent analysis
\begin{equation}\label{eq:weight-levels}
    w_{\mathrm{tar}}:=e^{\varrho(2\kappa+\alpha_0)},
    \quad
    w_{\mathrm{sp}}:=e^{\varrho(\kappa+\alpha_0)},
    \quad
    w_{\mathrm{ti}}:=e^{\varrho\left(\alpha_0-\frac{\alpha T_\star^{2}}{8}\right)}
    \quad \text{ and }
    w_{\max}:=e^{\varrho\left(\max_{x\in\overline{\mathcal O}}\ell(x)+\alpha_0\right)} .
\end{equation}
Since $\ell\geq3\kappa$ on $\mathcal O_2\setminus\mathcal O_3$, we have
$\max_{x\in\overline{\mathcal O}}\ell(x)\geq3\kappa>2\kappa$. Hence, 
$w_{\mathrm{ti}}<w_{\mathrm{sp}}<w_{\mathrm{tar}}<w_{\max}$, and define
\begin{equation}\label{eq:def-a-m}
    \mathfrak a:=2\bigl(w_{\max}-w_{\mathrm{tar}}\bigr)+1>0,
    \qquad
    \mathfrak m:=2\bigl(w_{\mathrm{tar}}-w_{\mathrm{sp}}\bigr)>0
    \quad\text{ and }\quad
    \theta:=\frac{\mathfrak m}{\mathfrak a+\mathfrak m}\in(0,1).
\end{equation}
We now want to apply the Carleman estimate of Lemma
\ref{lemma:Carleman_estimate} on the domain
$(-T_\star,T_\star)\times\mathcal O$ to the function $V^{\eta}$. This function, however, depends on a complex variable, so we freeze the
real part of $z=t+\mathrm{i}s$ and set
$V_t^{\eta}(s,x):=V^{\eta}(t+\mathrm{i}s,x)$ for
$(s,x)\in(-T_\star,T_\star)\times\mathcal O$, where $t$ is a parameter,
and every constant obtained below will be independent of it. Since
$V_t^{\eta}$ vanishes neither at $s=\pm T_\star$ nor on
$\partial\mathcal O\setminus\Gamma$, so it does not satisfy the hypotheses of Lemma
\ref{lemma:Carleman_estimate}. The proof below therefore applies the
estimate to the cut-off function $\xi(s)\chi(x)V_t^{\eta}$ and
absorbs the resulting commutator terms, using the fact that the weight
$\Phi$ is strictly larger on $\mathcal O_2\setminus\mathcal O_3$ than on
the supports of the derivatives of $\xi$ and $\chi$. The precise
statement is the following.

\begin{lemma}\label{lemma:Veta-interpolation}
Let $(a,q)\in\mathcal A$ and  $u\in H^{2}(Q)$ is a solution of
\eqref{eq:wave-UCP-extended} with $F\in L^{2}(Q)$. Assume that the parameters
$\alpha,\alpha_0,T',T_0,T_{\star}$ and $T$ satisfy
\begin{equation}\label{eq:parameter-constraints}
   \max_{x\in\overline{\mathcal O}}\ell(x) \leq\frac{\alpha T_\star^{2}}{8},
    \qquad
    \alpha_0>\alpha T_\star^{2},
    \qquad
    0<T'<\min\left\{\frac{T_\star}{3},\ \sqrt{\frac{\kappa}{\alpha}},\ \frac{T_0}{8}\right\}
    \qquad \text{ and }\quad T>T_0> \frac{4T_\star}{3\sqrt\theta},
\end{equation}
with $\theta\in(0,1)$ as in \eqref{eq:def-a-m}. Then there exist constants $C>0$, $\alpha_1>0$, $\alpha_2>0$ and $\eta_0\geq1$, all independent of $t$, such that  the following estimate
\begin{equation}\label{eq:Veta-interpolation-norms}
\begin{aligned}
\|V_t^{\eta}\|_{H^{1}((-T',T')\times \left(\mathcal{O}_{2}\setminus\mathcal{O}_{3}\right))}
\leq
C\Bigl(
&e^{-\alpha_1\eta}\|u\|_{H^{1}((-T,T)\times\mathcal O)}
\\
&+
e^{\alpha_2\eta}
\bigl(
\|F\|_{L^{2}((-T,T)\times\mathcal O)}
+
\|\partial_\nu V_t^{\eta}\|_{L^{2}((-T_\star,T_\star)\times\Gamma)}
\bigr)
\Bigr).
\end{aligned}
\end{equation}
holds for every $t\in\left(-\frac{T_0}{4},\ \frac{T_0}{4}\right)$ and $\eta\geq\eta_0$.
\end{lemma}

\begin{proof}
We begin with introducing time and space dependent cut-off functions which will be required in our subsequent analysis. Let $\xi\in C_{c}^{\infty}\left(-\frac{3T_{\star}}{4},\frac{3T_{\star}}{4}\right)$ be such that $0\leq \xi\leq 1$,  $\xi(s)=1$, for every $s\in\left[-\frac{T_{\star}}{2},\frac{T_{\star}}{2}\right]$ and $\operatorname{supp}\xi' \cup \operatorname{supp}\xi'' \subset\left\{s\in\mathbb R:\frac{T_\star}{2}\leq |s|\leq\frac{3T_\star}{4} \right\}$,   where $\xi'$ and $\xi''$ stand for first and second order derivative of $\xi$ respectively.
Next, we introduce a spatial cutoff function. Let $\widetilde{\mathcal{O}}'\subset\mathcal{O}'$ be an open neighborhood of $\partial\mathcal{O}\setminus\partial\Omega$ in $\Omega$ such that $\overline{\widetilde{\mathcal{O}}'}\subset\mathcal{O}'$ and $\overline{\mathcal{O}_{2}\setminus\mathcal{O}_{3}}
    \subset    \mathcal{O}\setminus\overline{\mathcal{O}'}$. Now choose $\chi\in C^{\infty}(\overline{\mathcal{O}})$  such that
 $0\leq\chi\leq1$, and

\[
\chi(x)=
\begin{cases}
0 & x\in\widetilde{\mathcal{O}}',\\
1 & x\in\mathcal{O}\setminus\mathcal{O}'.
\end{cases}\]
Therefore, $\chi\equiv 1$, in
$\mathcal{O}_{2}\setminus\mathcal{O}_{3}$, and
$
\operatorname{supp}(\nabla\chi)
\cup
\operatorname{supp}(\Delta\chi)
\subset
\mathcal{O}'\setminus
\overline{\widetilde{\mathcal{O}}'}.$ To see it pictorially, we refer to the figure below.
\vspace{-1.5cm}
\begin{figure}[ht]
\centering
\begin{tikzpicture}[scale=1]

\def\blobA{plot [smooth cycle, tension=0.72] coordinates {
        (3.32,0.00) (3.25,1.88) (1.78,3.09) (0.00,3.04)
        (-1.45,2.51) (-2.74,1.58) (-3.18,0.00) (-2.41,-1.39)
        (-1.23,-2.13) (0.00,-2.46) (1.28,-2.22) (2.43,-1.40)
    }}
\def\blobB{plot [smooth cycle, tension=0.72] coordinates {
        (3.10,0.00) (3.03,1.75) (1.66,2.88) (0.00,2.83)
        (-1.35,2.34) (-2.56,1.48) (-2.97,0.00) (-2.25,-1.30)
        (-1.15,-1.99) (0.00,-2.30) (1.20,-2.08) (2.27,-1.31)
    }}
\def\blobC{plot [smooth cycle, tension=0.72] coordinates {
        (2.68,0.00) (2.62,1.51) (1.44,2.49) (0.00,2.45)
        (-1.17,2.02) (-2.21,1.28) (-2.57,0.00) (-1.94,-1.12)
        (-0.99,-1.72) (0.00,-1.98) (1.04,-1.79) (1.96,-1.13)
    }}
\def\blobD{plot [smooth cycle, tension=0.72] coordinates {
        (2.39,0.00) (2.34,1.35) (1.28,2.22) (0.00,2.19)
        (-1.04,1.81) (-1.98,1.14) (-2.29,0.00) (-1.73,-1.00)
        (-0.89,-1.53) (0.00,-1.77) (0.92,-1.60) (1.75,-1.01)
    }}
\def\blobE{plot [smooth cycle, tension=0.72] coordinates {
        (2.06,0.00) (2.01,1.16) (1.11,1.92) (0.00,1.88)
        (-0.90,1.56) (-1.70,0.98) (-1.97,0.00) (-1.49,-0.86)
        (-0.76,-1.32) (0.00,-1.52) (0.80,-1.38) (1.51,-0.87)
    }}
\def\blobF{plot [smooth cycle, tension=0.72] coordinates {
        (1.73,0.00) (1.69,0.98) (0.93,1.61) (0.00,1.58)
        (-0.75,1.31) (-1.43,0.82) (-1.66,0.00) (-1.25,-0.72)
        (-0.64,-1.11) (0.00,-1.28) (0.67,-1.16) (1.26,-0.73)
    }}

\fill[blue!7]    \blobA;
\fill[green!35]  \blobB;
\fill[blue!7]    \blobC;
\fill[orange!30] \blobD;
\fill[red!30]    \blobE;
\fill[white]     \blobF;
\draw[very thick]  \blobA;
\draw[dashed]      \blobD;
\draw[dashed]      \blobE;
\draw[thick,dotted]\blobF;

\begin{scope}
\clip (0,0) -- (35:4.4) arc (35:115:4.4) -- cycle;
\draw[very thick,blue] \blobA;
\end{scope}
\node at (0,0) {$\Omega_0$};
\node[font=\small] at (80:3.62) {\textcolor{blue}{$\Gamma$}};
\node[font=\small] at (172:3.58) {$\partial\Omega$};
\node[font=\small,align=center] at (268:0.82) {$\partial\mathcal{O}\setminus\partial\Omega$};
\draw[->,thin] (0,-0.98) -- (0,-1.25);
\node[right,font=\small] at (4.35,2.15) (la) {$\mathcal{O}_{1}$:  $\chi=1$, $\nabla\chi=0$};
\draw[->,thin] (la.west) -- (3.32,1.21);
\node[right,font=\small] at (4.35,1.05) (lb) {$\mathcal{O}_{2}\setminus\mathcal{O}_{3}$: $\chi=1$, $\ell\ge3\kappa$};
\draw[->,thin] (lb.west) -- (2.99,0.42);
\node[right,font=\small] at (4.35,-0.15) (lc) {$\mathcal{O}'\setminus\overline{\widetilde{\mathcal{O}}'}$: $0\le\chi\le1$};
\draw[->,thin] (lc.west) -- (1.98,-0.49);
\node[right,font=\small] at (4.35,-1.25) (ld) {$\widetilde{\mathcal{O}}'$: $\chi=0$, $\ell\le\kappa$};
\draw[->,thin] (ld.west) -- (1.54,-0.62);
\node[left,font=\small] at (-4.1,-2.1) (le) {$\mathcal{O}=\overline{\Omega}\setminus\overline{\Omega_{0}}$};
\draw[->,thin] (le.east) -- (-2.50,-1.17);
\end{tikzpicture}
\caption{}
\end{figure}

We now define $v_t^\eta(s,x) := \xi(s)\chi(x)V_t^\eta(s,x)$ for $(s,x)\in (-T_\star,T_\star)\times\mathcal O$. Using the  properties of $\xi$, we have that $v_t^\eta(\pm T_\star,x)=\partial_s v_t^\eta(\pm T_\star,x)=0,$  for $x\in \mathcal{O}$. Moreover, $\chi\equiv 0$, on $\partial\mathcal O\setminus\partial\Omega$ and  $V_t^\eta\equiv 0$ on $\partial\Omega$, yields that $v_t^\eta\equiv 0$, on $(-T_\star,T_\star)\times\partial\mathcal O$. Thus, $v_t^\eta$ satisfies all the conditions required to apply   Lemma \ref{lemma:Carleman_estimate} and  
\begin{align}\label{eq:expanded-Pv}
    \mathcal{P}_{a,q} v_t^\eta&= \xi\chi\,\mathcal{P}_{a,q}V_t^\eta+ 2(\partial_{s}\xi)\chi\,\partial_{s}V_t^\eta+(\partial_{s}^{2}\xi)\chi\, V_t^\eta\nonumber\\&\qquad\qquad+2\xi\, \nabla_x\chi\cdot\nabla_x V_t^\eta+\xi(\Delta_{x}\chi)\, V_t^\eta+\mathrm{i}a(x)(\partial_{s}\xi)\chi\, V_t^\eta .
\end{align}
We now introduce the  time and space  cutoff regions by 
\begin{equation*}
\mathcal{E}_t :=
\left(\operatorname{supp}\xi'\cup
\operatorname{supp}\xi''\right)\times\mathcal{O},\ \mbox{and}\ \mathcal{E}_x:=
\operatorname{supp}\xi \times \left( \operatorname{supp}(\nabla\chi)\cup \operatorname{supp}(\Delta\chi)\right) 
\end{equation*}
 respectively. We also denote by   
 $\mathbf{1}_{\mathcal{E}_t}$ and $\mathbf{1}_{\mathcal{E}_x}$, the  characteristic functions
of the sets $\mathcal{E}_t$ and $\mathcal{E}_x$, respectively. 
Now observe that the  terms in \eqref{eq:expanded-Pv} containing $\xi'$ or $\xi''$ are supported in $\mathcal{E}_t$, whereas the terms containing $\nabla\chi$ or $\Delta\chi$ are supported in $\mathcal{E}_x$, therefore, after using  the boundedness of $\xi$, $\chi$, their derivatives and the fact that $a\in L^\infty(\mathcal{O})$ in  \eqref{eq:expanded-Pv}, we get that  there exists a constant $C>0$ such that
\begin{equation}\label{eq:pointwise-operator-estimate}
|\mathcal{P}_{a,q}v_t^\eta|^2\leq
C\left(|\mathcal{P}_{a,q}V_t^\eta|^2+ \mathbf{1}_{\mathcal{E}_t}
\left( |\partial_sV_t^\eta|^2+|V_t^\eta|^2\right) + \mathbf{1}_{\mathcal{E}_x}
\left( |\nabla_xV_t^\eta|^2+|V_t^\eta|^2\right) \right).
\end{equation}
Applying the Carleman estimate to $v_t^\eta$ from Lemma \ref{lemma:Carleman_estimate}, we obtain
\begin{align}\label{eq:Carleman-applied-localized}
&\varrho
\int_{(-T_{\star},T_{\star})\times \mathcal{O}}
\left(
    |\partial_sv_t^\eta|^2
    +
    |\nabla_xv_t^\eta|^2
    +
    |v_t^\eta|^2
\right)
e^{2\mu\Phi}\,dx\,ds
\notag\\
&\qquad\leq
C
\int_{(-T_{\star},T_{\star})\times \mathcal{O}}
|\mathcal{P}_{a,q} v_t^\eta|^2
e^{2\mu\Phi}\,dx\,ds +
C\mu\varrho
\int_{(-T_{\star},T_{\star})\times\Gamma}
\Phi
|\partial_\nu v_t^\eta|^2
e^{2\mu\Phi}\,dS_x\,ds.
\end{align}

Since $\mathcal O'\cap\overline{\mathcal O_1}=\emptyset$ and
$\mathcal O_1$ is a neighborhood of $\partial\Omega$, we have
$\mathcal O_1\subset\mathcal O\setminus\mathcal O'$, therefore, we obtain that 
$\chi\equiv 1$, a neighborhood of $\Gamma$ which gives us $\nabla\chi=0$, in  a neighborhood of $\Gamma$.
After utilizing this, we get  that 
\[\partial_\nu v_t^\eta=\xi\left(\chi\,\partial_\nu V_t^\eta+V_t^\eta\,\partial_\nu\chi\right)=\xi\,\partial_\nu V_t^\eta,\ 
\mbox{on} \ (-T_\star,T_\star)\times\Gamma.\] After combining this identity with \eqref{eq:pointwise-operator-estimate} and \eqref{eq:Carleman-applied-localized}, and restricting the left-hand side to
$(-T',T')\times(\mathcal{O}_{2}\setminus\mathcal{O}_{3})$, where
$\chi\equiv 1$ and $\xi\equiv 1$, gives
\begin{align}\label{eq:carleman-estimate-Vt-eta}
\begin{aligned} 
    &\varrho
\int_{\left(-T^{\prime},T^{\prime}\right)\times \left(\mathcal{O}_{2}\setminus\mathcal{O}_{3}\right)}
\left(
    |\partial_sV_t^\eta|^2
    +
    |\nabla_xV_t^\eta|^2
    +
    |V_t^\eta|^2
\right)e^{2\mu\Phi}\,dx\,ds\\
&\qquad\leq
C
\int_{(-T_{\star},T_{\star})\times \mathcal{O}}
|\mathcal{P}_{a,q}V_t^\eta|^2
e^{2\mu\Phi}\,dx\,ds + C \int_{\mathcal{E}_{t}}
\left( |\partial_s V_t^\eta|^2+|V_t^\eta|^2 \right)
e^{2\mu\Phi}\,dx\,ds\\
&\qquad\quad
+
C
\int_{\mathcal{E}_{x}}
\left(
    |\nabla_x V_t^\eta|^2+|V_t^\eta|^2
\right)
e^{2\mu\Phi}\,dx\,ds +
C\mu\varrho
\int_{(-T_{\star},T_{\star})\times\Gamma}
\Phi
|\partial_\nu V_t^\eta|^2
e^{2\mu\Phi}\,dS_x\,ds.
\end{aligned}
\end{align}
Next, we will compare the weight function $e^{2\mu\Phi}$ in the domains involved in the above integral. Observe that on the domain $\left(-T^{\prime},T^{\prime}\right)\times \mathcal{O}_{2}\setminus\mathcal{O}_{3}$, we have
\begin{align*}
    e^{2\mu\exp\{\varrho(\ell(x)-\alpha s^2+\alpha_{0})\}}&\geq e^{2\mu\exp\{\varrho(3\kappa-\alpha T'^{2}+\alpha_{0})\}}\geq e^{2\mu\exp\{\varrho(2\kappa+\alpha_{0})\}}=  e^{2\mu w_{\mathrm{tar}}}.
\end{align*}
On the spatial cutoff region $\mathcal{E}_{x}$, we have $ e^{2\mu\exp\{\varrho(\ell(x)-\alpha s^2+\alpha_{0})\}}\leq e^{2\mu\exp\{\varrho(\kappa+\alpha_{0})\}}=e^{2\mu w_{\mathrm{sp}}}$ whereas on the time cutoff region $\mathcal{E}_{t}$, we have
$|s|\geq T_{\star}/2$, hence,
\begin{align*}
    e^{2\mu\exp\{\varrho(\ell(x)-\alpha s^2+\alpha_{0})\}} \leq
    e^{2\mu\exp{\varrho\left(
        \frac{\alpha T_{\star}^{2}}{8}
        -\frac{\alpha T_{\star}^{2}}{4}
        +\alpha_{0}
    \right)}}
\leq
    e^{2\mu\exp{\varrho\left(
        \alpha_{0}
        -\frac{\alpha T_{\star}^{2}}{8}
    \right)}}=e^{2\mu w_{\mathrm{ti}}}.
\end{align*}
Also observe that  $e^{2\mu\Phi}\leq e^{2\mu w_{\max}}$
in $(-T_\star,T_\star)\times\mathcal O$ and  $\kappa>0$,  implies  that $w_{\mathrm{tar}}>w_{\mathrm{sp}}$ and $w_{\mathrm{tar}}>w_{\mathrm{ti}}$.  Now, after dividing the  inequality in \eqref{eq:carleman-estimate-Vt-eta} by $e^{2\mu w_{\mathrm{tar}}}$ and   using \eqref{eq:def-a-m}, we obtain
\begin{align}\label{Eq: intermediate inequality}
\begin{aligned}
    &\int_{\left(-T^{\prime},T^{\prime}\right)\times \mathcal{O}_{2}\setminus\mathcal{O}_{3}}
\left(
    |\partial_sV_t^\eta|^2
    +
    |\nabla_xV_t^\eta|^2
    +
    |V_t^\eta|^2
\right)\,dx\,ds\\
&\qquad\leq
C e^{\mathfrak{a}\mu}\left(
\int_{(-T_{\star},T_{\star})\times \mathcal{O}}
|\mathcal{P}_{a,q}V_t^\eta|^2\,dx\,ds+
\int_{(-T_{\star},T_{\star})\times\Gamma}
|\partial_\nu V_t^\eta|^2\,dS_x\,ds\right)\\
&\qquad\quad
+
C e^{-\mathfrak{m}\mu}
\int_{(-T_{\star},T_{\star})\times \mathcal{O}}
    \left(|\partial_sV_t^\eta|^2+|\nabla_xV_t^\eta|^2+|V_t^\eta|^2\right)
\,dx\,ds,\qquad \text{ for } \mu\geq\mu_{0},
\end{aligned}
\end{align}
where $\mathfrak{a}>0$ and $\mathfrak{m}>0$  are independent of $\mu$. Next, we will minimize the right hand side with respect to $\mu$. To make further calculations easier, we define
\[
Y
:=
\|\mathcal{P}_{a,q}V_t^\eta\|_{L^2\left((-T_{\star},T_{\star})\times \mathcal{O}\right)}^{2}
+
\|\partial_{\nu}V_t^\eta\|_{L^2((-T_{\star},T_{\star})\times\Gamma)}^{2}
\]
and
\[
Z
:=
\int_{(-T_{\star},T_{\star})\times \mathcal{O}}
\left(
|\partial_sV_t^\eta|^2+|\nabla_xV_t^\eta|^2+|V_t^\eta|^2
\right)\,dx\,ds.
\]
Suppose $F(\mu):=e^{\mathfrak a\mu}Y+e^{-\mathfrak m\mu}Z$; then the critical points of $F$ are given by $\mu
=
\frac{1}{\mathfrak a+\mathfrak m}
\log\left(\frac{\mathfrak m Z}{\mathfrak a Y}\right)=:\mu_{\min}$. Also, observe that $F''(\mu_{\min})>0$; hence, $\mu_{\min}$ is the point of minima. Next we divide the proof in  two cases viz. $\mu\geq \mu_{\min}$ and $\mu<\mu_{\min}$. 
Now if $\mu_{\min}\geq\mu_0$, then
\begin{equation}
\begin{aligned}
F(\mu_{\min})
&=
e^{\mathfrak a\mu_{\min}}Y
+
e^{-\mathfrak m\mu_{\min}}Z  \\
&=
Y
\left(
\frac{\mathfrak m Z}{\mathfrak a Y}
\right)^{\frac{\mathfrak a}{\mathfrak a+\mathfrak m}}
+
Z
\left(
\frac{\mathfrak m Z}{\mathfrak a Y}
\right)^{-\frac{\mathfrak m}{\mathfrak a+\mathfrak m}}  \\
&=
\left(\frac{\mathfrak m}{\mathfrak a}\right)^{\frac{\mathfrak a}{\mathfrak a+\mathfrak m}}
Y^{\frac{\mathfrak m}{\mathfrak a+\mathfrak m}}
Z^{\frac{\mathfrak a}{\mathfrak a+\mathfrak m}}
+
\left(\frac{\mathfrak a}{\mathfrak m}\right)^{\frac{\mathfrak m}{\mathfrak a+\mathfrak m}}
Y^{\frac{\mathfrak m}{\mathfrak a+\mathfrak m}}
Z^{\frac{\mathfrak a}{\mathfrak a+\mathfrak m}}  \\
&=
\left[
\left(\frac{\mathfrak m}{\mathfrak a}\right)^{\frac{\mathfrak a}{\mathfrak a+\mathfrak m}}
+
\left(\frac{\mathfrak a}{\mathfrak m}\right)^{\frac{\mathfrak m}{\mathfrak a+\mathfrak m}}
\right]
Y^{\frac{\mathfrak m}{\mathfrak a+\mathfrak m}}
Z^{\frac{\mathfrak a}{\mathfrak a+\mathfrak m}}.
\end{aligned}    
\end{equation}
Since the quantity inside the square brackets depends only on the fixed constants
$\mathfrak a$ and $\mathfrak m$, it can be absorbed into the generic constant $C$.
Therefore $F(\mu_{\min})\leq C Y^{\frac{\mathfrak m}{\mathfrak a+\mathfrak m}}Z^{\frac{\mathfrak a}{\mathfrak a+\mathfrak m}}$. Next if $\mu_{\min}<\mu_0$, then by the definition of
$\mu_{\min}$ we have $Z<C_0Y$, where
$C_0:=\frac{\mathfrak a}{\mathfrak m}e^{(\mathfrak a+\mathfrak m)\mu_0}$. Now
since $(-T',T')\times(\mathcal O_2\setminus\mathcal O_3)\subset(-T_\star,T_\star)\times\mathcal O$,
therefore the left hand side of the inequality \eqref{Eq: intermediate inequality}, is bounded by $Z$, and
hence
\[
    Z=Z^{\frac{\mathfrak m}{\mathfrak a+\mathfrak m}}Z^{\frac{\mathfrak a}{\mathfrak a+\mathfrak m}}
    \leq
    C_0^{\frac{\mathfrak m}{\mathfrak a+\mathfrak m}}
    Y^{\frac{\mathfrak m}{\mathfrak a+\mathfrak m}}Z^{\frac{\mathfrak a}{\mathfrak a+\mathfrak m}} .
\]
Thus, in both the cases, we get 
\begin{align}
\begin{aligned}
    &\int_{\left(-T^{\prime},T^{\prime}\right)\times \left(\mathcal{O}_{2}\setminus\mathcal{O}_{3}\right)}
\left(
    |\partial_s V_t^\eta|^2
    +
    |\nabla_x V_t^\eta|^2
    +
    |V_t^\eta|^2
\right)\,dx\,ds\\
&\qquad\leq
C \left(
\int_{(-T_{\star},T_{\star})\times \mathcal{O}}
|\mathcal{P}_{a,q} V_t^\eta|^2\,dx\,ds+
\int_{(-T_{\star},T_{\star})\times\Gamma}
|\partial_\nu V_t^\eta|^2\,dS_x\,ds\right)^{\frac{\mathfrak{m}}{\mathfrak{a}+\mathfrak{m}}}\\
&\qquad\qquad\times\left(\int_{(-T_{\star},T_{\star})\times \mathcal{O}}
\left(
|\partial_sV_t^\eta|^2+|\nabla_xV_t^\eta|^2+|V_t^\eta|^2
\right)\,dx\,ds\right)^{\frac{\mathfrak{a}}{\mathfrak{a}+\mathfrak{m}}}.
\end{aligned} 
\end{align}
Taking the square root above, combined with $\sqrt{A^2+B^2}\leq A+B$ for $A,B\geq0$, we conclude
\begin{align}\label{eq:interpolation-Veta}
&\|V_t^\eta\|_{H^1((-T',T')\times(\mathcal O_2\setminus\mathcal O_3))}
\notag\\
&\qquad\leq
C
\left(
\|\mathcal P_{a,q}V_t^\eta\|_{L^2((-T_\star,T_\star)\times\mathcal O)}
+
\|\partial_\nu V_t^\eta\|_{L^2((-T_\star,T_\star)\times\Gamma)}
\right)^\theta
\times
\|V_t^\eta\|_{H^1((-T_\star,T_\star)\times\mathcal O)}^{1-\theta},
\end{align}
where $\theta=\frac{\mathfrak m}{\mathfrak a+\mathfrak m}\in(0,1)$ is as in \eqref{eq:def-a-m}. We can rewrite \eqref{eq:correct-fbi-integral} as follows
 $\mathcal P_{a,q}V_t^\eta=-F_{t}^\eta+R_{t}^\eta$, where $F_{t}^\eta$ is the FBI transform of $\psi F$, and $R_{t}^\eta$ is defined as follows
\begin{equation}
\label{eq:R-eta-definition}
\begin{aligned}
R_{t}^\eta(s,x)
&:=-\mathcal{T}^{\eta}\!\left(\psi''u\right)(t+\mathrm{i}s,x)
-2\mathcal{T}^{\eta}\!\left(\psi'\partial_\gamma u\right)(t+\mathrm{i}s,x)
-a(x)\mathcal{T}^{\eta}\!\left(\psi'u\right)(t+\mathrm{i}s,x) \\
&=
-\mathcal{T}^{\eta}\!\left(
\psi''u
+2\psi'\partial_\gamma u
+a(x)\psi'u
\right)(t+\mathrm{i}s,x).
\end{aligned}
\end{equation}
Let $h:= \psi''u
+2\psi'\partial_\gamma u
+a(x)\psi'u$.  Now since \(\psi\equiv 1\) on \([-T_{0},T_{0}]\), we have
\[
\operatorname{supp}\psi'\cup\operatorname{supp}\psi''
\subseteq\{\gamma: T_{0}\leq |\gamma|\leq T  \}.
\]
Now for  $|t|<\frac{T_0}{4}$ and $\gamma\in \operatorname{supp}\psi'\cup\operatorname{supp}\psi''$, we have
$|t-\gamma|\geq|\gamma|-|t|\geq\frac{3T_0}{4}$; hence, for $|s|\leq T_\star$
\[
\left|\mathfrak K^\eta(t+\mathrm{i}s-\gamma)\right|
=
\exp\left(
-\frac{\eta}{2}\bigl((t-\gamma)^2-s^2\bigr)
\right)\leq \exp\left(
-\frac{\eta}{2}\left(\left(\frac{3T_{0}}{4}\right)^2-T_{\star}^{2}\right)
\right)\leq e^{-c_0\eta},
\]
where $ c_0
    :=
    \frac12
    \left[
    \left(\frac{3T_0}{4}\right)^2-T_{\star}^{2}
    \right]$, which is positive because
$T_0>\frac{4T_\star}{3\sqrt\theta}>\frac{4T_\star}{3}$.
Using the definition of $R_{t}^\eta$ from \eqref{eq:R-eta-definition}, we obtain
\[
\begin{aligned}
\left|R_{t}^\eta(s,x)\right|
&=
\left|
\sqrt{\frac{\eta}{2\pi}}
\int_{\mathbb R}
\mathfrak K^\eta(t+\mathrm{i}s-\gamma)
h(\gamma,x)\,d\gamma
\right| \\
&\leq
\sqrt{\frac{\eta}{2\pi}}\,
e^{-c_0\eta}
\int_{-T}^{T}|h(\gamma,x)|\,d\gamma.
\end{aligned}
\]
After utilizing the Cauchy-Schwarz inequality, we have
\[
    \int_{-T}^{T}|h(\gamma,x)|\,d\gamma
    \leq
    (2T)^{1/2}
    \|h(\cdot,x)\|_{L^2(-T,T)}.
\]
Thus
\[
    \left|R_{t}^\eta(s,x)\right|
    \leq
    C\sqrt{\eta}\,e^{-c_0\eta}
    \|h(\cdot,x)\|_{L^2(-T,T)}\text{ for } (s,x)\in(-T_\star,T_\star)\times\mathcal O.  
\]
Squaring and taking the $L^2$-norm with respect to $(s,x)\in(-T_\star,T_\star)\times\mathcal O$, we get
\[
\|R_t^\eta\|_{L^2((-T_\star,T_\star)\times\mathcal O)}^{2}
\leq
C\,\eta\,e^{-2c_0\eta}
\|h\|_{L^2((-T,T)\times\mathcal O)}^{2}.
\]
Since \(\psi'\), \(\psi''\), and \(a\) are bounded, we have
\[
\begin{aligned}
\|h\|_{L^2((-T,T)\times\mathcal O)}
&\leq
C\left(
\|u\|_{L^2((-T,T)\times\mathcal O)}
+
\|\partial_\gamma u\|_{L^2((-T,T)\times\mathcal O)}
\right) \leq
C\|u\|_{H^1((-T,T)\times\mathcal O)}.
\end{aligned}
\]
Consequently,
\[
\|R_t^\eta\|_{L^2((-T_\star,T_\star)\times\mathcal O)}^{2}
\leq C\,\eta\,e^{-2c_0\eta}\|u\|^2_{H^1((-T,T)\times\mathcal O)}.
\]
Taking the square root, the factor \(\sqrt{\eta}\) can be absorbed
into the exponential. Indeed, for any \(0<c_R<c_0\), there exists \(C>0\)
such that $\sqrt{\eta}\,e^{-c_0\eta}\leq C e^{-c_R\eta}.$ Therefore,
\[
\|R_t^\eta\|_{L^2((-T_\star,T_\star)\times\mathcal O)}
\leq
C e^{-c_R\eta}
\|u\|_{H^1((-T,T)\times\mathcal O)}.
\]
Since $F_{t}^\eta$ is the FBI transform of $\psi F$ combined with the fact that \(\psi\in C_c^\infty(-T,T)\), we have
\begin{align*}
\begin{aligned} 
   F_{t}^\eta(s,x)
&=
\sqrt{\frac{\eta}{2\pi}}
\int_{\mathbb R}
\mathfrak K^\eta(t+\mathrm{i}s-\gamma)
\psi(\gamma)F(\gamma,x)\,d\gamma \\
& =
\sqrt{\frac{\eta}{2\pi}}
\int_{-T}^{T}
\mathfrak K^\eta(t+\mathrm{i}s-\gamma)
\psi(\gamma)F(\gamma,x)\,d\gamma.
\end{aligned}
\end{align*}
Therefore,
\begin{align*}
|F_{t}^\eta(s,x)|
&\leq
\sqrt{\frac{\eta}{2\pi}}
\int_{-T}^{T}
\left|
\mathfrak K^\eta(t+\mathrm{i}s-\gamma)
\psi(\gamma)F(\gamma,x)
\right|\,d\gamma .
\end{align*}
Using the explicit form of \(\mathfrak K^\eta\), we get
\[
\left|\mathfrak K^\eta(t+\mathrm{i}s-\gamma)\right|
=
\exp\left(
-\frac{\eta}{2}(t-\gamma)^2+\frac{\eta}{2}s^2
\right)
\leq
\exp\left(\frac{\eta}{2}T_\star^2\right),
\]
where we used \((t-\gamma)^2\geq 0\) and \(|s|\leq T_\star\). Hence,
\begin{align*}
|F_{t}^\eta(s,x)|
&\leq
\sqrt{\frac{\eta}{2\pi}}
\exp\left(\frac{\eta}{2}T_\star^2\right)
\int_{-T}^{T}
|\psi(\gamma)F(\gamma,x)|\,d\gamma .
\end{align*}
Using  the Cauchy-Schwarz inequality, we get
\begin{align*}
    | F_{t}^\eta(s,x)|&\leq \sqrt{\frac{\eta}{2\pi}}\exp\left(\frac{\eta}{2}T_\star^2\right)\left(
\int_{-T}^{T}|\psi(\gamma)|^2\,d\gamma
\right)^{1/2}
\left(
\int_{-T}^{T}|F(\gamma,x)|^2\,d\gamma
\right)^{1/2}\\
&\leq
C\sqrt{\eta}\,
\exp\left(\frac{\eta}{2}T_\star^2\right)
\left(
\int_{-T}^{T}|F(\gamma,x)|^2\,d\gamma
\right)^{1/2}.
\end{align*}
Squaring and taking the \(L^2\)-norm with respect to
\((s,x)\in(-T_\star,T_\star)\times\mathcal O\), we get
\begin{align*}
    \|F_t^\eta\|_{L^2((-T_\star,T_\star)\times\mathcal O)}^{2}
&\leq
C\eta e^{T_\star^2\eta}
\int_{-T_\star}^{T_\star}
\int_{\mathcal O}\int_{-T}^{T}|F(\gamma,x)|^2\,d\gamma\,dx\,ds\\
&= 2T_{\star}C\eta e^{T_\star^2\eta}\|F\|^2_{L^2((-T,T)\times\mathcal O)} .
\end{align*}
Taking square roots and choosing any \(c_F>\frac{T_\star^2}{2}\), so that
\(\sqrt{2T_{\star}\eta}\,e^{\frac{T_\star^2}{2}\eta}\leq C e^{c_F\eta}\) for all
\(\eta\geq 1\), we conclude
\begin{equation}\label{eq:F-bound}
\|F_t^\eta\|_{L^2((-T_\star,T_\star)\times\mathcal O)}
\leq
C e^{c_F\eta}
\|F\|_{L^2((-T,T)\times\mathcal O)}.
\end{equation}
Therefore
\begin{equation}\label{eq:PVeta-bound}
\|\mathcal P_{a,q}V_{t}^\eta\|_{L^2((-T_\star,T_\star)\times\mathcal O)}
\leq
C e^{c_F\eta}\|F\|_{L^2((-T,T)\times\mathcal O)}
+
C e^{-c_R\eta}\|u\|_{H^1((-T,T)\times\mathcal O)}.
\end{equation}
Similarly, using $\partial_sV^\eta=\mathrm{i}\mathcal T^\eta(\psi'u)+\mathrm{i}\mathcal T^\eta\left(\psi\partial_\gamma u\right)$ and $\nabla_xV^\eta=\mathcal T^\eta(\psi\nabla_xu)$, we obtain
\begin{equation}\label{eq:Veta-global-bound}
\|V_{t}^\eta\|_{H^1((-T_\star,T_\star)\times\mathcal O)}
\leq
C e^{c_V\eta}\|u\|_{H^1((-T,T)\times\mathcal O)},
\end{equation}
for any   $c_{V}>T_\star^2/2$. 
Substituting \eqref{eq:PVeta-bound},  \eqref{eq:Veta-global-bound} into
\eqref{eq:interpolation-Veta}, and using
$(A+B)^\theta\leq A^\theta+B^\theta$, we obtain
\begin{align}\label{eq:interpolation-after-FBI-bounds}
\begin{aligned} 
&\|V_{t}^\eta\|_{H^1((-T',T')\times(\mathcal O_2\setminus\mathcal O_3))}
\notag\\
&\qquad\leq
C
\left(
e^{-\left(
\theta c_R-(1-\theta)c_V
\right)\eta}\|u\|_{H^1((-T,T)\times\mathcal O)}+
\left(
e^{c_F\eta}\mathcal D_t^\eta
\right)^\theta
\left(
e^{c_V\eta}
\|u\|_{H^1((-T,T)\times\mathcal O)}
\right)^{1-\theta}\right),
\end{aligned}
\end{align}
where $\mathcal D_t^\eta:= \|F\|_{L^2((-T,T)\times\mathcal O)}+ \left\|\partial_\nu V_t^\eta\right\|_{L^2((-T_\star,T_\star)\times\Gamma)} $. Further, we choose $\theta c_R>(1-\theta)c_V$, which can be achieved by the hypothesis of the Lemma, i.e, $T>T_{0}>\frac{4T_\star}{3\sqrt{\theta}}$. Finally, fix $\delta>\frac{c_V}{\theta}$ and apply the weighted Young
inequality
\[
    A^{\theta}B^{1-\theta}
    \leq
    \theta\lambda A+(1-\theta)\lambda^{-\frac{\theta}{1-\theta}}B,
    \qquad \text{ for } A,B\geq0\ \text{ and } \lambda>0,
\]
with $A=e^{c_F\eta}\mathcal D_t^{\eta}$,
$B=e^{c_V\eta}\|u\|_{H^{1}((-T,T)\times\mathcal O)}$ and
$\lambda=e^{\delta(1-\theta)\eta}$. Since
$\lambda^{-\frac{\theta}{1-\theta}}=e^{-\delta\theta\eta}$, this gives
\begin{equation}\label{eq:weighted-Young-FBI}
\begin{aligned}
&\bigl(e^{c_F\eta}\mathcal D_t^{\eta}\bigr)^{\theta}
\bigl(e^{c_V\eta}\|u\|_{H^{1}((-T,T)\times\mathcal O)}\bigr)^{1-\theta}
\\
&\qquad\leq
Ce^{\left(c_F+\delta(1-\theta)\right)\eta}\mathcal D_t^{\eta}
+
Ce^{-\left(\delta\theta-c_V\right)\eta}
\|u\|_{H^{1}((-T,T)\times\mathcal O)} .
\end{aligned}
\end{equation}
Define $\alpha_1:= \min\left\{\theta c_R-(1-\theta)c_V,\, \delta\theta-c_V\right\} >0$ and $\alpha_2:= c_F+\delta(1-\theta)>0$. Hence
\[
\|V_{t}^\eta\|_{H^1((-T',T')\times(\mathcal O_2\setminus\mathcal O_3))}
\leq
C\left(
e^{-\alpha_1\eta}\|u\|_{H^1((-T,T)\times\mathcal O)}
+
e^{\alpha_2\eta}
\left(
\|F\|_{L^2((-T,T)\times\mathcal O)}
+
\|\partial_\nu V_{t}^\eta\|_{L^2((-T_\star,T_\star)\times\Gamma)}
\right)
\right).
\]
This proves the desired estimate which completes the proof of Lemma.
\end{proof}
Next, in  the following Lemma, we  prove an estimate that involves the projection of  $V^{\eta}$ to the real-axis. For notational convenience, we denote the projection of $V^{\eta}$ to the real axis by $V^{\eta}_{0}$. More precisely, 
\begin{equation*}
    V^{\eta}_{0}(t,x):= V^{\eta}(t+\mathrm{i}0,x)= \sqrt{\frac{\eta}{2\pi}} \int_{\mathbb{R}} \mathfrak{K}^{\eta}(t-\gamma)\psi(\gamma)u(\gamma,x)\,d\gamma= \sqrt{\frac{\eta}{2\pi}}\left(\mathfrak{K}^{\eta}\ast_{t}(\psi u)\right)(t,x),
\end{equation*}
where $\ast_{t}$ represents the convolution in the time-variable.

\begin{lemma}\label{Lemma:Projection_V0} For $(a,q)\in\mathcal A$ and $F\in L^{2}(Q)$, let $u\in H^{2}(Q)$ be a solution of
\eqref{eq:wave-UCP-extended}   and  $T'\in \left(0, \min\left\{ \frac{T_\star}{3},\ \sqrt{\frac{\kappa}{\alpha}},\ \frac{T_0}{8} \right\}\right)$, $\alpha_1,\alpha_2>0$ and $\eta_0\geq1$ be  constants as in  Lemma
\ref{lemma:Veta-interpolation}. Also denote by $ \beta:=\alpha_2+\frac{T_\star^{2}}{2}$, then there exists $C>0$ such that the following estimate 
\begin{equation}\label{eq:projection-estimate}
\begin{aligned}
\|V_0^{\eta}\|_{H^{1}((-T',T')\times \left(\mathcal O_2\setminus\mathcal O_3)\right)}
\leq
C\Bigl(
e^{-\alpha_1\eta}\|u\|_{H^{1}((-T,T)\times\mathcal O)}
+
e^{\beta\eta}
\bigl(
\|F\|_{L^{2}((-T,T)\times\mathcal O)}
+
\|\partial_\nu u\|_{L^{2}((-T,T)\times\Gamma)}
\bigr)
\Bigr)
\end{aligned}
\end{equation}
for every $\eta\geq\eta_0$. 
\begin{proof}
Choose $r>0$ sufficiently small such that $0<r<\min\{T',\,T_\star-T'\}$. Then
\begin{equation}\label{choices-r}
(-r,r)\subset(-T',T'),\quad T'+r<2T'<\frac{T_0}{4}
    \quad\text{and}\quad
    (-T'-r,T'+r)\subset(-T_\star,T_\star).
\end{equation}
Since the function $V^{\eta}(\cdot,x) $ defined in \eqref{definition of V eta} is entire for every fixed $x\in\mathcal{O}$, the function $z\longmapsto |V^\eta(z,x)|^2$ is subharmonic. Therefore, using the mean value property for subharmonic functions, we arrive at 
\[
|V^\eta(\tau,x)|^2
\leq
\frac{C}{r^2}
\int_{D(\tau,r)}
|V^\eta(z,x)|^2\,\,ds\,dt,\ \text{ for }(\tau,x)\in (-T',T')\times \left(\mathcal O_2\setminus\mathcal O_3\right) ,
\]
where $D(\tau,r)\subset\mathbb C$ is an open disk of radius $r$ with center at $\tau\in (-T',T')$. Now since $D(\tau,r)\subset\left\{t+\mathrm{i}s:\ |t-\tau|<r,\ |s|<r\right\}$, therefore we get
\begin{equation}
|V^\eta(\tau,x)|^2
\leq
\frac{C}{r^2}
\int_{\tau-r}^{\tau+r}
\int_{-r}^{r}
|V^\eta(t+\mathrm{i}s,x)|^2\,ds\,dt .
\end{equation}
Integrating over $(\tau,x)\in (-T',T')\times \left(\mathcal O_2\setminus\mathcal O_3\right)$ and using the Fubini's theorem, we obtain
\begin{align}
\label{eq:L2-real-axis-expanded}
\begin{aligned} 
\int_{-T'}^{T'}\int_{\mathcal O_2\setminus\mathcal O_3} |V_0^\eta(\tau,x)|^2\,dx\,d\tau
&\leq
\frac{C}{r^2}
\int_{-T'}^{T'}\int_{\mathcal O_2\setminus\mathcal O_3}
\int_{\tau-r}^{\tau+r}
\int_{-r}^{r}
|V^\eta(t+\mathrm{i}s,x)|^2
\,ds\,dt\,dx\,d\tau\\
&\leq
\frac{C}{r^2}
\int_{\mathcal O_2\setminus\mathcal O_3}
\int_{-T'-r}^{T'+r}
\int_{-r}^{r}
|V^\eta(t+\mathrm{i}s,x)|^2
\left(
\int_{-T'}^{T'} \mathbf{1}_{\{|t-\tau|<r\}}\,d\tau
\right)
ds\,dt\,dx
\\
&\leq
\frac{C}{r^2}
\int_{\mathcal O_2\setminus\mathcal O_3}
\int_{-T'-r}^{T'+r}
\int_{-r}^{r}
|V^\eta(t+\mathrm{i}s,x)|^2
(2r)
\,ds\,dt\,dx
\\
&\leq
C\int_{-T'-r}^{T'+r}
\int_{-r}^{r}
\int_{\mathcal O_2\setminus\mathcal O_3}
|V^\eta(t+\mathrm{i}s,x)|^2\,dx\,ds\,dt,
\end{aligned} 
\end{align}
where in the last step the constant $C$ has absorbed the
factor $2/r$ as $r$ is fixed.
Next by utilizing  \eqref{choices-r}, we obtain that  $|t|<T_0/4$, for every
$t\in(-T'-r,T'+r)$, so Lemma \ref{lemma:Veta-interpolation}, gives
\begin{align}
\begin{aligned} 
&\int_{-r}^{r}
\int_{\mathcal O_2\setminus\mathcal O_3}
|V^\eta(t+\mathrm{i}s,x)|^2\,dx\,ds
\\
&\qquad\leq
C\left[
e^{-\alpha_1\eta}\|u\|_{H^1((-T,T)\times\mathcal O)}
+
e^{\alpha_2\eta}
\left(
\|F\|_{L^2((-T,T)\times\mathcal O)}
+
\|\partial_\nu V_t^\eta\|_{L^2((-T_{\star},T_{\star})\times\Gamma)}
\right)
\right]^2 ,
\end{aligned} 
\end{align}
for $\eta\geq \eta_0$. Using the elementary inequality $(a+b+c)^2\leq 3(a^2+b^2+c^2)$, we get
\begin{align}\label{eq:Veta-L2-O2-O3-time-integrated}
\begin{aligned}
&\int_{-T'-r}^{T'+r}
\int_{-r}^{r}
\int_{\mathcal O_2\setminus\mathcal O_3}
|V^\eta(t+\mathrm{i}s,x)|^2\,dx\,ds\,dt
\\
&\quad\leq
C\int_{-T'-r}^{T'+r}
\left(
e^{-2\alpha_1\eta}\|u\|_{H^1((-T,T)\times\mathcal O)}^2
+
e^{2\alpha_2\eta}\|F\|_{L^2((-T,T)\times\mathcal O)}^2
+
e^{2\alpha_2\eta}
\|\partial_\nu V_t^\eta\|_{L^2((-T_{\star},T_{\star})\times\Gamma)}^2
\right)dt. 
\end{aligned}
\end{align}
Since the first two terms do not depend on $t$, the integration in $t$
only gives the finite factor $2(T'+r)$, which is absorbed in 
constant $C$. Hence
\begin{align}
&\int_{-T'-r}^{T'+r}
\int_{-r}^{r}
\int_{\mathcal O_2\setminus\mathcal O_3}
|V^\eta(t+\mathrm{i}s,x)|^2\,dx\,ds\,dt
\notag\\
&\qquad\leq
C\left(
e^{-2\alpha_1\eta}\|u\|_{H^1((-T,T)\times\mathcal O)}^2
+
e^{2\alpha_2\eta}\|F\|_{L^2((-T,T)\times\mathcal O)}^2
\right)
\notag\\
&\qquad\quad
+
C e^{2\alpha_2\eta}
\int_{-T'-r}^{T'+r}
\int_{-T_{\star}}^{T_{\star}}
\int_{\Gamma}
|\partial_\nu V^\eta(t+\mathrm{i}s,y)|^2\,dS_y\,ds\,dt .
\end{align}
We now estimate the last term. Since the normal derivative acts only on
the spatial variable,
\[
\partial_\nu V^\eta(t+\mathrm{i}s,y)
=
\sqrt{\frac{\eta}{2\pi}}
\int_{\mathbb R}
\mathfrak K^\eta(t+\mathrm{i}s-\gamma)
\psi(\gamma)\partial_\nu u(\gamma,y)\,d\gamma .
\]
Taking the absolute value gives
\begin{align*}
  \left|\partial_\nu V^\eta(t+\mathrm{i}s,y)\right|
&\leq
e^{\frac{\eta s^2}{2}}
\sqrt{\frac{\eta}{2\pi}}
\int_{\mathbb R}
e^{-\frac{\eta}{2}(t-\gamma)^2}
\left|\psi(\gamma)\partial_\nu u(\gamma,y)\right|\,d\gamma = e^{\frac{\eta s^2}{2}}\sqrt{\frac{\eta}{2\pi}}
\bigl(\mathfrak{K}^{\eta} *_t |\psi\partial_\nu u| \bigr)(t,y),
\end{align*}
where the convolution is taken only in the time variable, and we have used that $|\mathfrak K^\eta(t+\mathrm i s-\gamma)|=e^{\eta s^2/2}e^{-\eta(t-\gamma)^2/2}$.
Now taking the \(L^2\)-norm over
\((-T'-r,T'+r)\times\Gamma\), and then enlarging the time interval to
the whole real line, we obtain
\[
\begin{aligned}
&\|\partial_\nu V^\eta(\cdot+\mathrm{i}s,\cdot)\|_
{L^2((-T'-r,T'+r)\times\Gamma)}
\leq
e^{\frac{\eta s^2}{2}}\sqrt{\frac{\eta}{2\pi}}
\left\|\mathfrak{K}^{\eta} * |\psi\partial_\nu u|\right\|_
{L^2(\mathbb R\times\Gamma)} .
\end{aligned}
\]
Next using the Young's convolution inequality in conjugation with $0\leq\psi\leq1$ and $\operatorname{supp}\psi\subset(-T,T)$, we get
\[
\|\partial_\nu V^\eta(\cdot+\mathrm{i}s,\cdot)\|_
{L^2((-T'-r,T'+r)\times\Gamma)}
\leq
C e^{\frac{\eta s^2}{2}}
\|\partial_\nu u\|_{L^2((-T,T)\times\Gamma)} ,
\]
where we have used $\|\mathfrak{K}^{\eta}\|_{L^1(\mathbb R)}=\sqrt{\frac{2\pi}{\eta}}.$
Finally, using the Fubini's theorem and enlarging the domain of  $t$-integration from $(-T'-r,T'+r)$ to
$\mathbb R$, we obtain
\begin{align}\label{eq:Bt-integrated}
\begin{aligned} 
&\int_{-T'-r}^{T'+r}
\int_{-T_{\star}}^{T_{\star}}
\int_{\Gamma}
|\partial_\nu V^\eta(t+\mathrm{i}s,y)|^2\,dS_y\,ds\,dt
\leq
\int_{\mathbb{R}}
\int_{-T_\star}^{T_\star}
\int_\Gamma
|\partial_\nu V^\eta(t+\mathrm{i}s,y)|^2
\,dS_y\,ds\,dt
\\
& \qquad \leq
\int_{-T_\star}^{T_\star}
\|\partial_\nu V^\eta(\cdot+\mathrm{i}s,\cdot)\|_
{L^2(\mathbb R\times\Gamma)}^2\,ds
\leq
\left(
\int_{-T_\star}^{T_\star}e^{\eta s^2}\,ds
\right)
\|\partial_\nu u\|_{L^2((-T,T)\times\Gamma)}^2
\\
& \qquad \leq
2T_\star e^{\eta T_\star^2}
\|\partial_\nu u\|_{L^2((-T,T)\times\Gamma)}^2.
\end{aligned}
\end{align}
After substituting the above estimates in \eqref{eq:Veta-L2-O2-O3-time-integrated}, and using $e^{2\alpha_2\eta}e^{\eta T_\star^2}=e^{2\beta\eta}$ together with $\beta\geq\alpha_2$, we get
\begin{align}
&\int_{-T'-r}^{T'+r}
\int_{-r}^{r}
\int_{\mathcal O_2\setminus\mathcal O_3}
|V^\eta(t+\mathrm{i}s,x)|^2\,dx\,ds\,dt
\notag\\
&\qquad\leq
C\left(
e^{-2\alpha_1\eta}\|u\|_{H^1((-T,T)\times\mathcal O)}^2
+
e^{2\beta\eta}\left(\|F\|_{L^2((-T,T)\times\mathcal O)}^2+\|\partial_\nu u\|_{L^2((-T,T)\times\Gamma)}^2\right)\right).
\end{align}
Finally, using the above estimate in \eqref{eq:L2-real-axis-expanded}, we get 
\begin{align}
&\|V_0^\eta\|_{L^2((-T',T')\times {\left(\mathcal O_2\setminus\mathcal O_3\right)})}^2\leq
    C\int_{-T'-r}^{T'+r}\int_{-r}^{r}\int_{\mathcal O_2\setminus\mathcal O_3}|V^\eta(t+\mathrm{i}s,x)|^2\,dx\,ds\,dt
\notag\\
&\qquad\qquad\leq
C\left(
e^{-2\alpha_1\eta}\|u\|_{H^1((-T,T)\times\mathcal O)}^2
+
e^{2\beta\eta}\left(\|F\|_{L^2((-T,T)\times\mathcal O)}^2
+
\|\partial_\nu u\|_{L^2((-T,T)\times\Gamma)}^2\right)
\right).
\end{align}
Now since $V^\eta$ is holomorphic in $z=t+\mathrm{i}s$, therefore the Cauchy Riemann equations  give  $$|\partial_t V^\eta(t+\mathrm{i}s,x)|=|\partial_s V^\eta(t+\mathrm{i}s,x)|.$$
Next repeating the similar argument related to use of mean value inequality to the holomorphic function
$\partial_tV^\eta(\cdot,x)$, we obtain
\[
\|\partial_t V_0^\eta\|_{L^2((-T',T')\times {\mathcal O_2\setminus\mathcal O_3})}^2
\leq
C
\int_{-T'-r}^{T'+r}
\int_{-r}^{r}
\int_{\mathcal O_2\setminus\mathcal O_3}
|\partial_s V^\eta(t+\mathrm{i}s,x)|^2
\,dx\,ds\,dt .
\]
Since the left-hand side of the estimate of Lemma
\ref{lemma:Veta-interpolation} also controls
$\int|\partial_sV_t^\eta|^2$, the previous analysis of this Lemma yields
\[
\|\partial_t V_0^\eta\|_{L^2((-T',T')\times {\left(\mathcal O_2\setminus\mathcal O_3\right)})}^2
\leq
C\left(
e^{-2\alpha_1\eta}\|u\|_{H^1((-T,T)\times\mathcal O)}^2
+
e^{2\beta\eta}\left(\|F\|_{L^2((-T,T)\times\mathcal O)}^2
+
\|\partial_\nu u\|_{L^2((-T,T)\times\Gamma)}^2\right)
\right)
\]
Similarly, for each spatial derivative, the function $z\longmapsto \partial_{x_j}V^\eta(z,x)$ is also entire. Hence, after repeating the previous arguments to 
$\partial_{x_j}V^\eta$, and summing over $j=1,\dots,d$, we get 
\begin{align*}
\|\nabla_x V_0^\eta\|_{L^2((-T',T')\times {\left(\mathcal O_2\setminus\mathcal O_3\right)})}^2
&\leq
C
\int_{-T'-r}^{T'+r}
\int_{-r}^{r}
\int_{\mathcal O_2\setminus\mathcal O_3}
|\nabla_x V^\eta(t+\mathrm{i}s,x)|^2
\,dx\,ds\,dt\\
&\leq C\left(
e^{-2\alpha_1\eta}\|u\|_{H^1((-T,T)\times\mathcal O)}^2
+
e^{2\beta\eta}\left(\|F\|_{L^2((-T,T)\times\mathcal O)}^2
+
\|\partial_\nu u\|_{L^2((-T,T)\times\Gamma)}^2\right)
\right)
.
\end{align*}
Now combining the previously derived  $L^2$ estimate of $V^{\eta}_{0}$ along with the estimates for its  time and spatial derivatives,  we get that 
\[
\|V_0^\eta\|_{H^1((-T',T')\times {\left(\mathcal O_2\setminus\mathcal O_3\right)})}^2
\leq
C\left(
e^{-2\alpha_1\eta}\|u\|_{H^1((-T,T)\times\mathcal O)}^2
+
e^{2\beta\eta}\left(\|F\|_{L^2((-T,T)\times\mathcal O)}^2
+
\|\partial_\nu u\|_{L^2((-T,T)\times\Gamma)}^2\right)
\right)
\]
Taking square roots and using $\sqrt{A+B}\leq\sqrt A+\sqrt B$ twice
yields \eqref{eq:projection-estimate}, which is what we wanted to prove. This completes the proof of Lemma. 
\end{proof}
\end{lemma}
Note that so far we have estimated the FBI transform of $\psi u$ and not the solution $u$ itself. Since $V_{0}^{\eta}$ is the convolution of $\psi u$ with a Gaussian kernel in the time variable, we estimate the solution $u$ by taking the Fourier transform in time variable. The precise statement is given by the following lemma, with the Fourier transform convention $\displaystyle \widehat g(\xi):=\int_{\mathbb R}e^{-\mathrm{i}t\xi}g(t)\,dt$, provided the integral exists. 
\begin{lemma}\label{lemma:u-versus-V0}
For  $(a,q)\in\mathcal A$, $F\in L^{2}(Q)$,  let $u\in H^{2}(Q)$ be the solution of
\eqref{eq:wave-UCP-extended} and let $T'\in \left(0, \min\left\{\frac{T_\star}{3},\ \sqrt{\frac{\kappa}{\alpha}} ,\ \frac{T_0}{8} \right\}\right)$. Also assume that  $\alpha_1>0$ and
$\beta=\alpha_2+\frac{T_\star^{2}}{2}$ be  constants as in Lemma
\ref{Lemma:Projection_V0}. Then there exist $C>0$ and $\eta_0\geq1$ such
that
\begin{equation}\label{eq:u-final-estimate}
\|u\|_{H^{1}\left((-T',T')\times(\mathcal O_2\setminus\mathcal O_3)\right)}
\leq
C\left(\frac{1}{\sqrt{\eta}}\|u\|_{H^{2}(Q)} +
e^{\beta\eta} \Bigl(
\|F\|_{L^{2}((0,T)\times\mathcal O)}+
\|\partial_\nu u\|_{L^{2}(\Gamma_T)}
\Bigr) \right).
\end{equation}
holds for every $\eta\geq\eta_0$. 
\end{lemma}

\begin{proof}
Let $f(t,x):=\psi(t)u(t,x)$ for $(t,x)\in\mathbb{R}\times\Omega$, where $f$ is extended by zero outside $(-T,T)$ in the time variable. For each fixed $x\in\Omega$, we take the Fourier transform with respect to the variable $t$. Since $\widehat{\mathfrak K^{\eta}}(\xi)=\sqrt{\frac{2\pi}{\eta}}\exp\left(-\frac{\xi^2}{2\eta}\right)$ and $V_0^\eta=\sqrt{\frac{\eta}{2\pi}}\,\mathfrak K^\eta\ast_t f$, the two normalizing factors cancel, and it follows that $\widehat{V_0^\eta}(\xi,x)=\exp\left(-\frac{\xi^2}{2\eta}\right)\widehat f(\xi,x)$. Consequently,
$\widehat f(\xi,x)-\widehat{V_0^\eta}(\xi,x)=\left(1-\exp\left(-\frac{\xi^2}{2\eta}\right)\right)\widehat f(\xi,x)$.
Now observe that  $0\leq1-e^{-a}\leq\min\{a,1\}$ and $\min\{a,1\}^2\leq a$, for any  $a\geq0$, therefore, choosing $a=\frac{\xi^2}{2\eta}$, we obtain
$\left|1-\exp\left(-\frac{\xi^2}{2\eta}\right)\right|^2\leq\frac{\xi^2}{2\eta}$. Integrating with respect to $\xi$ and using the Plancherel theorem, we obtain
\begin{equation}
\begin{aligned}
\left\|f(\cdot,x)-V_0^\eta(\cdot,x)\right\|_{L^2(\mathbb R)}^2
\leq
\frac{C}{\eta}
\|\xi\widehat f(\cdot,x)\|_{L^2(\mathbb R_\xi)}^2
=
\frac{C}{\eta}
\|\partial_t f(\cdot,x)\|_{L^2(\mathbb R)}^2.
\end{aligned}
\end{equation}
Integrating the above estimate with respect to $x\in\Omega$ gives
\begin{equation}\label{eq:smoothing}
\|f-V_0^\eta\|_{L^2(\mathbb R\times\Omega)}^2
\leq
\frac{C}{\eta}
\|\partial_t f\|_{L^2(\mathbb R\times\Omega)}^2.
\end{equation}
Using the given  hypothesis $T'<\frac{T_0}{8} <T_0$ and $\psi\equiv 1$, on $[-T_0,T_0]$, we have
\begin{equation}
f(t,x)=u(t,x)
\ \ 
\text{for every }(t,x)\in(-T',T')\times
\left(\mathcal O_2\setminus\mathcal O_3\right).
\end{equation}
Restricting the preceding estimate to $(-T',T')\times
\left(\mathcal O_2\setminus\mathcal O_3\right)$, we find
\begin{equation}
\|u-V_0^\eta\|_{L^2((-T',T')\times
\left(\mathcal O_2\setminus\mathcal O_3\right))}^2
\leq
\frac{C}{\eta}
\|\partial_tf\|_{L^2(\mathbb R\times\Omega)}^2.
\end{equation}
We next estimate the first-order derivatives of $u-V_0^\eta$. Since the convolution defining $V_0^\eta$ acts only in the time variable, it commutes with spatial differentiation. Thus, for every $j=1,\ldots,n$, we have $\partial_{x_j}V_0^\eta= \sqrt{\frac{\eta}{2\pi}}\mathfrak K^{\eta}\ast_t\partial_{x_j}f$.
Applying \eqref{eq:smoothing} to
$\partial_{x_j}f$, we obtain
\begin{equation}
\begin{aligned}
\|
\partial_{x_j}f
-
\partial_{x_j}V_0^\eta
\|_{L^2(\mathbb R\times\Omega)}^2
&\leq
\frac{C}{\eta}
\|
\partial_t\partial_{x_j}f
\|_{L^2(\mathbb R\times\Omega)}^2.
\end{aligned}
\end{equation}
Hence,
\begin{equation}
\|
\nabla_xf-\nabla_xV_0^\eta
\|_{L^2(\mathbb R\times\Omega)}^2
\leq
\frac{C}{\eta}
\|
\partial_t\nabla_xf
\|_{L^2(\mathbb R\times\Omega)}^2.
\end{equation}
Similarly, differentiation under the convolution gives $\partial_tV_0^\eta= \sqrt{\frac{\eta}{2\pi}}\mathfrak K^{\eta}\ast_t\partial_tf$. Applying \eqref{eq:smoothing} to $\partial_tf$, we obtain
\begin{equation}
\|
\partial_tf-\partial_tV_0^\eta
\|_{{L^2(\mathbb R\times\Omega)}}^2
\leq
\frac{C}{\eta}
\|
\partial_t^2f
\|_{L^2(\mathbb R\times\Omega)}^2.
\end{equation}
Combining the estimates for the function, its time derivative, and its spatial derivatives yields
\begin{equation}
\|f-V_0^\eta\|_{H^1(\mathbb R\times\Omega)}^2
\leq
\frac{C}{\eta}
\|f\|_{H^2(\mathbb R\times\Omega)}^2.
\end{equation}
Restricting the left-hand side to $(-T',T')\times \left(\mathcal O_2\setminus\mathcal O_3\right)$ and using $f=u$ on $(-T',T')\times \left(\mathcal O_2\setminus\mathcal O_3\right)$, we obtain
\begin{equation}
\|u-V_0^\eta\|_{H^1((-T',T')\times \left(\mathcal O_2\setminus\mathcal O_3\right))}^2
\leq
\frac{C}{\eta}
\|\psi u\|_{H^2(\mathbb R\times\Omega)}^2.
\end{equation}

Since $\psi$ is a fixed smooth cutoff function, with $\operatorname{supp}\psi\subset(-T,T)$, multiplication by $\psi$ defines a bounded linear operator on $H^2(\mathbb{R}\times \Omega)$, therefore 
\begin{equation}
\|\psi u\|_{H^2(\mathbb R\times\Omega)}
\leq
C
\|u\|_{H^2((-T,T)\times\Omega)}.
\end{equation}
Consequently,
\begin{equation}\label{est:u-v0}
\|u-V_0^\eta\|_{H^1((-T',T')\times \left(\mathcal O_2\setminus\mathcal O_3\right))}^2
\leq
\frac{C}{\eta}
\|u\|_{H^2((-T,T)\times\Omega)}^2.
\end{equation}
Using the  inequality $|A+B|^2\leq2|A|^2+2|B|^2$, we obtain
\begin{equation}
\begin{aligned}
\|u\|_{H^1((-T',T')\times \left(\mathcal O_2\setminus\mathcal O_3\right))}^2
&\leq
2\|u-V_0^\eta\|_{H^1((-T',T')\times \left(\mathcal O_2\setminus\mathcal O_3\right))}^2
+
2\|V_0^\eta\|_{H^1((-T',T')\times \left(\mathcal O_2\setminus\mathcal O_3\right))}^2.
\end{aligned}
\end{equation}
Using the above estimate along with \eqref{est:u-v0}, it follows that
\begin{equation}
\|u\|_{H^1((-T',T')\times \left(\mathcal O_2\setminus\mathcal O_3\right))}^2
\leq
\frac{C}{\eta}
\|u\|_{H^2((-T,T)\times\Omega)}^2
+
C\|V_0^\eta\|_{H^1((-T',T')\times \left(\mathcal O_2\setminus\mathcal O_3\right))}^2.
\label{eq:u-by-V0}
\end{equation}
Now using the estimate derived in Lemma \ref{Lemma:Projection_V0}, and
enlarging the domain of the $H^1$-norm on its right-hand side from
$\mathcal O$ to $\Omega$, we get 
\begin{equation}\label{eq:previous-V0-estimate}
\begin{aligned}
\|V_0^\eta\|_{H^1((-T',T')\times \left(\mathcal O_2\setminus\mathcal O_3\right))}^2
\leq
C\Bigg(
&e^{-2\alpha_1\eta}
\|u\|_{H^1((-T,T)\times\Omega)}^2
\\
&+
e^{2\beta\eta}
\left(
\|F\|_{L^2((-T,T)\times\mathcal O)}^2
+
\|\partial_\nu u\|_{L^2((-T,T)\times\Gamma)}^2
\right)
\Bigg).
\end{aligned}
\end{equation}
Finally, substituting  \eqref{eq:previous-V0-estimate} into
\eqref{eq:u-by-V0}, we obtain
\begin{equation}
\begin{aligned}
\|u\|_{H^1((-T',T')\times \left(\mathcal O_2\setminus\mathcal O_3\right))}^2
\leq
C\Bigg(
&\frac{1}{\eta}
\|u\|_{H^2((-T,T)\times\Omega)}^2
+
e^{-2\alpha_1\eta}
\|u\|_{H^1((-T,T)\times\Omega)}^2
\\
&+
e^{2\beta\eta}
\left(
\|F\|_{L^2((-T,T)\times\mathcal O)}^2
+
\|\partial_\nu u\|_{L^2((-T,T)\times\Gamma)}^2
\right)
\Bigg).
\end{aligned}
\end{equation}
Now since  $\alpha_1>0$, therefore we can choose $\eta_0\geq1$ such that $e^{-2\alpha_1\eta}
\leq
\frac{C}{\eta}$ for every $\eta\geq\eta_0$ and hence the  exponentially decaying term  can be merged  in the first term on the right-hand side in conjugation with $\| \cdot\|_{H^1((-T,T)\times\Omega)}\leq \| \cdot\|_{ H^2((-T,T)\times\Omega)}$.  Thus, we end up with 
\begin{equation}
\begin{aligned}
\|u\|_{H^1((-T',T')\times \left(\mathcal O_2\setminus\mathcal O_3\right))}^2
\leq
C\Bigg(
&\frac{1}{\eta}
\|u\|_{H^2((-T,T)\times\Omega)}^2
\\
&+
e^{2\beta\eta}
\left(
\|F\|_{L^2((-T,T)\times\mathcal O)}^2
+
\|\partial_\nu u\|_{L^2((-T,T)\times\Gamma)}^2
\right)
\Bigg),\ \mbox{for all  $\eta\geq \eta_0$.}
\end{aligned}
\end{equation}
Finally, taking square roots in the preceding inequality and using
$\sqrt{A+B+D}\leq\sqrt A+\sqrt B+\sqrt D$ for $A,B,D\geq0$, we obtain
\[
\|u\|_{H^1((-T',T')\times \left(\mathcal O_2\setminus\mathcal O_3\right))}
\leq
C\left(
\frac{1}{\sqrt\eta}\|u\|_{H^2((-T,T)\times\Omega)}
+
e^{\beta\eta}
\left(
\|F\|_{L^2((-T,T)\times\mathcal O)}
+
\|\partial_\nu u\|_{L^2((-T,T)\times\Gamma)}
\right)
\right),
\]
which gives \eqref{eq:u-final-estimate} and completes the proof.
\
\end{proof}
Note that the Lemma \ref{lemma:u-versus-V0} proves the estimate \eqref{carleman for ucp} of Theorem \ref{Carleman type estimate required for ucp} on the short time interval $(-T',T')$. We will use an iterative argument to complete the proof  of  Theorem \ref{Carleman type estimate required for ucp}. We start by defining  $\displaystyle T_\sharp:=\sup\bigl\{|t|:\ t\in\operatorname{supp}\psi\bigr\}$, where $\psi$ is the cutoff function defined in the proof of Lemma \ref{lemma:Veta-interpolation}. Since $\operatorname{supp}\psi$ is a compact subset of
$(-T,T)$ containing $[-T_{0},T_{0}]$, we have $T_{0}\leq T_\sharp<T$. In the subsequent Lemma, we will show that the same estimate holds around every center point 
$t_0$ with $|t_0|\leq T-T_\sharp$. To be precise, we prove the following. 
\begin{lemma}\label{Translated_local_estimate}
For  $(a,q)\in\mathcal A$ and $F\in L^{2}(Q)$, let $u\in H^{2}(Q)$ be a solution of
\eqref{eq:wave-UCP-extended}  and let $T'\in \left(0, \min\left\{\frac{T_\star}{3},\ \sqrt{\frac{\kappa}{\alpha}} ,\ \frac{T_0}{8} \right\}\right)$. Also assume that  $\alpha_1>0$ and
$\beta=\alpha_2+\frac{T_\star^{2}}{2}$ be constants as in Lemma
\ref{Lemma:Projection_V0}. Then there exist $C>0$ and $\eta_{0}\geq1$, both independent of $t_0$, such that for every
$\displaystyle t_{0}\in\left[-(T-T_\sharp),\,T-T_\sharp\right]$ the following estimate 
\begin{equation}\label{eq:translated-estimate}
    \|u\|_{H^{1}\left((t_{0}-T',\,t_{0}+T')\times
    (\mathcal O_{2}\setminus\mathcal O_{3})\right)}
    \leq
    C\left(
    \frac{1}{\sqrt\eta}\,\|u\|_{H^{2}(Q)}
    +
    e^{\beta\eta}
    \Bigl(
    \|F\|_{L^{2}((0,T)\times\mathcal O)}
    +
    \|\partial_\nu u\|_{L^{2}(\Gamma_T)}
    \Bigr)
    \right)
\end{equation} 
holds for every $\eta\geq\eta_{0}$.
\end{lemma}
\begin{proof}
    Fix $\displaystyle t_{0}\in\bigl[-(T-T_\sharp),\ T-T_\sharp\bigr]\subseteq (-T,T)$ and define the translated cutoff function $\displaystyle \psi_{t_{0}}(\gamma):=\psi(\gamma-t_{0})$ for $\gamma\in\R$. From the definition of $T_{\sharp}$, we have $\displaystyle \operatorname{supp}\psi\subseteq[-T_{\sharp},T_{\sharp}]$ and hence $\displaystyle\operatorname{supp}\psi_{t_{0}}\subseteq[t_{0}-T_{\sharp},\,t_{0}+T_{\sharp}]\subseteq[-T,T]$ where the last inclusion uses the fact that $|t_{0}|\leq T-T_\sharp$. Extending the function $\psi_{t_{0}}u$ by zero to the domain $\mathbb{R}\times\Omega$, we have $\psi_{t_{0}}u\in H^{2}(\R\times\Omega)$. Moreover, we have
    \begin{equation}\label{eq:shifted-cutoff-norm}
        \begin{aligned}
            \|\psi_{t_{0}}u\|_{H^{2}(\mathbb{R}\times\Omega)}
    \ \leq\ C\,\|u\|_{H^{2}((-T,T)\times\Omega)}
    \ =\ C\,\|u\|_{H^{2}(Q)},
        \end{aligned}
    \end{equation}
where constant $C$ is independent of $t_{0}$ and the last equality holds because the extension in \eqref{eq:wave-UCP-extended} vanishes for $t\leq0$. For $x\in \overline{\Omega}$, let
\begin{equation}\label{eq:shifted-fbi}
  V^{\eta,t_{0}}(z,x)
    :=\sqrt{\frac{\eta}{2\pi}}\int_{\mathbb{R}}\mathfrak{K}^{\eta}(z-\gamma)\,\psi_{t_{0}}(\gamma)\,u(\gamma,x)\,d\gamma, \ \ \text{ for }\ z\in \mathbb{C},  
\end{equation}
which is the transform as defined in \eqref{definition of V eta} with $\psi$ replaced by $\psi_{t_{0}}$ and which is entire in $z$ for $x\in\Omega$. Writing $z=t+\mathrm{i}s$, we fix $t\in\left(t_{0}-\frac{T_{0}}{4},\,t_{0}+\frac{T_{0}}{4}\right)$ and set
\[
    V^{\eta,t_{0}}_{t}(s,x):=V^{\eta,t_{0}}(t+\mathrm{i} s,x)
    \ \ \text{ for } (s,x)\in(-T_{\star},T_{\star})\times\mathcal{O}.
\]
The Carleman estimate of Lemma \ref{lemma:Carleman_estimate} applies to $V^{\eta,t_{0}}_{t}$ on
$(-T_{\star},T_{\star})\times\mathcal{O}$, in the variable $s\in(-T_{\star},T_{\star})$, where
$T_{\star}>0$ is the parameter fixed in Lemma \ref{lemma:Carleman_estimate}, subject to
$\alpha T_{\star}^{2}<\alpha_{0}$. We then project on $s=0$, that is, we take
$V^{\eta,t_{0}}_{t}(0,x)=V^{\eta,t_{0}}(t,x)$ for $t\in(t_{0}-T',\,t_{0}+T')$ and
$x\in\mathcal{O}_{2}\setminus\mathcal{O}_{3}$, as in Lemma \ref{Lemma:Projection_V0}.  Taking
$t_{0}=0$ gives the functions of Lemmas
\ref{lemma:Veta-interpolation}---\ref{lemma:u-versus-V0}. Therefore, the proof of Lemmas \ref{lemma:Veta-interpolation}---\ref{lemma:u-versus-V0} go through with $\psi_{t_0}$ instead of $\psi$, since the only properties of the cutoff they
use are $0\leq\psi_{t_{0}}\leq1$, $\psi_{t_{0}}=1$ on $[t_{0}-T_{0},\,t_{0}+T_{0}]$, the bound on
$\|\psi_{t_{0}}\|_{H^{2}(\mathbb{R})}$ and
\[
    \operatorname{supp}\psi_{t_{0}}'\cup\operatorname{supp}\psi_{t_{0}}''
    \subseteq\bigl\{\gamma\in\mathbb{R}:\ T_{0}\leq|\gamma-t_{0}|\leq T_{\sharp}\bigr\},
\]
none of which is affected by the translation. In particular, for $t$ as above and for $\gamma$ in
this last set we have $|t-\gamma|\geq|\gamma-t_{0}|-|t-t_{0}|\geq T_{0}-\frac{T_{0}}{4}
=\frac{3T_{0}}{4}$, so that $\bigl|\mathfrak{K}^{\eta}(t+\mathrm{i} s-\gamma)\bigr|\leq e^{-c_{0}\eta}$ for
$|s|\leq T_{\star}$ with the same constant
$c_{0}=\frac{1}{2}\left[\left(\frac{3T_{0}}{4}\right)^{2}-T_{\star}^{2}\right]$ as before.  Hence the
constants $C$, $\alpha_{1}$, $\alpha_{2}$, $\beta$ and $\eta_{0}$ are those of Lemmas
\ref{lemma:Veta-interpolation}-\ref{lemma:u-versus-V0} and are independent of $t_{0}$, and Lemma
\ref{lemma:u-versus-V0} yields, for every $\eta\geq\eta_{0}$,
\begin{equation}\label{eq:shifted-local}
\begin{aligned}
\|\psi_{t_{0}}u\|_{H^{1}\left((t_{0}-T',\,t_{0}+T')\times
    (\mathcal O_{2}\setminus\mathcal O_{3})\right)}
    \leq
    C\Biggl(
    &\frac{1}{\sqrt\eta}\,
    \|\psi_{t_{0}}u\|_{H^{2}(\mathbb{R}\times\Omega)}
    \\
    &+
    e^{\beta\eta}
    \Bigl(
    \|F\|_{L^{2}((-T,T)\times\mathcal O)}
    +
    \|\partial_\nu u\|_{L^{2}((-T,T)\times\Gamma)}
    \Bigr)
    \Biggr).
\end{aligned}
\end{equation}
It remains to convert the above estimate into an estimate for $u$. Since $T'<\frac{T_{0}}{8}<T_{0}$, we have
$\psi_{t_{0}}=1$ on $(t_{0}-T',\,t_{0}+T')$, so that $\psi_{t_{0}}u=u$ on
$(t_{0}-T',\,t_{0}+T')\times\Omega$ and the left hand side of \eqref{eq:shifted-local} is the left hand
side of \eqref{eq:translated-estimate}. On the right hand side, the first term is bounded by
$C\|u\|_{H^{2}(Q)}$ due to \eqref{eq:shifted-cutoff-norm} and for the remaining two terms, we use that the extensions vanish for $t\leq0$, which proves the estimate \eqref{eq:translated-estimate}  and completes the proof of the lemma.
\end{proof}
\subsection{Proof of Theorem \ref{Carleman type estimate required for ucp}} This subsection is devoted to proving the 
 main result of this section, stated in  Theorem \ref{Carleman type estimate required for ucp}. For every $\displaystyle t\in[0,T-T_{\sharp}]$ we have $t\in(t-T',\,t+T')$, and $\displaystyle \bigl\{(t_{0}-T',\,t_{0}+T')\bigr\}_{t_{0}\in[0,T-T_{\sharp}]}$ is a cover of the compact set $\displaystyle [0,T-T_{\sharp}]$ by open intervals. By compactness, there exists $N\in\mathbb N$ and $\displaystyle t_{1},\dots,t_{N}\in[0,T-T_{\sharp}]$
such that $\displaystyle [0,T-T_{\sharp}]\times \left(\mathcal{O}_{2}\setminus\mathcal{O}_{3}\right) \ \subseteq\
    \bigcup_{k=1}^{N}\bigl((t_{k}-T',\,t_{k}+T')\times \left(\mathcal{O}_{2}\setminus\mathcal{O}_{3}\right)\bigr)$. Without loss of generality, we can assume that $\displaystyle (0,T-T_{\sharp}+T')\times
    \left(\mathcal{O}_{2}\setminus\mathcal{O}_{3}\right)
    \ \subseteq\
    \bigcup_{k=1}^{N}\bigl((t_{k}-T',\,t_{k}+T')\times
    \left(\mathcal{O}_{2}\setminus\mathcal{O}_{3}\right)\bigr)$.
    Now, applying Lemma \ref{Translated_local_estimate}, we get for every $\eta\geq\eta_{0}$
    \begin{equation}
    \begin{aligned}
    \|u\|_{H^{1}\left((0,T-T_{\sharp}+T')\times \left(\mathcal{O}_{2}\setminus\mathcal{O}_{3}\right)\right)}^{2}
    &\leq\
    \sum_{k=1}^{N}
    \|u\|_{H^{1}\left((t_{k}-T',\,t_{k}+T')\times \left(\mathcal{O}_{2}\setminus\mathcal{O}_{3}\right)\right)}^{2} \\
    &\leq C\sum_{k=1}^{N}\left(\frac{1}{\sqrt\eta}\,\|u\|_{H^{2}(Q)}+ e^{\beta\eta}\Bigl(\|F\|_{L^{2}((0,T)\times\mathcal O)}+
    \|\partial_\nu u\|_{L^{2}(\Gamma_T)}\Bigr)\right)^2\\
    &\leq N C\,\left(\frac{1}{\sqrt\eta}\,\|u\|_{H^{2}(Q)}+ e^{\beta\eta}\Bigl(\|F\|_{L^{2}((0,T)\times\mathcal O)}+
    \|\partial_\nu u\|_{L^{2}(\Gamma_T)}\Bigr)\right)^2
    \end{aligned}
\end{equation}
Taking the square root of both sides it gives
 \begin{equation}\label{eq:sum-over-cover}
    \begin{aligned}
    \|u\|_{H^{1}\left((0,T-T_{\sharp}+T')\times \left(\mathcal{O}_{2}\setminus\mathcal{O}_{3}\right)\right)}
    &\leq C\,\left(\frac{1}{\sqrt\eta}\,\|u\|_{H^{2}(Q)}+ e^{\beta\eta}\Bigl(\|F\|_{L^{2}((0,T)\times\mathcal O)}+
    \|\partial_\nu u\|_{L^{2}(\Gamma_T)}\Bigr)\right)
    \end{aligned}
\end{equation}
for every $\eta\geq\eta_{0}$. Let $\displaystyle T_{\flat}:= T-T_{\sharp}+T'<T$ and also observe that $\displaystyle T_{\sharp}-T'\ >\ T_{0}-\frac{T_{0}}{8}\ =\ \frac{7T_{0}}{8}$. Now, assuming $T_0$ to be arbitrarily small such that $T_{\flat}\in (0,T)$ and $T_{\flat}$ are arbitrarily close to $T$. Then, the above estimates provide the proof of Theorem \ref{Carleman type estimate required for ucp}.
    \section{Proof of Theorem \ref{thm:main_result_partial_data}}\label{Section: proof of main result}
 We partition this section into four subsections. The first and fourth subsections address, respectively, the first-order and higher-order linearization of the given operator, a standard methodology in the analysis of nonlinear operators; see, for example, \cite{RuYuLaiEtAlPartialDataStabilityProblem, Kian2020PartialDI,
Choulli2021, HarrachLin2023Simultaneous,
LassasOksanenSahooSaloTetlow2025,KumarNakamuraVashisth2026}. In the second and third subsection, we prove unique identification of the coefficients $a$ and $q$, which is subsequently used along with higher order linearization established in fourth subsection, to prove  unique recovery of the nonlinear coefficient $r$ in Subsection \ref{Recovery of r}.
\subsection{First-order linearization}
Following \cite{bhardwaj2026reconstructionpotentialdampingcoefficients}, we  derive a first-order linearization of Equation \eqref{equation; IBVP}, which plays a crucial role in establishing the unique identification of the damping coefficient $a$ and the linear potential $q$. For a given set of Dirichlet data  $f_1,f_2,\dots,f_{\ell} \in H^{m+1}(\Sigma)$, we choose $\epsilon = (\epsilon_1,\dots,\epsilon_{\ell})$ such that 
\[\epsilon f:= \epsilon_1 f_1 +\dots + \epsilon_{\ell} f_{\ell} \in \mathcal{F}^\varrho_{m+1},\] 
where  $\epsilon_j\geq 0$ for each $1\leq j\leq \ell$.
Let $u(t,x):=u_{\epsilon f}(t,x)\in \mathscr{E}_{m+1}$ 
be the solution to following IBVP
\begin{align}\label{eq;Linearization}
\begin{cases}
\Box u(t,x) + a(x)\partial_t u(t,x) + q(x)u(t,x) =- r(x)u^{\ell}(t,x), 
&(t,x)\in  Q,\\
u(t,x)=\epsilon f(t,x), &(t,x)\in \Sigma,\\
u(0,x)=0,\quad \partial_t u(0,x)=0 ,& \quad x\in  \Omega. 
\end{cases}
\end{align}
We  differentiate the aforementioned IBVP with respect to $\epsilon_i\ (1\leq i\leq \ell)$, to arrive at
\begin{align}\label{eq;first order Linearization eqn}
\begin{cases}
\Box (\partial_{\epsilon_i}u)
+ a(x)\partial_t(\partial_{\epsilon_i}u)
+ q(x)\,\partial_{\epsilon_i}u
=- \ell\,r(x)\,u^{\ell-1}\,\partial_{\epsilon_i}u
,&
 (t,x)\in  Q,\\
\partial_{\epsilon_i}u(t,x) = f_i (t,x),  &(t,x)\in  \Sigma,\\
\partial_{\epsilon_i}u(0,x)=0,\quad\partial_t(\partial_{\epsilon_i}u)(0,x)=0,
&\quad x\in  \Omega.
\end{cases}
\end{align}
On evaluating the above expression at $\epsilon=0$, we achieve 
 \begin{align}\label{eq;Notations of first order Linearization eqn}
\begin{cases}
\Box v_{i}(t,x) + a(x)\partial_t v_{i}(t,x) + q(x)v_{i}(t,x) = 0,
& (t,x)\in  Q,\\
v_{i}(t,x) = f_i(t,x), & (t,x)\in  \Sigma,\\
v_{i}(0,x)=0,\quad \partial_t v_{i}(0,x)=0, 
&\quad x\in  \Omega,
\end{cases}
\end{align}
where 
\begin{equation}\label{eq:def-vk}
v_{i}(t,x):=\partial_{\epsilon_i}u(t,x)\big|_{\epsilon=0}, \ \ 1\leq i\leq \ell.
\end{equation}
 In the last step, we used the fact that the forward problem admit a unique solution and thus $u\lvert_{\epsilon=0}=0$. The DtN map corresponding to IBVP \eqref{eq;Linearization} is given by 
 \begin{align}\label{DN map for epsilon f}
\begin{aligned}
\mathcal{N}_{a,q,r}(\epsilon f):=\partial_{\nu} u \big|_{\Sigma},\ \ \epsilon f\in \mathcal{F}^\varrho_{m+1}.
\end{aligned} 
\end{align}
The first-order linearization of the above DtN map is given by 
\begin{align}
 \left.\partial_{\epsilon_i}\partial_\nu u(t,x,\epsilon_1,\dots,\epsilon_{\ell})\right|_{\Sigma,\, \epsilon=0}=\partial_\nu v_i(t,x)|_{\Sigma},\quad   1\leq i\leq \ell.
\end{align}
As a consequence, the DtN map
corresponds to Equation
\eqref{eq;Notations of first order Linearization eqn} is a first-order linearization of DtN map \eqref{DN map for epsilon f}, and is given by
\begin{align}\label{eq:DN-map-v_k}
\mathcal{N}_{a,q}^{~v_i} (f_i)\big|_{\Sigma}:=
\left(\partial_{\epsilon_i}
\mathcal{N}_{a,q,r}(\epsilon_1 f_1 +\dots+ \epsilon_{\ell} f_{\ell})
\big|_{\Sigma}\right)\big|_{\epsilon=0}
=
\partial_{\epsilon_i}
\partial_\nu u
\big|_{\Sigma, \epsilon=0}
=
\partial_\nu v_i \big|_{\Sigma}, 
\end{align}
for all $\ 1\leq i\leq \ell.$ The restriction of the above expression to the subset $\Gamma_T$ than reduces to the partial DtN map corresponding to Equation
\eqref{eq;Notations of first order Linearization eqn} and is given by
\begin{align}\label{dn for vi}
  \Lambda_{a,q}^{v_i}(f_i):=  \left(\mathcal{N}_{a,q}^{~v_i} (f_i)\big|_{\Sigma}\right)\bigg|_{\Gamma_T}= \partial_\nu v_i|_{\Gamma_T}.
\end{align}
In the analysis of the unique recovery of the damping coefficient \(a\) and the zeroth-order coefficient $q$, we use  the partial DtN map $\Lambda^{v_i}_{a,q}$ corresponding to  the linearized equation  \eqref{eq;Notations of first order Linearization eqn}. Without loss of generality, we set \(v_i := v\) and \(f_i := \mathtt{f}\). Under this convention, we obtain the following representation
\begin{align}\label{eq;Notations of first order Linearization eqn 1}
\begin{cases}
\Box v(t,x) + a(x)\partial_t v(t,x) + q(x)v(t,x) = 0,
& (t,x)\in  Q,\\
v(t,x) = \mathtt{f}(t,x), & (t,x)\in  \Sigma,\\
v(0,x)=0,\quad \partial_t v(0,x)=0, 
&\quad x\in  \Omega,
\end{cases}
\end{align}
and the partial DtN map corresponding to \eqref{eq;Notations of first order Linearization eqn 1} becomes
\begin{align}\label{partial dn map for linearized pde}
\Lambda_{a,q}^v(\mathtt{f})=\mathcal{N}_{a,q}^{~v} (\mathtt{f})\big|_{\Gamma_T}
=
\partial_\nu v \big|_{\Gamma_T},\ \mathtt{f} \in H^{m+1}(\Sigma)
\end{align}
and it is known for all $\mathtt{f} \in H^{m+1}(\Sigma)$. 
Next, following \cite{Isakov1991AnIH}, we present the so‑called geometric optics solution, a particular class of special solutions associated with Equations \eqref{eq;Notations of first order Linearization eqn 1} and its formal $L^2-$adjoint. 
\begin{lemma}{\cite[Lemma~~2]{Isakov1991AnIH}}\label{lemma: geom opt sol v_{1}}
Let $T>0$, and let $\Omega \subset \mathbb{R}^d$ $(d\geq 2)$ be an open, bounded, and connected set with smooth boundary $\partial\Omega$. Consider any triple $(a,q,\Phi)\in C_{c}^{\infty}(\Omega)\times C_{c}^{\infty}(\Omega)\times C_{c}^{\infty}(\mathbb{R}^{d})$ such that $\supp(\Phi)\cap\Omega=\emptyset$, and $\omega \in \mathbb{S}^{d-1}$. Then, for every $x \in \mathbb{R}^d$ and $\tau>0$, Equation \eqref{eq;Notations of first order Linearization eqn 1} admits a solution of the form 
\begin{align}\label{eq: v_{(1)}}
\begin{split}
 v(t,x) = \Phi(x+t\omega)T_v(t,x)\exp\bigl(\mathrm{i}\tau(x\cdot\omega + t)\bigr) + R_{v}(t,x,\tau),
\end{split}
\end{align}
where 
\begin{align*}
    T_v(t,x)=\exp\left(-\frac{1}{2}\int_0^t a(x+s\omega)\,ds\right),
\end{align*}
and the correction term $R_{v}(t,x,\tau)$ vanishes at the initial time, that is,
\begin{align}\label{R_{1} initia}
   R_{v}(0,x,\tau) = \partial_{t}R_{v}(0,x,\tau) = 0 \quad \text{in } \Omega.
\end{align}
Moreover, $R_{v}(t,x,\tau)$ satisfies the estimate
\begin{align}\label{eq: r_{(1)}}
  \tau\|R_{v}\|_{L^{2}(Q)}+\|\partial_{t} R_{v}\|_{L^{2}(Q)} \leq C \|\Phi\|_{H^{3}(\mathbb{R}^d)}.
\end{align}
\end{lemma}

\begin{lemma}{\cite[Lemma~~3]{Isakov1991AnIH}}\label{lemma: geom opt sol v}
Let $T>0$, and let $\Omega \subset \mathbb{R}^d$ $(d\geq 2)$ be an open, bounded, and connected set with smooth boundary $\partial\Omega$. Let $a=a_1 \in C_{c}^{\infty}(\Omega)$ and $(q,\Phi)\in  C_{c}^{\infty}(\Omega)\times C_{c}^{\infty}(\mathbb{R}^{d})$ such that 
\[
\bigl(\supp(\Phi)\pm T\omega\bigr)\cap\Omega=\emptyset,
\]
and $\omega \in \mathbb{S}^{d-1}$. Then, for every $x \in \mathbb{R}^d$ and $\tau>0$, Equation \eqref{FBVP; v} (given below)  admits a solution of the form 
\begin{align}\label{eq: v}
  w(t,x) = \Phi(x+t\omega)T_w(t,x)\exp\bigl(\mathrm{i}\tau(x\cdot\omega + t)\bigr) + R_w(t,x,\tau),
\end{align}
where
\begin{align*}
    T_w(t,x)=\exp\left(\frac{1}{2}\int_0^t a_1(x+s\omega)\,ds\right),
\end{align*}
and the correction term $R_w(t,x,\tau)$ vanishes at the final time, that is,
\begin{align}
   R_w(T,x,\tau) = \partial_{t}R_w(T,x,\tau) = 0 \quad \text{in } \Omega.
\end{align}
Furthermore, $R_w(t,x,\tau)$ satisfies the estimate
\begin{align}\label{eq: estimate of r_{(2)}}
  \tau\|R_w\|_{L^{2}(Q)}+\|\partial_{t} R_w\|_{L^{2}(Q)} \leq C \|\Phi\|_{H^{3}(\mathbb{R}^d)}.
\end{align}
\end{lemma}
\subsection{\texorpdfstring{Recovery of damping coefficient}{Recovery of a}} 
In this subsection, we first derive the integral identity, which is subsequently used to obtain the unique recovery of the damping coefficient. To do so, 
for $j=\{1,2\}$, let $v_{j}$  be the solution of the IBVP 
\begin{align}\label{equation; IBVP u_ell}
    \begin{cases}
      \Box v_{j}(t,x) + a_j(x)\partial_t v_{j}(t,x) + q_j(x) v_{j}(t,x) = 0,  & (t,x) \in Q,\\
      v_{j}(t,x)  = \mathtt{f}(t,x), & (t,x) \in\Sigma,\\
      v_{j}(0, x )=0,\quad \partial_t v_{j}(0, x ) = 0,  &\quad x \in\Omega.
    \end{cases}
\end{align}
From the hypothesis of Theorem \ref{thm:main_result_partial_data}, we have 
\[
\Lambda_{a_1,q_1,r_1}(\epsilon f)=\Lambda_{a_2,q_2,r_2}(\epsilon f), \qquad \text{for all } \epsilon f \in \mathcal{F}^\varrho_{m+1},
\]
which are precisely the restrictions of the DtN map given by \eqref{DN map for epsilon f} to the subset $\Gamma_T$. Using this identity, together with Equations \eqref{eq:DN-map-v_k} and \eqref{partial dn map for linearized pde}, we have
\begin{align}
  \Lambda_{a_1,q_1}^{v_1}(\mathtt{f})=  \Lambda_{a_2,q_2}^{v_2}(\mathtt{f}), \quad \mathtt{f} \in H^{m+1}(\Sigma).
\end{align}
That is, $\partial_\nu v_1|_{\Gamma_T}=\partial_\nu v_2|_{\Gamma_T}$.
Now if we denote 
\begin{align}
    \begin{aligned}
       v(t,x):=(v_{1}-v_{2})(t,x),\ 
       a(x):=(a_2-a_1)(x),\ 
        q(x):=(q_2-q_1)(x),
    \end{aligned}
\end{align}
we obtain
\begin{align}\label{rhs pde}
    \begin{cases}
      \Box v(t,x) + a_1(x)\partial_t v(t,x) + q_1(x) v(t,x) = a(x)\partial_t v_2(t,x)+q(x)v_2(t,x),  & (t,x) \in Q,\\
      v(t,x)  = 0, & (t,x) \in\Sigma,\\
      v(0, x )=0, \quad \partial_t v(0, x ) = 0,  &\quad  x \in\Omega,
    \end{cases}
\end{align}
 and \(\partial_\nu v|_{\Gamma_T}=0.\)
We then define a smooth cut-off function $\Xi \in C^{\infty}(\overline\Omega,[0,1])$ by
\begin{equation}\label{eqn: construction theta}
\Xi:= \begin{cases}
0 \qquad\text{ in } \mathcal{O}_{3},\\
1 \qquad\text{ in } \Omega_{2},
\end{cases}
\end{equation}
where $\Omega_{j} = \Omega \setminus \overline{\mathcal{O}_{j}}$.
Note that 
\begin{align}\label{i equation}
    \begin{aligned}
           \Box (\Xi v) + a_1(x)\partial_t (\Xi v) + b_1(x) (\Xi v)=\Xi I_1- v\Delta_x \Xi -2 \nabla_x v \cdot \nabla_x  \Xi,
    \end{aligned}
\end{align}
where $I_1=\Box v +a_1(x) \partial_t v +q_1(x) v$. 
Invoking \eqref{rhs pde}, we obtain  \(I = a(x)\,\partial_t v_2(t,x) + q(x)\,v_2(t,x)\) in \(Q\). Substituting this expression into the preceding relation then yields
\begin{align}
    \begin{aligned}
           \Box (\Xi v) + a_1(x)\partial_t (\Xi v) + q_1(x) (\Xi v)=\Xi\left(a(x)\partial_t v_2+q(x)v_2\right)- v\Delta_x \Xi -2 \nabla_x v \cdot \nabla_x  \Xi.
    \end{aligned}
\end{align}
Now, if we denote $\widetilde{v}:=\Xi v$, then $\widetilde{v}$ solves the following initial boundary value problem
\begin{align}\label{v tilde equation}
    \begin{cases}
      \Box \widetilde{v} + a_1(x)\partial_t \widetilde{v} + q_1(x) \widetilde{v} = \Xi\left(a(x)\partial_t v_2+q(x)v_2\right)- v\Delta_x \Xi -2 \nabla_x v \cdot \nabla_x  \Xi,  & (t,x) \in Q,\\
      \widetilde{v}(t,x)  = 0, & (t,x) \in\Sigma,\\
      \widetilde{v}(0, x )=0, \quad \partial_t \widetilde{v}(0, x ) = 0,  &\quad  x \in\Omega,
    \end{cases}
\end{align}
where $v$ is a solution to Equation \eqref{rhs pde}.
Since $(a_1,q_1)=(a_2,q_2)$ in $\mathcal{O}$, we have $(a,q)=0$ in $\mathcal{O}$. Also, we observe that $\Xi a=a$ and $\Xi q=q$ in $\overline{\Omega}$. Indeed, if $x\in \mathcal{O}_2$, we have $a=0$, as $\mathcal{O}_2\subset \mathcal{O}$. That means, $\Xi a=0=a$ in $\mathcal{O}_2$. Now, if we consider $x\in \overline{\Omega}\setminus\mathcal{O}_2$, then $\Xi=1$ which then leads to $\Xi a=a$ in $\overline{\Omega}\setminus\mathcal{O}_2$. Thus, we conclude that $\Xi a=a$ in $\overline{\Omega}$.
Similarly, we can derive $\Xi q=q$ in $\overline{\Omega}$. Moreover, $- v\Delta_x \Xi -2 \nabla_x v \cdot \nabla_x  \Xi\equiv0$ in $(0,T)\times (\mathcal{O}_3\cup \Omega_2)$. In view of the foregoing calculations, we can rewrite Equation \eqref{v tilde equation} as 
\begin{align}\label{before achieved}
    \begin{cases}
      \Box \widetilde{v} + a_1(x)\partial_t \widetilde{v} + q_1(x) \widetilde{v} = F,&   (t,x) \in Q,\\
      v(t,x)  = 0, & (t,x) \in\Sigma,\\
      v(0, x )=0, \quad \partial_t v(0, x ) = 0,  &\quad  x \in\Omega.
    \end{cases}
\end{align}
where \[F:=a(x)\partial_t v_2+q(x)v_2- v\Delta_x \Xi -2 \nabla_x v \cdot \nabla_x  \Xi\] is supported in $(0,T)\times (\O\setminus \mathcal{O})$. 
Since $\widetilde{v}=\Xi v$ vanishes in $\mathcal{O}_3$, we have 
\[
\partial_\nu \widetilde{v}|_{\Sigma}=0.
\]
Next, we use Corollary \ref{UCP} to arrive at 
\[\widetilde{v} = 0 \quad \text{in } (0,T) \times (\Omega_{3}\setminus \Omega_{2}).\]
Recall that $(a,q)$ vanish in $\mathcal{O}$ and thus in $\Omega_{3}\setminus \Omega_{2}$. Since $   \Box \widetilde{v} + a_1(x)\partial_t \widetilde{v} + q_1(x) \widetilde{v} = F$ in $Q$ and $\widetilde{v} = 0 \quad \text{in } (0,T) \times (\Omega_{3}\setminus \Omega_{2})$, we obtain 
$- v\Delta_x \Xi -2 \nabla_x v \cdot \nabla_x  \Xi=0 \text{ in } (0,T) \times (\Omega_{3}\setminus \Omega_{2})$. Also, we have seen that $- v\Delta_x \Xi -2 \nabla_x v \cdot \nabla_x  \Xi\equiv0$ in $(0,T)\times (\mathcal{O}_3\cup \Omega_2)$. This results in $- v\Delta_x \Xi -2 \nabla_x v \cdot \nabla_x  \Xi\equiv0$ in $Q$.
Finally, Equation \eqref{v tilde equation} boils down to
\begin{align}\label{achived}
    \begin{cases}
      \Box \widetilde{v} + a_1(x)\partial_t \widetilde{v} + q_1(x) \widetilde{v} = a(x)\partial_t v_2+q(x)v_2,  & (t,x) \in Q,\\
      \widetilde{v}(t,x)  = 0, & (t,x) \in\Sigma,\\
      \widetilde{v}(0, x )=0, \quad \partial_t \widetilde{v}(0, x ) = 0,  &\quad  x \in\Omega,
    \end{cases}
\end{align}
where $v_2$ is a solution to Equation \eqref{equation; IBVP u_ell}.
Next, we multiply the previous equation by $\overline{w}$ and integrate over $Q$, where $w$ is a  solution to the backward damped wave operator given by
\begin{align}\label{FBVP; v}
    \begin{cases}
      \Box w - a_1(x)\partial_t w + q_1(x) w = 0,  & (t,x) \in Q,\\
      w(T, x )=0, \quad \partial_t w(T, x ) = 0,  &\quad  x \in\Omega.
    \end{cases}
\end{align}
and use the integration by parts formula to arrive at
\begin{align}
    \begin{aligned}
       \int_{\Sigma} \widetilde{v} \partial_\nu \overline{w}\ dS_x dt -   \int_{\Sigma} \overline{w}\partial_\nu \widetilde{v}\ dS_x dt=\int_{Q}[a(x)\partial_t v_2 +q(x)v_2]\overline{w}\ dx dt.
    \end{aligned}
\end{align}
Since $\partial_\nu \widetilde{v}|_{\Sigma}=0$ and $\widetilde{v}|_{\Sigma}=0$, the above integral reduces to 
\begin{align}\label{55}
    \begin{aligned}
       \int_{Q}[a(x)\partial_t v_2 +q(x)v_2]\overline{w}\ dx dt=0.
    \end{aligned}
\end{align}
Now, we will substitute the special solutions $v_2$ and $w$ in the above expression. In particular, following Lemma \ref{lemma: geom opt sol v_{1}} and \ref{lemma: geom opt sol v}, we
choose 
\begin{align}\label{particular cgo v2 and w}
    \begin{aligned}
       v_2(t,x)&:= \Phi(x+t\omega)T_{v_2}(t,x)\exp{(\mathrm{i}\tau(x\cdot\omega+t))}
+ R_{v_2}(t,x,\tau),\\
w(t,x)&:=\Phi(x+t\omega)T_w(t,x)\exp{(\mathrm{i}\tau(x\cdot\omega+t))}
+ R_w(t,x,\tau),
    \end{aligned}
\end{align}
for all $\Phi\in \mathfrak{C}_c^{\infty}(\mathbb{R}^d)$ and $\omega\in\mathbb{S}^{d-1}$. 
Here $\mathfrak{C}_c^{\infty}(\mathbb{R}^d)$ is a subspace of  $ C_c^{\infty}(\mathbb{R}^d)$ and is given by
\begin{align}\label{eq; definition of C(Rn)}
\mathfrak{C}_c^{\infty}(\mathbb{R}^d):=\{\psi\in C_c^{\infty}(\mathbb{R}^d):\supp(\psi)\cap\Omega=\emptyset\ \mbox{and}\ (\supp(\psi)\pm T\omega)\cap \Omega=\emptyset,\ \forall \ \omega\in\mathbb{S}^{d-1}\}.
\end{align}
Thus, we have
\begin{align}\label{eqn:Reconst-a}
    \begin{aligned}
        &\int_{Q}[a(x)\partial_t v_2 +q(x)v_2]\overline{w}\ dx dt\\&=\int_{Q}\bigg(a(x) 
\big(\mathrm{i}\tau\,\Phi(x+t\omega)T_{v_2}e^{\mathrm{i}\tau(x\cdot\omega+t)}+\left(\omega\cdot\nabla_x\Phi(x+t\omega)\,T_{v_2}
+ \Phi(x+t\omega)\partial_t T_{v_2}
\right)e^{\mathrm{i}\tau(x\cdot\omega+t)}
 \\& \qquad\qquad  + \partial_t R_{v_2}
\big)+ q(x) \Phi(x+t\omega)T_{v_2}e^{\mathrm{i}\tau(x\cdot\omega+t)}
+ q(x)R_{v_2}\bigg)
\left(\Phi(x+t\omega)T_{w}e^{-\mathrm{i}\tau(x\cdot\omega+t)}
+ \overline{R_w}\right) \ dx dt
    \end{aligned}
\end{align}
Next, we simplify the integrand in Equation~\eqref{eqn:Reconst-a} to get 
\begin{align}\label{eq:expanded-identity}
\begin{aligned}
&\mathrm{i}\tau \int_{Q}
a(x)\,\Phi^{2}(x+t\omega)\,T_{v_2}T_{w}\,dxdt  + \mathrm{i}\tau \int_{Q}
a(x)\,\Phi(x+t\omega)\,T_{v_2}
e^{\mathrm{i}\tau(x\cdot\omega+t)}\overline{R_w}\,dxdt \\&\quad +\int_{Q} a(x) \Phi(x+t\omega)\left(\left(\omega\cdot\nabla_x\Phi(x+t\omega)\right)\,T_{v_2}T_w+\partial_t T_{v_2}e^{\mathrm{i}\tau(x\cdot\omega+t)}\overline{R_w}\right)\,dxdt
\\&\quad+\int_{Q} a(x)\Phi(x+t\omega) T_{w}e^{-\mathrm{i}\tau(x\cdot\omega+t)}\partial_t R_{v_2}\,dxdt+\int_{Q} a(x)\Phi^2(x+t\omega)T_w\partial_t T_{v_2} \ dx dt\\&\quad+\int_{Q} a(x) \left(\omega\cdot\nabla_x\Phi(x+t\omega)\right)\,T_{v_2}e^{\mathrm{i}\tau(x\cdot\omega+t)}\overline{R_w}\,dxdt+\int_{Q} a(x) \overline{R_w} \partial_t R_{v_2}\ dx dt
\\
& \quad+ \int_{Q}
q(x)\,\Phi^{2}(x+t\omega)\,T_{v_2}T_w\,dxdt + \int_{Q}
q(x)\,\Phi(x+t\omega)\,T_{v_2}
e^{\mathrm{i}\tau(x\cdot\omega+t)}\overline{R_w}\,dxdt \\
& \quad+ \int_{Q}
q(x)\,\Phi(x+t\omega)T_w
e^{-\mathrm{i}\tau(x\cdot\omega+t)}R_{v_2}\,dxdt + \int_{Q}
q(x)\,R_{v_2}\overline{R_w}\,dxdt = 0, 
\end{aligned}
\end{align}
for all $\Phi\in \mathfrak{C}_c^{\infty}(\mathbb{R}^d)$ and $\omega\in\mathbb{S}^{d-1}$. 
After dividing Equation \eqref{eq:expanded-identity} by $\mathrm{i}\tau$, we get
\begin{align}\label{eq:final-splitting}
\int_{Q}
a(x)\,\Phi^{2}(x+t\omega)\,T_{v_2}T_{w}\,dx\,dt
+ \sum_{i=1}^{6} J_i =0,\ \mbox{for all $\Phi\in \mathfrak{C}_c^{\infty}(\mathbb{R}^d)$ and $\omega\in\mathbb{S}^{d-1}$}, 
\end{align}
where
\begin{align*}
   J_1 &=
\int_{Q}
a(x)\,\Phi(x+t\omega)\,T_{v_2}
e^{\mathrm{i}\tau(x\cdot\omega+t)}\overline{R_w}\,dx\,dt,
\\J_2&=\frac{1}{\mathrm{i}\tau}
\int_{Q} a(x) \Phi(x+t\omega)\left(\left(\omega\cdot\nabla_x\Phi(x+t\omega)\right)\,T_{v_2}T_w+\partial_t T_{v_2}e^{\mathrm{i}\tau(x\cdot\omega+t)}\overline{R_w}\right)\,dxdt,
\\
J_3&=\frac{1}{\mathrm{i}\tau}\int_{Q} a(x)\Phi(x+t\omega) T_{w}e^{-\mathrm{i}\tau(x\cdot\omega+t)}\partial_t R_{v_2}\,dxdt+\frac{1}{\mathrm{i}\tau}\int_{Q} a(x)\Phi^2(x+t\omega)T_w\partial_t T_{v_2} \ dx dt,\\
J_4 &=\frac{1}{\mathrm{i}\tau}\int_{Q} a(x) \left(\omega\cdot\nabla_x\Phi(x+t\omega)\right)\,T_{v_2}e^{\mathrm{i}\tau(x\cdot\omega+t)}\overline{R_w}\,dxdt+\frac{1}{\mathrm{i}\tau}\int_{Q} a(x) \overline{R_w} \partial_t R_{v_2}\ dx dt\\
J_5 &=
\frac{1}{\mathrm{i}\tau}\int_{Q}
q(x)\,\Phi^{2}(x+t\omega)\,T_{v_2}T_w\,dxdt+\frac{1}{\mathrm{i}\tau}\int_{Q}
q(x)\,\Phi(x+t\omega)
e^{-i\tau(x\cdot\omega+t)}T_w R_{v_2}\,dxdt,
\\
J_6 &=
\frac{1}{\mathrm{i}\tau}\int_{Q}
q(x)\,\Phi(x+t\omega)\,T_{v_2}
e^{i\tau(x\cdot\omega+t)}\overline{R_w}\,dx\,dt+\frac{1}{\mathrm{i}\tau}\int_{Q}
q(x)\,R_{v_2}\overline{R_w}\,dxdt.
\end{align*}
We now derive explicit upper bounds for each term in the expression above. The estimate of the above terms have been proved in  \cite{bhardwaj2026reconstructionpotentialdampingcoefficients,Isakov1991AnIH}; however, we preferred to provide a proof of it for the sake of completeness.  In doing so, we use H\"older’s inequality together with the decay estimates provided by \eqref{eq: r_{(1)}} and \eqref{eq: estimate of r_{(2)}}, without referring to these estimates at every step. Furthermore, we repeatedly use the assumption that $a, a_1,q \in C_c^\infty(\Omega)$.\\

\noindent\textbf{Estimate of $J_1$.}
\begin{align*}
|J_1|
&\le
\int_{Q}
\left|a(x)\Phi(x+t\omega)T_{v_2}\overline{R_w}\right|
\,dx\,dt 
\leq
\frac{C}{\tau}\|\Phi\|^{2}_{H^{3}(\R^d)}.
\end{align*}
\noindent\textbf{Estimate of $J_2$.}
\begin{align}
    \lvert J_2\rvert&\leq \frac{1}{\tau} 
\int_{Q} \left\lvert a(x) \Phi(x+t\omega)\left(\left(\omega\cdot\nabla_x\Phi(x+t\omega)\right)\,T_{v_2}T_w+\partial_t T_{v_2}e^{\mathrm{i}\tau(x\cdot\omega+t)}\overline{R_w}\right)\right\rvert\,dxdt\\&\leq \frac{1}{\tau} \int_{Q}\left|a(x)\Phi(x+t\omega)
\left(
\omega\cdot\nabla_x\Phi(x+t\omega)T_{v_2}T_w
\right)
\right|
\,dxdt+\frac{1}{\tau} 
\int_{Q}\left|a(x)\Phi(x+t\omega)\overline{R_w}\partial_t T_{v_2} \right|\ dx dt\\&\leq \frac{C}{\tau}\left( \lVert \Phi\rVert^2_{H^1(\mathbb{R}^d)}+\frac{1}{\tau}\|\Phi\|^{2}_{H^{3}(\R^d)}\right).
\end{align}
\noindent\textbf{Estimate of $J_3$.}
\begin{align}
   \lvert J_3\rvert&\leq  \frac{1}{\tau}\int_{Q} \lvert a(x)\Phi(x+t\omega) T_{w}\partial_t R_{v_2}\rvert\,dxdt+\frac{1}{\tau}\int_{Q} \lvert a(x)\Phi^2(x+t\omega)T_w\partial_t T_{v_2}\rvert \ dx dt\\&\leq \frac{C}{\tau}\left(\|\Phi\|^{2}_{H^{3}(\R^n)}+\lVert \Phi\rVert^2_{L^2(\mathbb{R}^d)}\right).
\end{align}
\noindent\textbf{Estimate of $J_4$.}
\begin{align}
    \lvert J_4\rvert &\leq \frac{1}{\tau}\int_{Q} \lvert a(x) \left(\omega\cdot\nabla_x\Phi(x+t\omega)\right)\,T_{v_2}\overline{R_w}\rvert \,dxdt+\frac{1}{\tau}\int_{Q} \lvert a(x) \overline{R_w} \partial_t R_{v_2}\rvert\ dx dt\\&\leq \frac{C}{\tau^2}\|\Phi\|^{2}_{H^{3}(\R^d)}.
\end{align}
\noindent\textbf{Estimate of $J_5$.}
\begin{align}
   \lvert J_5 \rvert &\leq  \frac{1}{\tau}\int_{Q}
\lvert q(x)\,\Phi^{2}(x+t\omega)\,T_{v_2}T_w\rvert\,dxdt+\frac{1}{\tau}\int_{Q}
\lvert q(x)\,\Phi(x+t\omega)
T_w R_{v_2}\rvert \,dxdt\\&\leq \frac{C}{\tau}\left(
\lVert \Phi\rVert^2_{L^2(\mathbb{R}^d)}+ \frac{1}{\tau}\|\Phi\|^{2}_{H^{3}(\R^d)}
\right).
\end{align}
\noindent\textbf{Estimate of $J_6$.}
\begin{align*}
|J_6| &\leq \frac{1}{ \tau}\int_{Q} \left|q(x)\Phi(x+t\omega)T_{v_2}\overline{R_w}\right|\,dxdt+ \frac{1}{ \tau}\int_{Q} \left|q(x) R_{v_2} \overline{R_w}\right|\,dxdt\\
&\leq
    \frac{C}{\tau^2}\left(1+\frac{1}{\tau}\right)\|\Phi\|^{2}_{H^{3}(\R^d)}.
\end{align*}
Using the above estimates in \eqref{eq:final-splitting}, we achieve
\begin{align}
   \left \lvert \int_{Q}
a(x)\,\Phi^{2}(x+t\omega)\,T_{v_2}T_{w}\,dx\,dt \right \rvert \leq     \frac{C}{\tau}\left(1+\frac{1}{\tau}+\frac{1}{\tau^2}\right)\|\Phi\|^{2}_{H^{3}(\R^n)},
\end{align}
for all $\Phi\in \mathfrak{C}_c^{\infty}(\mathbb{R}^d)$ and $\omega\in\mathbb{S}^{d-1}$.
Letting
$\tau \to \infty$ in the above inequality, we have 
\begin{align}\label{eq:main-identity}
\int_{Q}
a(x)\,\Phi^{2}(x+t\omega)\,T_{v_2}T_{w}\,dx\,dt
=0,\ \mbox{for all $\Phi\in \mathfrak{C}_c^{\infty}(\mathbb{R}^d)$ and $\omega\in\mathbb{S}^{d-1}$}. 
\end{align} 
We seek to reduce the integral above to a suitable ray transform, thereby allowing us to analyze the uniqueness properties of the damping coefficient $a$. To do so, we first extend $a$ to all of $\mathbb{R}^d$ by setting it equal to zero outside $\Omega$, and continue to denote this extension by $a$. Under this convention, Equation \eqref{eq:main-identity} can be reformulated as
\begin{align*}
\int_0^T\int_{\mathbb{R}^d} a(x)\Phi^2(x+t\omega)T_{v_2}T_{w}\,dx\,dt = 0, \ \mbox{for all $\Phi\in \mathfrak{C}_c^{\infty}(\mathbb{R}^d)$ and $\omega\in\mathbb{S}^{d-1}$}.
\end{align*}
Substituting 
\begin{align}\label{sol of transport eq for v2 and w}
T_{v_2}(t,x)=\exp\left(-\frac{1}{2}\int_0^ta_2(x+s\omega)\,ds\right) \mbox{ and }  T_w(t,x)=\exp\left(\frac{1}{2}\int_0^ta_1(x+s\omega)\,ds\right)\end{align}
in the above equation, we get
\begin{align*}
\int_0^T\int_{\mathbb{R}^d} a(x)\Phi^2(x+t\omega)\exp\left(-\frac{1}{2}\int_0^ta(x+s\omega)\,ds\right)\,dx\,dt = 0,
\end{align*}
for all $\Phi\in \mathfrak{C}_c^{\infty}(\mathbb{R}^d)$ and $\omega\in\mathbb{S}^{d-1}$. By performing the change of variables \(z = x + t\omega\), followed by an interchange of the order of integration, we arrive at
\begin{align}\label{eq; distribution = known}
\int_{\mathbb{R}^d}\left(\int_0^T a(z -t\omega)\exp\left(-\frac{1}{2}\int_0^t a(z+(s-t)\o)\,ds\right)\,dt\right)\Phi^2(z)\,dz = 0,
\end{align}
for all $\Phi\in \mathfrak{C}_c^{\infty}(\mathbb{R}^d)$ and $\omega\in\mathbb{S}^{d-1}$. If we denote 
\[ \mathcal{R}_{a}(z,\omega) := \int_0^T a(z -t\omega)\exp\left(-\frac{1}{2}\int_0^t a(z+(s-t)\o)\,ds\right)\,dt\]
 for $(z,\omega)\in \mathbb{R}^d\times \mathbb{S}^{d-1}$, then Equation \eqref{eq; distribution = known} reduces to   
\begin{align}\label{T operator is known}
 \int_{\mathbb{R}^d}\mathcal{R}_{a}(z,\omega)\Phi^2(z)\,dz=0,\  \mbox{for all $\Phi\in \mathfrak{C}_c^{\infty}(\mathbb{R}^d)$ and $\omega\in\mathbb{S}^{d-1}$}.
\end{align}
Observe that
\begin{align}
\begin{aligned} 
 \mathcal{R}_{a}(z,\omega) &= \int_0^T a(z -t\omega)\exp\left(-\frac{1}{2}\int_0^t a(z+(s-t)\o)\,ds\right)\,dt\\&=\int_0^T a(z -t\omega)\exp\left(-\frac{1}{2}\int_0^t a(z-s_1\o)\,ds_1\right)\,dt
 \\&=-2\int_0^T \partial_t \exp\left(-\frac{1}{2}\int_0^t a(z-s_1\o)\,ds_1\right)\ dt\\&=-2\exp\left(-\frac{1}{2}\int_0^T a(z-s_1\o)\,ds_1\right)+2
 \end{aligned}
 \end{align}
 Thus, Equation \eqref{T operator is known} reduces to 
 \begin{align}\label{integral holds for all phi}
\int_{\mathbb{R}^d}\left(\exp\left(-\frac{1}{2}\int_0^T a(z-s_1\o)\,ds_1\right)-1 \right) \Phi^2(z)\,dz=0,\  \mbox{for all $\Phi\in \mathfrak{C}_c^{\infty}(\mathbb{R}^d)$ and $\omega\in\mathbb{S}^{d-1}$}.
 \end{align}
 Further, we simplify the above expression by constructing a suitable family of cutoff functions from $\mathfrak{C}_{c}^{\infty}(\mathbb{R}^{d})$. Let $\zeta\in C_{c}^{\infty}(B_1(0);[0,1])$ such that  $\int_{\mathbb{R}^d}\zeta^2(x)dx=1$. Now  
for $R \in \left(0, \min\left\{1, \frac{T - \operatorname{diam}(\Omega)}{3}\right\}\right)$, we define open set $\Omega_{R}\subset\mathbb{R}^d$ by 
\begin{align}\label{dom:Omega_epsilon}
    \Omega_{R}
    := \left\{ x \in \mathbb{R}^d \setminus \overline{\Omega} : 
    \operatorname{dist}(x, \Omega) < R \right\}. 
\end{align}
Let $\{\zeta_k\}_{k>0}\subset \mathfrak{C}_{c}^{\infty}(\mathbb{R}^{d})$ be a sequence of cutoff function given by
\begin{align}\label{Phi_{delta}}
    \zeta_{k}(z)=k^{-\frac{d}{2}}\zeta\left(\frac{z-z_0}{k}\right),
\end{align}
where $z_0\in \O_R$. Since Equation \eqref{integral holds for all phi} is valid for every cutoff function $\Phi \in \mathfrak{C}_{c}^{\infty}(\mathbb{R}^{d})$, in particular it also holds for $\zeta_k \in \mathfrak{C}_{c}^{\infty}(\mathbb{R}^{n})$.
Thus, Equation \eqref{integral holds for all phi} boils down to
\begin{align}
    k^{-d}\int_{B_k(z_0)}\left(\exp\left(-\frac{1}{2}\int_0^T a(z-s_1\o)\,ds_1\right)-1 \right)\zeta^2\left(\frac{z-z_0}{k}\right)\,dz=0,
\end{align}
for each $0<k<R$ and $z_0 \in \O_R$. Letting $k\rightarrow\infty$, we obtain 
\begin{align}
    \exp\left(-\frac{1}{2}\int_0^T a(z_0-s_1\o)\,ds_1\right)=1, 
\end{align}
for any $z_0 \in \O_R$ and $\o \in \mathbb{S}^{d-1}$. Apply the logarithmic function in the above expression to arrive at  
\begin{align}\label{integral of a in term of T}
    \int_0^T a(z_0-s_1\o)\,ds_1=0,\ \text{ for all } z_{0}\in \Omega_{R}\ \mbox{and}\ \omega\in \mathbb{S}^{d-1}.
\end{align}
Next, we seek to derive the $X$-Ray transform from the above integral. Since $T> \diam(\O)$, therefore we have 
\[(\O_R\pm T_1\o)\cap \O=\emptyset,\ \  \mbox{ for all $T_1\geq T$ and $\o\in \mathbb{S}^{d-1}$}.\]
From the foregoing calculation along with the hypothesis $a \in C_c^\infty(\O)$, we conclude that $a(z_0-s_1\o)=0$ for $s_1\geq T$ and therefore Equation \eqref{integral of a in term of T} can be written as
\begin{align}\label{estimate a_{omega_{eps}}}
    \int_0^\infty a(z_0-s_1\o)\,ds_1=0,\ \text{ for all } z_{0}\in \Omega_{R}\ \mbox{and}\ \omega\in \mathbb{S}^{d-1}.
\end{align}

Next following the ideas used in \cite{Rakesh01011988,Isakov1991AnIH,bhardwaj2026reconstructionpotentialdampingcoefficients},  we will show  that the above integral holds for every $z_{0}\in \mathbb{R}^d$ and $\omega\in \mathbb{S}^{d-1}$. We first observe that we can replace $z_{0}$ with $z_{0}+t\omega$ for any $t\in \mathbb{R}$ and $\omega\in \mathbb{S}^{d-1}$. Indeed, this follows from the change of variable formula and the symmetry of $\mathbb{S}^{d-1}$ i.e., $\omega\in\mathbb{S}^{d-1}$ if and only if $-\omega\in \mathbb{S}^{d-1}$.

We first consider the case $z_{0}\in \mathbb{R}^{d}\setminus\overline{\Omega}$. There are two possibilities. If the line $z_{0}+t\omega$ does not intersect $\Omega_{R}$, then it does not intersect $\overline{\Omega}$. Since $a\in C_{c}^{\infty}(\Omega)$, the integrand vanishes along this line, and hence the left hand side of the integral equation is zero, in agreement with its right hand side. Suppose now that the line passing through $z_{0}$ intersect $\Omega_{R}$. Then there exists $t\in\mathbb{R}$ such that $z_{0}+t\omega\in\Omega_{R}$. Therefore, by the preceding observation, we conclude \eqref{estimate a_{omega_{eps}}} is valid for $z_{0}\in\mathbb{R}^{d}\setminus\overline{\Omega}$. It remains to consider $z_{0}\in\overline{\Omega}$. In this case, there exists $t\in\mathbb{R}$ such that $z_{0}+t\omega\in \Omega_{R}$. Once again, the earlier observation implies that \eqref{estimate a_{omega_{eps}}} holds for such $z_{0}$. Therefore, the integral equation \eqref{estimate a_{omega_{eps}}} holds for every $z_{0}\in \mathbb{R}^{d}$. Thus, we obtain that 
 \begin{align*}
      \int_{0}^{\infty} a(z_{0}-\rho\omega)\,d\rho = 0, \ \text{ for all } z_{0}\in \mathbb{R}^d\ \mbox{and}\ \omega\in \mathbb{S}^{d-1}.
 \end{align*}
Thus, using the inversion of $X$-ray transform (see Theorem 1.1, \cite{Natterer86}), we conclude that $a=0$ in $\Omega$, which gives the uniqueness result for the damping coefficient.
\subsection{Recovery of zeroth-order coefficient}
In this subsection, we establish the unique recovery of the zeroth-order potential $q$. Substitute $a=0$ in Equation \eqref{55}, we get
\begin{align}
    \begin{aligned}
       \int_{Q}q(x)v_2(t,x)\overline{w(t,x)}\ dx dt=0.
    \end{aligned}
\end{align}
In above expression, we now substitute the special solution $v_2$ and $w$ given by \eqref{particular cgo v2 and w} to get
\begin{align}\label{eq:expanded-identity in b}
\begin{aligned}
\int_{Q}
q(x)\,\Phi^{2}(x+t\omega)\,T_{v_2}T_w\,dxdt +K_1+K_2+K_3=0, \mbox{ for all $\Phi\in \mathfrak{C}_c^{\infty}(\mathbb{R}^d)$ and $\omega\in\mathbb{S}^{d-1}$,}
\end{aligned}
\end{align}
where 
\begin{align}
    K_1&:=\int_{Q}
q(x)\,\Phi(x+t\omega)\,T_{v_2}
e^{\mathrm{i}\tau(x\cdot\omega+t)}\overline{R_w}\,dxdt,\\
K_2&:=\int_{Q}
q(x)\,\Phi(x+t\omega)T_w
e^{-\mathrm{i}\tau(x\cdot\omega+t)}R_{v_2}\,dxdt,\\
K_3&:=\int_{Q}
q(x)\,R_{v_2}\overline{R_w}\,dxdt. 
\end{align}
Again, we use H\"older’s inequality together with the decay estimates provided by \eqref{eq: r_{(1)}} and \eqref{eq: estimate of r_{(2)}}, to obtain the desired upper bounds for $K_1, K_2$ and $K_3$.\\

\noindent\textbf{Estimate of $K_1$.}
\begin{align*}
|K_1|
&\le
\int_{Q}
\left|q(x)\Phi(x+t\omega)T_{v_2}\overline{R_w}\right|
\,dx\,dt 
\leq
\frac{C}{\tau}\|\Phi\|^{2}_{H^{3}(\R^d)}.
\end{align*}
\noindent\textbf{Estimate of $K_2$.}
\begin{align*}
|K_2|
&\le
\int_{Q}
\left|q(x)\Phi(x+t\omega)T_{w}R_{v_2}\right|
\,dx\,dt 
\leq
\frac{C}{\tau}\|\Phi\|^{2}_{H^{3}(\R^d)}.
\end{align*}
\noindent\textbf{Estimate of $K_3$.}
\begin{align*}
|K_3|
&\le
\int_{Q}
\left|q(x)R_{v_2}\overline{R_w}\right|
\,dx\,dt 
\leq
\frac{C}{\tau^2}\|\Phi\|^{2}_{H^{3}(\R^d)}.
\end{align*}
Using the above estimates in \eqref{eq:expanded-identity in b}, we achieve
\begin{align}
   \left \lvert \int_{Q}
q(x)\,\Phi^{2}(x+t\omega)\,T_{v_2}T_{w}\,dx\,dt \right \rvert \leq     \frac{C}{\tau}\left(1+\frac{1}{\tau}\right)\|\Phi\|^{2}_{H^{3}(\R^d)},
\end{align}
for all $\Phi\in \mathfrak{C}_c^{\infty}(\mathbb{R}^d)$ and $\omega\in\mathbb{S}^{d-1}$.
Letting
$\tau \to \infty$ in the above inequality, we have 
\begin{align}
\int_{Q}
q(x)\,\Phi^{2}(x+t\omega)\,T_{v_2}T_{w}\,dx\,dt
=0,\ \mbox{for all $\Phi\in \mathfrak{C}_c^{\infty}(\mathbb{R}^d)$ and $\omega\in\mathbb{S}^{d-1}$}. 
\end{align} 
The substitution of $T_{v_2}$ and $T_{w}$ given by \eqref{sol of transport eq for v2 and w} in the foregoing expression leads to
\begin{align*}
\int_0^T\int_{\mathbb{R}^d} q(x)\Phi^2(x+t\omega)\exp\left(-\frac{1}{2}\int_0^ta(x+s\omega)\,ds\right)\,dx\,dt = 0,
\end{align*}
for all $\Phi\in \mathfrak{C}_c^{\infty}(\mathbb{R}^d)$ and $\omega\in\mathbb{S}^{d-1}$.
Since the damping coefficient $a=0$ and $q\in C_c^{\infty}(\Omega)$, we have 
\begin{align}
    \int_0^T \int_{\mathbb{R}^d} q(x) \Phi^2(x+t\omega)\,dx\,dt = 0, \ \mbox{for all $\Phi\in \mathfrak{C}_c^{\infty}(\mathbb{R}^d)$ and $\omega\in\mathbb{S}^{d-1}$}.
\end{align} 
The application of the change of variables formula with \(z := x + t\omega\), in conjunction with Fubini’s theorem, to the expression above yields
\begin{align}
\int_{\mathbb{R}^d}\left(\int_0^T q(z -t\omega)\,dt\right)\Phi^2(z)\,dz =0,\  \mbox{for all $\Phi\in \mathfrak{C}_c^{\infty}(\mathbb{R}^d)$ and $\omega\in\mathbb{S}^{d-1}$}.
\end{align} 
As discussed during the recovery of the damping coefficient, we now use the family of test functions $\zeta_k$ given by \eqref{Phi_{delta}}, to obtain
\begin{align}
    \int_0^T q(z_0 -t\omega)\,dt=0,\ \text{ for all } z_{0}\in \Omega_{R}\ \mbox{and}\ \omega\in \mathbb{S}^{d-1}.
\end{align}
By employing an analogous calculation to that previously used to recover the damping coefficient \(a\) from Equation \eqref{integral of a in term of T}, we obtain \[
    q(x)=0, \ \  \text{for all}\  x \in \Omega.\]
This completes the unique recovery of the zeroth-order potential $q$ in $\Omega$.\qed
\subsection{Higher-order linearization}
This subsection is devoted to deriving a higher-order linearization of Equation \eqref{equation; IBVP}. This linearization plays a key role while deriving the unique recovery of the nonlinear coefficient $r$. Given  $f_1,f_2,\dots,f_{\ell} \in H^{m+1}(\Sigma)$, we choose $\epsilon = (\epsilon_1,\dots,\epsilon_{\ell})$ such that 
\[\epsilon f:= \epsilon_1 f_1 +\dots + \epsilon_{\ell} f_{\ell} \in \mathcal{F}^\varrho_{m+1},\] 
with  $\epsilon_j\geq 0$ for each $1\leq j\leq \ell$.
Let $u(t,x):=u_{\epsilon f}(t,x)\in \mathscr{E}_{m+1}$ 
be the solution to the following equation
\begin{align}\label{eq;Linearizatiin non linear part}
\begin{cases}
\Box u(t,x) + a(x)\partial_t u(t,x) + q(x)u(t,x) =- r(x)u^{\ell}(t,x), 
&(t,x)\in  Q,\\
u(t,x)=\epsilon f(t,x), &(t,x)\in \Sigma,\\
u(0,x)=0,\quad \partial_t u(0,x)=0 ,& \quad x\in  \Omega. 
\end{cases}
\end{align}
Now after applying $\dfrac{\partial^\ell}{\partial \epsilon_{1}\cdots\partial \epsilon_{\ell}}$  to above IBVP and evaluating it at $\epsilon=0$, we obtain
 \begin{align}\label{eq;higher order Linearization eqn}
\begin{cases}
\Box U(t,x) + a(x)\partial_t U(t,x) + q(x)U(t,x) = -r(x)\ell!\prod\limits_{i=1}^{\ell}v_i(t,x),
& (t,x)\in  Q,\\
U(t,x) = 0, & (t,x)\in  \Sigma,\\
U(0,x)=0,\quad \partial_t U(0,x)=0, 
&\quad x\in  \Omega,
\end{cases}
\end{align}
where 
\begin{equation}\label{eq:def-U}
U(t,x):=\dfrac{\partial^\ell}{\partial \epsilon_{1}\cdots\partial \epsilon_{\ell}}u(t,x)\big|_{\epsilon=0},
\end{equation}
and $v_i's$ are given by Equation \eqref{eq:def-vk}.
The $\ell^{th}$-order linearization of the DtN map \eqref{eq:DN map} is given by 
\begin{align}
 \left. \dfrac{\partial^\ell}{\partial \epsilon_{1}\cdots\partial \epsilon_{\ell}}\partial_\nu u(t,x,\epsilon_1,\dots,\epsilon_{\ell})\right|_{\Sigma,\, \epsilon=0}=\partial_\nu U(t,x)|_{\Sigma}.
\end{align}
\subsection{\texorpdfstring{Recovery of nonlinearity coefficient}{Recovery of q}}\label{Recovery of r} 
This subsection is devoted to establishing the unique recovery of the nonlinearity coefficient $r$. To do so, 
for $j=\{1,2\}$, let $U_{j}$  be the solution to the initial-boundary value problem
\begin{align}\label{equation; IBVP Uj}
    \begin{cases}
      \Box U_{j}(t,x) + a(x)\partial_t U_{j}(t,x) + q(x) U_{j}(t,x) = -r_j(x)\ell!\prod\limits_{i=1}^{\ell}v_i(t,x),  & (t,x) \in Q,\\
      U_{j}(t,x)  = 0, & (t,x) \in\Sigma,\\
      U_{j}(0, x )=0,\quad \partial_t U_{j}(0, x ) = 0,  &\quad x \in\Omega.
    \end{cases}
\end{align}
Since 
\(
\Lambda_{a_1,q_1,r_1}(\epsilon f)=\Lambda_{a_2,q_2,r_2}(\epsilon f),  \text{ for all } \epsilon f \in \mathcal{F}^\varrho_{m+1},
\) 
we have
\(\partial_\nu U_1|_{\Gamma_T}=\partial_\nu U_2|_{\Gamma_T}.\)
Now if we denote 
\begin{align}
    \begin{aligned}
       U(t,x):=(U_{1}-U_{2})(t,x)\ \mbox{ and } 
      \ r(x)=r_2-r_1(x),
    \end{aligned}
\end{align}
we obtain
\begin{align}\label{U having rhs term}
    \begin{cases}
      \Box U(t,x) + a(x)\partial_t U(t,x) + q(x) U(t,x) = r(x)\ell!\prod\limits_{i=1}^{\ell}v_i(t,x),  & (t,x) \in Q,\\
      U(t,x)  = 0, & (t,x) \in\Sigma,\\
      U(0, x )=0, \quad \partial_t U(0, x ) = 0,  &\quad  x \in\Omega,
    \end{cases}
\end{align}
 and \(\partial_\nu U|_{\Gamma_T}=0.\)
Since $r(x)=0$ in $\mathcal{O}$, we have $r(x)\ell!\prod\limits_{i=1}^{\ell}v_i=0$ in $(0,T)\times \mathcal{O}$. Following Corollary \ref{UCP}, we obtain 
\[U = 0 \quad \text{in } (0,T) \times (\Omega_{3}\setminus \Omega_{2}).\]
Next, consider $\Xi \in C^{\infty}(\overline\Omega,[0,1])$  given by Equation \eqref{eqn: construction theta} and define $X:=\Xi U$ satisfying
\begin{align}\label{widetilde U having rhs term}
    \begin{cases}
      \Box X(t,x) + a(x)\partial_t X(t,x) + q(x) X(t,x) = I_2(t,x),  & (t,x) \in Q,\\
      X(t,x)  = 0, & (t,x) \in\Sigma,\\
      X(0, x )=0, \quad \partial_t X(0, x ) = 0,  &\quad  x \in\Omega,
    \end{cases}
\end{align}
where \[
I_2=\Xi r(x)\ell!\prod\limits_{i=1}^{\ell}v_i-U\Delta_x \Xi-2\nabla_xU\cdot\nabla_x \Xi.
\]
From the definition of $\Xi$ given by \eqref{eqn: construction theta}, we have $\partial_\nu X|_{\Sigma}=0$.
Next, an analogue calculation discussed to achieve \eqref{achived} after \eqref{before achieved} leads to
\begin{align}
    \begin{cases}
      \Box X(t,x) + a(x)\partial_t X(t,x) + q(x) X(t,x) = r(x)\ell!\prod\limits_{i=1}^{\ell}v_i(t,x),  & (t,x) \in Q,\\
      X(t,x)  = 0, & (t,x) \in\Sigma,\\
      X(0, x )=0, \quad \partial_t X(0, x ) = 0,  &\quad  x \in\Omega,
    \end{cases}
\end{align}
where  $v_i's$ are given by Equation \eqref{eq:def-vk}.
Now, we multiply the above equation by $\overline{w}$ and integrate over $Q$, where $w$ is a  solution to the backward damped wave operator given by
\begin{align}\label{FBVP; w}
    \begin{cases}
      \Box w - a(x)\partial_t w + q(x) w = 0,  & (t,x) \in Q,\\
      w(T, x )=0, \quad \partial_t w(T, x ) = 0,  &\quad  x \in\Omega,
    \end{cases}
\end{align}
and apply the integration by parts to arrive at
\begin{align}\label{integral identity in q}
    \begin{aligned}
\int_{Q}r(x)\ell!\overline{w}\prod\limits_{i=1}^{\ell}v_i\ dx dt=0.
    \end{aligned}
\end{align}
Now, invoking the special asymptotic solutions $v_i's$ and $w$ from \cite{bhardwaj2026reconstructionpotentialdampingcoefficients}, given by
\begin{equation}
 v_{i}(t,x)= e^{\mathrm{i}\tau\,(t+x \cdot\omega)}\left(\Phi(x+ t\omega)
\exp\left(-\frac{1}{2}\int_{0}^{t}a(x+\tau\omega)\,d\tau\right)+\sum\limits_{j=1}^{N}\frac{m_{j}(t,x)}{\tau^j}\right)+R_{v_i}(t,x,\tau), 
    \end{equation}
    and
    \begin{equation}
 w(t,x)= e^{{\ell}\mathrm{i}\tau\,(t+x \cdot\omega)}\left(\Phi(x+ t\omega)
\exp\left(\frac{1}{2}\int_{0}^{t}a(x+\tau\omega)\,d\tau\right)+\sum\limits_{j=1}^{N}\frac{\beta_{j}(t,x)}{\tau^j}\right)+R_w(t,x,\tau)
    \end{equation}
   respectively,  where $\omega\in \mathbb{S}^{d-1}$, $\Phi\in \mathfrak{C}^{\infty}_{c}(\mathbb{R}^d)$, and the correction terms  $R_{v_i}(t,x,\tau)$ for $1\leq i\leq \ell$ and $R_w(t,x,\tau)$, satisfies the following estimates: \[\lVert R_{v_i}\rVert_{L^{\infty}(Q)}\leq \frac{C}{\tau} \quad \mbox{ and}\quad \lVert R_w\rVert_{L^{\infty}(Q)}\leq \frac{C}{\tau},\] 
   where $C>0$ is independent of $\tau$. Also, we have 
    \[R_w(T,x,\tau)=\partial_{t}R_{w}(T,x,\tau)=0,\  \mbox{and}\  R_{v_i}(0,x,\tau)=\partial_t R_{v_i}(0,x,\tau)=0,\ (1\leq i\leq \ell),\   \mbox{for}\ x\in \Omega.\]
    As indicated previously, we once more substitute the particular solution into the integral identity above and attempt to reformulate it into an desired ray transform, whose inversion guarantees the unique recovery of the nonlinear coefficient \(r\). We shall begin by the following observation:
\begin{align}
    \prod_{i=1}^\ell v_i&=    \prod_{i=1}^\ell \left(e^{\mathrm{i}\tau\,(t+x \cdot\omega)}\left(\Phi(x+ t\omega)
\exp\left(-\frac{1}{2}\int_{0}^{t}a(x+\tau\omega)\,d\tau\right)+\sum\limits_{j=1}^{N}\frac{m_{j}(t,x)}{\tau^j}\right)+R_{v_i}(t,x,\tau)\right)\\&=\sum_{I\subset\{1,\dots,\ell\}}
\mathcal{A}_+^{\ell-|I|}\prod_{i\in I} R_{v_i}(t,x,\tau)
\end{align}
where 
$\mathcal{A}_+(t,x)=e^{\mathrm{i}\tau\,(t+x \cdot\omega)}\left(\Phi(x+ t\omega)
\exp\left(-\frac{1}{2}\int_{0}^{t}a(x+\tau\omega)\,d\tau\right)+\sum\limits_{j=1}^{N}\frac{m_{j}(t,x)}{\tau^j}\right)$
 and thus
 \begin{align}\label{where we use binomial}
\overline{w}\prod\limits_{i=1}^{\ell}v_i&=\left(\mathcal{A}_{\ell,-}+\overline{R_w}\right)\times\left(\sum_{I\subset\{1,\dots,\ell\}}
\mathcal{A}^{\ell-|I|}\prod_{i\in I} R_{v_i}\right)
 \end{align}
where \begin{align}\mathcal{A}_{\ell,-}(t,x)&=e^{-{\ell}\mathrm{i}\tau\,(t+x \cdot\omega)}\left(\Phi(x+ t\omega)
\exp\left(\frac{1}{2}\int_{0}^{t}a(x+\tau\omega)\,d\tau\right)+\sum\limits_{j=1}^{N}\frac{\beta_{j}(t,x)}{\tau^j}\right).\end{align}
On further simplification along with using binomial theorem, Equation \eqref{where we use binomial} reduces to
\begin{align}\label{factoe of integral identity}
\begin{aligned}
\overline{w}\!\left(\prod_{i=1}^{\ell}v_i\right)
&= \mathcal{F} \mathcal{P}^{\ell} + \mathcal{G} \mathcal{P}^{\ell}
+ (\mathcal{F}+\mathcal{G})\sum_{j=1}^{\ell} \mathcal{P}^{\ell-j}\mathcal{Q}^{j}  + \mathcal{A}_{\ell,-}\left(
\sum_{\emptyset\neq I\subset\{1,\dots,\ell\}}
\mathcal{A}_+^{\ell-|I|}\prod_{i\in I} R_{v_i}\right)\\ &\qquad  + \overline{R_w} \sum_{I\subset\{1,\dots,\ell\}}
\mathcal{A}_+^{\ell-|I|}\prod_{i\in I} R_{v_i}.
\end{aligned} 
\end{align}
where
\begin{align}\label{eq; a few notations after IE}
\begin{aligned}
   &  \mathcal{F}:= \Phi(x+ t\omega)
\exp\left(\frac{1}{2}\int_{0}^{t}a(x+\tau\omega)\,d\tau\right), \quad \mathcal{G} := \sum\limits_{i=1}^{N}\frac{\beta_{i}(t,x)}{\tau^i},\\
& \mathcal{P}:= \Phi(x+ t\omega)
\exp\left(-\frac{1}{2}\int_{0}^{t}a(x+\tau\omega)\,d\tau\right),\quad
\mathcal{Q} := \sum\limits_{i=1}^{N}\frac{m_{i}(t,x)}{\tau^i}.
\end{aligned}
\end{align}
Now we insert \eqref{factoe of integral identity} to the integral identity given by Equation \eqref{integral identity in q} and obtain
\begin{align}\label{eq:final-splitting in q}
\begin{aligned}
\int_{Q}
r(x)\,\Phi^{\ell+1}(x+t\omega)\,\left(T_{v}(t,x)\right)^{\ell -1}
\,dx\,dt
+ \sum_{j=1}^{4}L_j=\;& 0, \ \mbox{for all $\Phi\in \mathfrak{C}_c^{\infty}(\mathbb{R}^d)$},
\end{aligned}
\end{align}
 where 
\begin{align*}
    \begin{aligned}
{L}_{1}:=\int_{Q}
r(x)\mathcal{G}(t,x)
\Phi^{\ell}(x+t\omega)
(T_v(t,x))^\ell\,dx\,dt,
\end{aligned} 
\end{align*}
\begin{align*}
\begin{aligned}
L_{2}&:=\int_{Q}
r(x)\left(
\Phi(x+t\omega)
T_w(t,x)
+\mathcal{G}(t,x)
\right)
\left(\sum_{i=1}^{\ell}
\Phi^{\ell-i}(x+t\omega)
T_v^{\ell-i}(t,x)
\mathcal{Q}^{i}(t,x)
\right)\,dx\,dt, 
\end{aligned} 
\end{align*} 
\begin{align*}
\begin{aligned}
L_{3}&:=\int_{Q}
r(x)\, \mathcal{A}_{\ell,-}(t,x)
\left(
\sum_{\emptyset\neq I\subset\{1,\dots,\ell\}}\left(
\mathcal{A}_+(t,x)
\right)^{\ell-|I|}
\prod_{i\in I}R_{v_i}\right)
\,dx\,dt ,
\end{aligned} 
\end{align*} 
\begin{align*}
\begin{aligned}
L_{4}&:= \int_{Q}
r(x)\,R_w
\Bigg(
\sum_{I\subset\{1,\dots,\ell\}}
\left(
\mathcal{A}_+(t,x)
\right)^{\ell-|I|}
\prod_{i\in I}R_{v_i}
\Bigg)\,dx\,dt.
\end{aligned}
\end{align*} 
 The bounds for the above term are as follows:
\begin{align}
|L_1|
\le C\,|Q|\,\tau^{-1}
\le C{\tau}^{-1}, \quad  
|L_2|
\le C{\tau}^{-1},  
\end{align}
\begin{align}
    |L_3|
\le C\,\|R_{v_i}\|_{L^\infty(Q)}
\le C{\tau}^{-1}, \quad 
    |L_4|
\le C\,\|R_w\|_{L^\infty(Q)}
\le  C{\tau}^{-1}.
\end{align}
Letting $\tau \to \infty$ in  \eqref{eq:final-splitting in q} along with above estimates, we achieve
\begin{align*}
\int_{Q}r(x)\Phi^{\ell+1}(x+t\omega)\left(T_v(t,x)\right)^{\ell-1}\,dx\,dt\,= 0,\ \mbox{for all $\Phi\in \mathfrak{C}_c^{\infty}(\mathbb{R}^d)$ and $\o\in \mathbb{S}^{d-1}$.} 
\end{align*}
As discussed previously, we extend 
$r \in C_c^{\infty}(\Omega)$ by zero to $\mathbb{R}^d \setminus \Omega$, and, continue to denote this extension by $r$ in order to obtain 
\begin{align*}
\int_0^T\int_{\mathbb{R}^d} r(x)\Phi^{\ell+1}(x+t\omega)\left(T_v(t,x)\right)^{\ell-1}\,dx\,dt = 0,\ \mbox{for all $\Phi\in \mathfrak{C}_c^{\infty}(\mathbb{R}^d)$}.
\end{align*}
By performing the change of variables \(x + t\omega = z\), followed by an interchange in the order of integration, the above expression can be simplified to
\begin{align}
\int_{\mathbb{R}^d}\left(\int_0^T r(z -t\omega)\left(T_v(t,z -t\omega)\right)^{\ell-1}\,dt\right)\Phi^{\ell+1}(z)\,dz = 0,\ \mbox{for all $\Phi\in \mathfrak{C}_c^{\infty}(\mathbb{R}^d)$ and $\omega\in\mathbb{S}^{d-1}$}.
\end{align}
Inserting $T_v$ back in the above expression to get
\begin{align}
&\int_{\mathbb{R}^d}\left(\int_0^T r(z -t\omega)\exp\left(-\frac{(\ell-1)}{2}\int_0^ta(z-t\o+\tau \o)\,d\tau\right)\,dt\right)\Phi^{\ell+1}(z)\,dz \\&= \int_{\mathbb{R}^d}\left(\int_0^T r(z -t\omega)\exp\left(-\frac{(\ell-1)}{2}\int_0^ta(z-\sigma\o)\,d\sigma\right)\,dt\right)\Phi^{\ell+1}(z)\,dz =0,
\end{align}
for all $\Phi\in \mathfrak{C}_c^{\infty}(\mathbb{R}^d)$ and $\omega\in\mathbb{S}^{d-1}$.
Using  $\zeta\in C_{c}^{\infty}(B_1(0);[0,1])$ satisfying $\int_{\mathbb{R}^d}\zeta^{\ell +1}(x)dx=1$ and $\O_R$ given by Equation \eqref{dom:Omega_epsilon}, we define 
\begin{align}\label{Phi_{delta 1}}
    \zeta_{\delta}(z)=\delta^{-\frac{d}{\ell+1}}\zeta\left(\frac{z-z_0}{\delta}\right), \ \delta>0 \text{ and } z\in \mathbb{R}^d,
\end{align}
where $z_0\in \O_R$. 
Through this construction, we observe that the family $\{\zeta_{\delta}\}_{\delta>0}\subset \mathfrak{C}_c^{\infty}(\mathbb{R}^d)$, whenever $0<\delta<R$.
 Thus, we get   
\begin{align}\label{eq; distribution = known with function q}
\int_0^T r(z_0 -t\omega)\exp\left(-\frac{(\ell-1)}{2}\int_0^ta(z_0-\sigma\o)\,d\sigma\right)\,dt=0,  \ \mbox{for } \ z_0\in\Omega_R. 
\end{align}
Repeating a similar argument used for reconstructing the damping coefficient $a$, we conclude that
\begin{align}\label{attenuated ray transform}
  \int_0^{\infty} r(z -t\omega)\exp\left(-\frac{(\ell-1)}{2}\int_0^t a(z-\sigma\omega)\,d\sigma\right)\,dt \, =0, \  \text{ for } z\in \mathbb{R}^d. 
\end{align}
Following Section~4 of \cite{bhardwaj2026reconstructionpotentialdampingcoefficients}, we may invert the attenuated divergent-ray transform defined in Equation~\eqref{attenuated ray transform}, thereby deriving the following associated transport equation:
\begin{align}\label{eq:transport-eq}
\omega \cdot \nabla_{z}\, \Psi(z) + \frac{(\ell-1)}{2} \,a(z)\, \Psi(z) = r(z),\, \mbox{for all $z\in \mathbb{R}^d$}, 
\end{align}
 where \[\Psi(z):=  \int_0^{\infty} r(z -t\omega)\exp\left(-\frac{(\ell-1)}{2}\int_0^t a(z-\sigma\omega)\,d\sigma\right)\,dt.\]
Within the framework of our setup, we have $\Psi(z)=0$, which in turn implies that $r \equiv 0$ in $\O$. This concludes the proof of Theorem \ref{thm:main_result_partial_data}.

    \section*{Acknowledgments}
 \noindent Mandeep Kumar acknowledges the Prime Minister's Research Fellowship (PMRF), Government of India, for financial support. Parveen Kumar expresses his gratitude to the Council of Scientific and Industrial Research (CSIR), India, for the financial support provided through the research fellowship (File No. 09/1005(19269)/2024-EMR-I).
 Manmohan Vashisth acknowledges the support of the Anusandhan National Research Foundation (ANRF), Government of India, through the ARG-MATRICS Grant (File No. ANRF/ARGM/2025/002368/MTR). The authors also acknowledge partial financial support from the FIST Programme of the Department of Science and Technology (DST), Government of India (Reference No. SR/FST/MS-I/2018/22(C)).

 \bibliography{references}
 \bibliographystyle{alpha}
\end{document}